\documentclass[11pt]{article}

\usepackage[utf8]{inputenc}
\usepackage{array}
\usepackage{amsmath}
\usepackage{amsfonts}
\usepackage{amssymb}
\usepackage{amsthm}
\usepackage{extarrows}
\usepackage{hyperref}
\usepackage{mathrsfs}
\usepackage{tikz}
\usepackage{tikz-cd}

\theoremstyle{plain}
\newtheorem{thm}{Theorem}[section]
\newtheorem*{thm*}{Theorem}
\newtheorem{prop}{Proposition}[section]
\newtheorem{lem}{Lemma}[section]
\newtheorem{cor}{Corollary}[section]
\newtheorem*{problem}{Problem}

\theoremstyle{definition}

\theoremstyle{remark}
\newtheorem*{rem}{Remark}
\newtheorem*{example}{Example}

\normalsize
\numberwithin{equation}{section}

\newcommand{\Addresses}{{
  \bigskip
  \footnotesize

	(F.~Haiden) \textsc{Harvard University Department of Mathematics, Science Center, One Oxford Street, Cambridge, MA 02138, USA}\par\nopagebreak
	\textit{E-mail:} \texttt{haiden@math.harvard.edu}
	\medskip
	
	(L.~Katzarkov) \textsc{Fakult\"at f\"ur Mathematik, Universit\"at Wien, 
	Oskar-Morgenstern-Platz 1, 1090 Wien, Austria}\par\nopagebreak
	\textit{E-mail:} \texttt{lkatzarkov@gmail.com}
	\medskip
	
	(M.~Kontsevich) \textsc{Institut des Hautes \'Etudes Scientifiques,
	35 route de Chartres, 91440 Bures-sur-Yvette, France}\par\nopagebreak
	\textit{E-mail:} \texttt{maxim@ihes.fr}
}}

\begin{document}

\title{Flat surfaces and stability structures}
\author{F. Haiden, L. Katzarkov, M. Kontsevich}
\date{}

\maketitle

\begin{abstract}
We identify spaces of half-translation surfaces, equivalently complex curves with quadratic differential, with spaces of stability structures on Fukaya-type categories of punctured surfaces.
This is achieved by new methods involving the complete classification of objects in these categories, which are defined in an elementary way. 
We also introduce a number of tools to deal with surfaces of infinite area, where structures similar to those in cluster algebra appear.
\end{abstract}

\tableofcontents


\section{Introduction}

\subsection{Spaces of stability structures and flat surfaces}

The space of stability structures, $\mathrm{Stab}(\mathcal C)$, of a triangulated category $\mathcal C$ was introduced in seminal work of Bridgeland~\cite{bridgeland07} based on a proposal of M.R.~Douglas on D-branes in super-conformal field theories. 
The data of a stability structure, $\sigma$, is a distinguished class of objects in $\mathcal C$, the \textit{semistable} ones, each with a \textit{phase} $\phi\in\mathbb R$, and an additive map $Z:K_0(\mathcal C)\to\mathbb C$, the \textit{central charge}.
These have to satisfy a number of axioms, see Subsection~\ref{subsec_bridgeland}.
Although the definition seems somewhat strange at first from a mathematical point of view, it leads to the remarkable fact, proven by Bridgeland, that $\mathrm{Stab}(\mathcal C)$ naturally has the structure of a complex manifold, possibly infinite-dimensional.
Furthermore, $\mathrm{Stab}(\mathcal C)$ comes with an action of $\mathrm{Aut}(\mathcal C)$ and the universal cover of $GL^+(2,\mathbb R)$.
Thanks to the efforts of a number of people, the structure of $\mathrm{Stab}(\mathcal C)$ is understood in many particular cases, see e.g. \cite{MR3121850, MR2376815, MR2549952, MR2313537, MR2335991, bs}. 

Moduli spaces of very similar nature appear in the theory of complex curves with quadratic differential, also known as \textit{half-translation surfaces} or simply \textit{flat surfaces}, as they carry a flat Riemannian metric with conical singularities.
Each of these moduli spaces, $\mathcal M(S)$, like $\mathrm{Stab}(\mathcal C)$, comes with a wall-and-chamber structure and an action of $GL^+(2,\mathbb R)$.
An explanation for this similarity was proposed by the third named author and Soibelman in \cite{ks}: 
The spaces $\mathcal M(S)$ should be realized as moduli spaces of stability structures for a suitably defined Fukaya category of the surface $S$.
The main achievement of the present paper is to make precise and verify this claim, leading to the following result (Theorem~\ref{ModuliMapThm} in the main text).
\begin{thm*} 
Let $S$ be a marked surface of finite type, $\mathcal M(S)$ the space of marked flat structures on $S$, and $\mathcal F(S)$ the Fukaya category of $S$.
The there is a natural map
\begin{equation}
\mathcal M(S)\to \mathrm{Stab}(\mathcal F(S))
\end{equation}
which is bianalytic onto its image, which is a union of connected components.
\end{thm*}
For now it remains an open question whether the image is in fact all of $\mathrm{Stab}(\mathcal F(S))$.
The method of proof is new and relies on a detailed understanding of the category $\mathcal F(S)$.

\subsection{Fukaya categories of surfaces}

In Section~\ref{sec_category_def} we give a self-contained elementary definition of the Fukaya $A_\infty$-category $\mathcal F(S)$ of a surface, and provide arguments to show that it agrees with other definitions found in the literature.
Broadly speaking, these are either based on Floer intersection theory (see \cite{FOOO}, \cite{seidel08}, and many others), sheaves on ribbon graphs \cite{stz, dyckerhoff_kapranov, nadler}, or more direct approaches \cite{abouzaid08}.
Our own approach to define $\mathcal F(S)$ falls into the latter category and is based on explicit $A_\infty$-structures and contractibility of the classifying space of arc systems on a marked surface.

The most significant result about $\mathcal F(S)$ which we prove is a complete classification of objects (Theorem~\ref{thm_classification} in the main text).
\begin{thm*}
Isomorphism classes of objects in $\mathcal F(S)$ are classified by isotopy classes of certain immersed curves on $S$ together with local system of vector spaces on them.
In particular, these categories are of tame representation type.
\end{thm*}
Another way to say this is that all objects are ``geometric''.
This is not just a direct consequence of the definition, since $\mathcal F(S)$ is required to be closed under taking cones.
We reduce the problem to the classification of certain pairs of filtrations on vector spaces with subquotients identified.

An important tool, e.g. for proving the equivalence with other definitions and computing K-theory, are formal generators of $\mathcal F(S)$ with endomorphism algebras given by graded quivers with quadratic monomial relations.
These exist only for some $S$, but the other cases are obtained by localization of $\mathcal F(S)$.
They are graded versions of the \textit{gentle algebras} introduced by Assem--Skowro\'nski~\cite{assem_skowronski}.

\subsection{Finite vs infinite area, finite length property}

The flat surfaces of interest in ergodic theory are the ones with finite area, corresponding to quadratic differentials with arbitrary order zeros and simple poles.
For a generic flat surface with infinite area, one the other hand, the structure of the horizontal foliation is very explicitly described by certain generalized ribbon graphs (see Section~\ref{sec_cluster}), and the horizontal measure is not ergodic. 
In fact, the wall-and-chamber structure on the space of flat surfaces of infinite area is described by mutations similar to those in the theory of cluster algebras.

Another central tool in the study of flat surfaces of infinite area, apart from the horizontal foliation, is the \textit{core}, $\mathrm{Core}(S)$, which is defined to be the convex hull of the conical points.
It is a compact subset of $S$ with geodesic boundary, and in fact $S$ is completely determined by its core.
Changes to the structure of the core are related to wall-crossings of the first kind, (central charges becoming colinear over $\mathbb R$) while mutations are associated with wall-crossings of the second kind (central charges becoming real).

In terms of the associated stability structures, this dichotomy manifests itself in the finite length property of the t-structure, i.e. whether every descending chain of objects is eventually constant. 
Any stability structure on a category $\mathcal C$ determines a family of t-structures. 
If the surface has infinite area, the generic t-structure in this family is finite length with a finite number of isomorphism classes of simple objects.
If the surface has finite area on the other hand, the t-structures are not finite length.
Closely related is the question of the distribution of phases of stable objects, in particular density, which was investigated by Dimitrov and the authors in \cite{dhkk}.

\subsection{Meromorphic connections on $T^*C$ and Nevanlinna's \textit{Kernpolygon}}

For a meromorphic quadratic differential, $\varphi$, on a compact Riemann surface, $C$, there is a corresponding connection $D$ on $T^*C$ so that the scalar multiples of $\pm\sqrt{\varphi}$ are its flat sections.
This connection has a simple pole at the points where $\varphi$ has a zero or pole of any order.
Note also that the monodromy of $D$ is necessarily contained in $\{\pm 1\}$.
Conversely, given a meromorphic connection $D$ on $T^*C$ with monodromy $\{\pm 1\}$ and simple poles, we get a quadratic differential, up to scalar multiple.

If the connection $D$ has higher order poles then the associated quadratic differentials are not meromorphic but locally of the form
\begin{equation}
\exp(f(z))g(z)dz^2
\end{equation}
where $f$ and $g$ are meromorphic and $f$ has a pole of order one greater than $D$.
We call this an \textit{exponential singularity} of a quadratic differential.
In terms of flat geometry, this singularity gives rise to $n$ infinite-angle singularities of the metric, where $n$ is the order of the pole of $f$. 
Such a singularity is modeled on the metric completion of the universal cover of the punctured Euclidean plane, and can be regarded as a limit of the conical singularity when the cone angle goes to infinity.
We show the following:
\begin{thm*}
Let $S$ be a half-translation surface of finite type with conical singularities, possibly infinite-angle ones, then $S$ comes from a compact Riemann surface with quadratic differential with exponential singularities.
\end{thm*}
See Theorem~\ref{thm_flatcx} in the main text.
Thus allowing exponential singularities is quite natural from a geometric point of view, and turns out to be necessary to get stability structures on \textit{partially} wrapped Fukaya categories of surfaces.

The previous theorem is similar in nature and has some overlap with an old result of Nevanlinna~\cite{nevanlinna} on the class of functions which have Schwarzian derivative a polynomial.
Indeed, ideas of his are found in our proof.
The analysis of the topological structure of the branched covering defined by these functions leads him to a certain combinatorial structure, the \textit{Kernpolygon}, which translates into the t-structure of the stability structure for a quadratic differential $\exp(P(z))dz^2$.

\subsection{Comparison with the work of Bridgeland--Smith}

A connection between moduli spaces of quadratic differentials and stability structures is also made in pioneering work of Bridgeland--Smith~\cite{bs}. 
They do not use the Fukaya category of the surface directly, but  a certain CY3 category which is defined in terms of quivers with potential attached to ideal triangulations.
These categories are connected to cluster theory via their quivers, and were shown to be subcategories of Fukaya categories of certain non-compact 6-folds by Smith~\cite{smith13}.
Being CY3, they have generalized Donaldson--Thomas invariants in the sense of \cite{ks}.
The DT invariants and their wall-crossing had previously appeared in work of Gaiotto--Moore--Neitzke~\cite{gmn}, who also introduce the ``WKB triangulation'', a central tool in \cite{bs}.
Indeed, one of the motivations of Bridgeland--Smith is to develop a mathematical theory for \cite{gmn}.

The methods of \cite{bs} require one to restrict to meromorphic quadratic differentials without higher order zeros, and, at a generic point in the moduli space, at least one higher order pole.
Some quadratic differentials with simple poles do appear in the limit.
It is expected that their results extend in some form to holomorphic quadratic differentials with simple zeros, but the stability structures are not described by quivers with potential and would require different methods to construct.
In this paper we do indeed construct stability structures from these quadratic differentials, though not on CY3 categories.

It is somewhat puzzling that the same geometry should give rise to stability structures on, at least superficially, very different categories.
The relation between the two categories (the wrapped Fukaya category of the surface and the CY3 category) is still to be clarified.
It seems to be simpler in the case when all poles have order $\ge 3$.
One can show that there is a non-commutative divisor in the sense of Seidel~\cite{seidel14} with total space the wrapped category and central fiber the CY3 category.
We suspect that in the case of double poles an additional deformation is necessary.

More recently, an extension of the results of Bridgeland--Smith to certain polynomial quadratic differentials on $\mathbb C$ was carried out in \cite{ikeda,bqs}.
There, CY$n$ categories appear for $n\ge 2$.

\subsection{Acknowledgments}
 
The idea that quadratic differentials with exponential-type singularities can be identified with stability structures was suggested to the authors by Alexander Polishchuk.
We also thank Mohammed Abouzaid, Denis Auroux, Tom Bridgeland, Igor Burban, Tobias Dyckerhoff, Alexander Efimov, Alexander Goncharov, Dominic Joyce, Mikhail Kapranov, Gabriel Kerr, Andrew Neitzke, Tony Pantev, Yan Soibelman, Pawel Sosna, and Zachary Sylvan for discussions relating to this paper.
The authors were partially supported by the FRG grant DMS-0854977, the research grants DMS-0854977, DMS-0901330, DMS-1201321, DMS-1201475, DMS-1265230, OISE-1242272 PASI from the National Science Foundation, the FWF grant P24572-N25, and by an ERC GEMIS grant.
We are grateful to the IHES for providing perfect working conditions.
The authors were supported by two Simons Foundation grants.


\section{Moduli spaces of flat surfaces}

In this section we define the moduli spaces $\mathcal M(X)$ of marked flat structures.
Subsection~\ref{subsec_grading} discusses some topological preliminaries which are mostly relevant later when we define Fukaya categories of surfaces, but also play a small role in the definition of $\mathcal M(X)$.
Next, in Subsection~\ref{subsec_flatsurf} we define flat surfaces with conical singularities and their moduli spaces and state a theorem about their local structure.
This theorem is proven in the case of finite area in Subsection~\ref{subsec_saddle}, using triangulations by saddle connections.
Subsection~\ref{subsec_hsd} deals with the case of infinite area, where the main tool is the horizontal strip decomposition.
In Subsection~\ref{subsec_complexanalytic} we show that our class of flat surfaces may be described in terms of quadratic differentials on compact Riemann surfaces. 
Finally, in Subsection~\ref{subsec_voronoi} we review the dual Voronoi and Delaunay partitions.

\subsection{Grading of surfaces and curves}
\label{subsec_grading}

The notion of grading discussed in this subsection is essentially a special case of the one for symplectic manifolds with almost complex structure and their Lagrangian submanifolds, due to the third named author and Seidel~\cite{kontsevich95,seidel_grading}.
In the case of surfaces, no symplectic geometry is involved.

We use the notation $\Pi_1(X,x,y)$ for the set of homotopy classes of paths from $x$ to $y$ in a space $X$.
Let $X$ be a smooth oriented surface.
A \textbf{grading}, $\eta$, on $X$ is a foliation, i.e. section of the projectivized tangent bundle, $\mathbb{P}(TX)$.
The pair $(X,\eta)$ is a \textbf{graded surface}.
A morphism of graded surfaces $(X,\eta) \to (Y,\theta)$ is a pair $(f,\tilde{f})$ where $f:X\to Y$ is an orientation preserving local diffeomorphism, and $\tilde{f}$ is a homotopy from $f^*\theta$ to $\eta$, up to homotopy, i.e. $\tilde{f}\in\Pi_1(\Gamma(X,\mathbb{P}(TX)),f^*\theta,\eta)$.
Composition is given by
\begin{equation}
(f,\tilde{f})\circ(g,\tilde{g})=(f\circ g, (g^*\tilde{f})\cdot\tilde{g})
\end{equation}
where $\alpha\cdot\beta$ denotes concatenation of paths.
Every graded surface $X$ has a \textbf{shift} automorphism given by the pair $(\mathrm{id}_X,\sigma)$ where $\sigma$ restricts, for every $x\in S$, to the generator of $\pi_1(\mathbb{P}(T_xX))$ given by the orientation of $X$, i.e. the path which rotates a line counterclockwise by an angle of $\pi$. 
Anticipating the connection with triangulated categories, we write this automorphism as $[1]$ and its integer powers as $[n]$.

For an oriented surface $X$ observe that the set 
\begin{equation}
\mathcal G:=\pi_0(\Gamma(X,\mathbb P(TX)))
\end{equation}
is a torsor over $H^1(X,\mathbb Z)=[X,S^1]$.
The set $\mathcal G$ classifies gradings on $X$ up to graded diffeomorphism isotopic to the identity.
The set of gradings of $X$ up to isomorphism is thus a quotient of $\mathcal G$.
\begin{example}
Consider the cylinder $X=\mathbb{R}\times S^1$.
Any non-contractible simple closed curve in $X$ is isotopic to $S=\{0\}\times S^1$.
If we fix an orientation on $S$ then the degree $d$ of any grading $\eta$, restricted to $S$, is a well-defined integer, and $|d|\in\mathbb Z_{\ge 0}$ is a complete invariant of the graded surface.
\end{example}

Next, we consider immersed curves in graded surfaces.
A \textbf{graded curve} in a graded surface $(X,\eta)$ is a triple $(I,c,\tilde{c})$ where $I$ is a 1-manifold, possibly disconnected and with boundary, $c:I\to S$ is an immersion, and $\tilde{c}\in\Pi_1(\Gamma(I,c^*\mathbb{P}(TX)),c^*\eta,\dot{c})$.
To spell this out, $\tilde{c}$ is given by a homotopy class of paths in $\mathbb{P}(T_{c(t)}X)$ from the subspace given by the grading to the tangent space of the curve, varying continuously with $t\in I$. 
The pushforward of a graded curve as above by a graded morphism $(f,\tilde{f})$ is given by $(I,f\circ c,(c^*\tilde{f})\cdot\tilde{c})$.

A point of transverse intersection of a pair $(I_1,c_1,\tilde{c}_1)$, $(I_2,c_2,\tilde{c}_2)$ of graded curves determines an integer as follows.
Suppose $t_i\in I_i$ with
\begin{equation}
c_1(t_1)=c_2(t_2)=p\in X,\qquad \dot{c}_1(t_1)\neq \dot{c}_2(t_2)\in\mathbb{P}(T_pX).
\end{equation}
We have the following homotopy classes of paths in $\mathbb{P}(T_pX)$:
1) $\tilde{c}_1(t_1)$ from $\eta(p)$ to $\dot{c}_1(t_1)$,
2) $\tilde{c}_2(t_2)$ from $\eta(p)$ to $\dot{c}_2(t_2)$,
3) $\kappa$ from $\dot{c}_1(t_1)$ to $\dot{c}_2(t_2)$ given by counterclockwise rotation in $T_pX$ by an angle $<\pi$.
Define the \textbf{intersection index} of $c_1,c_2$ at $p$
\begin{equation}
i_p(c_1,c_2)=\tilde{c}_1(t_1)\cdot\kappa\cdot\tilde{c}_2(t_2)^{-1}\quad\in\pi_1(\mathbb{P}(T_pX))\cong\mathbb{Z}.
\end{equation}
More precisely, this notation is correct only if $c_1$,$c_2$ pass through $p$ exactly once, in particular if $c_i$ are in general position, otherwise the index may depend on the $t_i$ as well.
We note that
\begin{equation}
i_p(c_1,c_2)+i_p(c_2,c_1)=1
\end{equation}
and
\begin{equation}
i_p(c_1[m],c_2[n])=i_p(c_1,c_2)+m-n.
\end{equation}
Another observation which will be used repeatedly, is that if graded curves $c_1$ and $c_2$ intersect in $p$ with $i_p(c_1,c_2)=1$, then we may perform a kind of smoothing near $p$ which produces again a grade curve (see Figure \ref{fig_smoothing}).
\begin{figure}[h]
\centering
\includegraphics[scale=1]{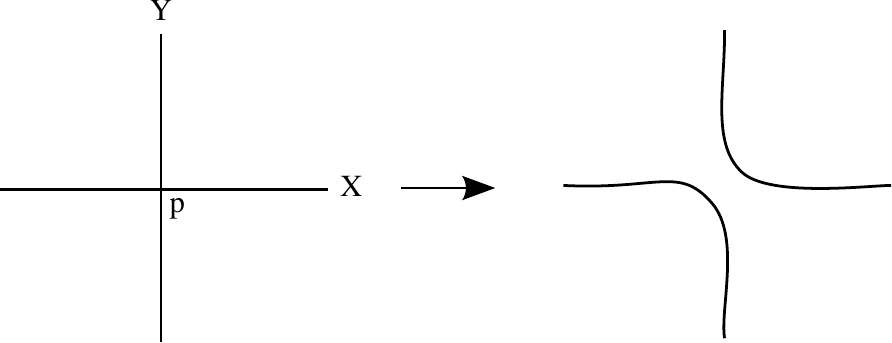}
\caption{Smoothing an intersection with $i_p(c_1,c_2)=1$.}
\label{fig_smoothing}
\end{figure}

A graded surface $(X,\eta)$ has a canonical double cover $\tau$ of $X$ with fiber over $p\in X$ the set of orientations of $\eta(p)$.
We will consider $\tau$ as a locally constant sheaf and write 
\begin{equation}
\mathbb Z_{\tau}:=\mathbb Z\otimes_{\mathbb Z/2}\tau
\end{equation}
and similarly for $\mathbb R_{\tau}$ and $\mathbb C_{\tau}$.
Singular (co)homology with coefficients in any of these local systems is defined, as is deRham cohomology with coefficients in $\mathbb R_{\tau}$ or $\mathbb C_{\tau}$.
Integration provides a pairing
\begin{equation}\label{H1pairing}
H_1(X,\partial X;\mathbb Z_{\tau})\otimes_{\mathbb Z}H^1_{\mathrm{dR}}(X,\partial X;\mathbb C_{\tau})\to\mathbb C
\end{equation}
where $H^*_{\mathrm{dR}}(M,\partial M;E)$ is the cohomology of forms with coefficients in $E$ which vanish when pulled back to $\partial M$.

\subsection{Flat surfaces with conical singularities}
\label{subsec_flatsurf}

The following types of structures on an orientable surface $X$ are in one-to-one correspondence:
\begin{enumerate}
\item Flat Riemannian metrics together with a covariantly constant foliation $\eta$ and orientation.
\item Complex structures together with a non-vanishing holomorphic quadratic differential (section of the square of the holomorphic cotangent bundle).
\item Pairs of pointwise linearly independent commuting vector fields, defined up to simultaneous change of sign, i.e. a section of the quotient of the frame bundle of $TX$ by $\mathbb{Z}/2$.
\end{enumerate}
On $\mathbb R^2$ the standard versions of these structures are: (1) the standard euclidean metric with the foliation by horizontal lines, (2) the standard complex structure $\mathbb R^2=\mathbb C$ and the quadratic differential $dz^2$, (3) the pair of vector fields $(\partial_x,\partial_y)$.
A general surface with one of the structures above is locally isomorphic to standard $\mathbb R^2$ by charts obtained from: (1) the exponential map of the metric, (2) the complex path integral of a square-root of the quadratic differential, (3) the flows of the vector fields. 
The transition functions for such charts are of the form
\begin{equation}
v \mapsto \pm v+c.
\end{equation}
Indeed, these are the symmetries of $\mathbb R^2$ with any of the standard structures above.
A \textbf{smooth flat surface} is a surface with any of the above structures.
The covariantly constant foliation makes it a graded surface and is called the \textit{horizontal foliation}.
Another more precise but longer term used for this kind of structure is \textit{half-translation surface}, emphasizing the fact that the holonomy of the metric must be a subgroup of $\{\pm 1\}$.

The metric completion of a smooth flat surface may have fairly complicated singularities in general \cite{bowman_valdez}.
Here, we will only consider singularities which are conical with cone angle a positive integer multiple of $\pi$ or infinite.
They may be described as follows.
For $n\in\mathbb Z$ let $C_n$ be the space
\begin{equation}
C_n:=(\mathbb Z/n)\times\mathbb R\times\mathbb R_{\ge 0}/\sim, \qquad (k,x,0)\sim (k+1,-x,0) \text{ for } x\le 0
\end{equation}
and
\begin{equation}
C_\infty:=\mathbb Z\times\mathbb R\times\mathbb R_{\ge 0}/\sim, \qquad (k,x,0)\sim (k+1,-x,0) \text{ for } x\le 0.
\end{equation}
Let $C$ be one of the above spaces and $0\in C$ the distinguished point $(0,0,0)\sim(k,0,0)$.
The standard flat structure on $\mathbb R^2$ gives $C\setminus 0$ the structure of a smooth flat surface and $C$ is its metric completion.
The metric is singular unless $C=C_2$.
The singularity is called a \textbf{conical singularity} with cone angle $n\pi$ if $C=C_n$ and an \textbf{infinite angle singularity} if $C=C_\infty$.

It is technically convenient to work not with the above spaces directly, but instead their real blow-up at the origin, denoted $\widehat{C}$.
This construction produces a surface with boundary in which the cone point $0\in C$ is replaced by the set of geodesics starting at that point, which is topologically $S^1$ in the case of $C_n$ and $\mathbb R$ in the case of $C_\infty$.
The complex structure on $C\setminus 0$ does not extend to $\widehat{C}$, but the horizontal foliation does. 
Also, locally, any square-root of the quadratic differential extends as a closed complex-valued 1-form to $\widehat{C}$ which vanishes on vectors tangent to the boundary.
Indeed, if we identify $\widehat{C_\infty}$ with $\mathbb R_{\ge 0}\times\mathbb R$ the 1-form is given by the expression $e^{iy}(dx+ixdy)$.

Define the structure of a \textbf{real blow-up of a flat surface} on a surface with boundary $Y$ to be the structure of a smooth flat surface on $Y\setminus\partial Y$ such that each component of $\partial Y$ has a neighborhood diffeomorphic to a neighborhood of the boundary of some $\widehat{C_n}$ or $\widehat{C_\infty}$ by a diffeomorphism preserving the flat structure on the interior, and satisfying the following completeness condition:
The space $X$ obtained by identifying points lying in the same component of $\partial Y$ is the metric completion of $Y\setminus \partial Y$. 
A metric space $X$ as above, together with the foliation on the smooth part and marked points corresponding to components of $\partial Y$ modeled on $\widehat{C_2}$, is by definition a \textbf{flat surface}.
We may recover $Y=:\widehat{X}$ from $X$ by performing the real blow-up construction above at all conical singularities and marked points.
We write $X_\mathrm{sm}$ and $X_\mathrm{sg}$ for the sets of smooth and singular points on $X$ respectively.
By definition, we also include the marked points in $X_\mathrm{sg}$.

Fixing a graded surface $X$ we may consider the \textbf{moduli space of marked flat surfaces} $\mathcal M(X)$.
An element of $\mathcal M(X)$ is represented by a real blow-up of a flat surface $Y$ with underlying oriented surface $X$ together with a morphism of graded surfaces $f:X\to Y$ with underlying map the identity on $X$.
Two such pairs $(Y,f)$, $(Z,g)$ are equivalent if there is a diffeomorphism $h:Y\to Z$ preserving the flat structure so that $h\circ f$ and $g$ are isotopic as maps of graded surfaces.
We give flat structures (before identification) the topology of pointwise convergence on the real blow-up, and $\mathcal M(X)$ the quotient topology.

If $X$ is the real blow-up of a flat surface then the square-roots of the quadratic differential on $X\setminus\partial X$ may be viewed as a closed 1-form $\alpha$ with values in $\mathbb C_{\tau}$, and as such it extends to the boundary.
In particular we get a class $[\alpha]\in H^1_{\mathrm{dR}}(X,\partial X;\mathbb C_{\tau})$ so using the pairing \eqref{H1pairing} we get a homomorphism $H_1(X,\partial X;\mathbb Z_{\tau})\to\mathbb C$ which represents the \textit{periods} of the flat surface.
We say that a graded surface $X$ is of \textbf{finite type} if $X$ and $\partial X$ have a finite number of connected components and $H_1(X,\partial X;\mathbb Z_{\tau})$ is finitely generated.

\begin{thm}
\label{THM_period_map}
Let $X$ be a graded surface of finite type.
The map
\begin{equation}
\Pi:\mathcal M(X) \to \mathrm{Hom}(H_1(X,\partial X;\mathbb Z_{\tau}),\mathbb C).
\end{equation}
assigning to each flat surface its periods, is a local homeomorphism.
\end{thm}

The strategy of the proof will depend on whether the area of the flat surface is finite or infinite and will be completed in the following two subsections.

\subsubsection*{Grading of curves on flat surfaces}

Let $X$ be a surface with flat structure defined in terms of a complex structure and non-vanishing holomorphic quadratic differential $\varphi\in\Gamma(X,(T^*X)^{\otimes 2})$.
In this context, grading of curves and intersection index can be defined somewhat more concretely.
Recall that the holomorphic and real tangent bundles may be identified, and the horizontal foliation is the unique section of $\mathbb P(TX)$ on which $\varphi$ is real and positive.
Moreover, we have an identification of $\mathbb P(TX)$ with the bundle with constant fiber $\mathbb C^*/\mathbb R_{>0}\cong U(1)$ under which this section corresponds to $1\in U(1)$.

Given an immersed curve $\alpha:I\to X$ a grading is defined by a function $\phi:I\to\mathbb R$ such that
\begin{equation}
\varphi(\dot{\alpha}(t),\dot{\alpha}(t))\in\mathbb R_{>0}e^{2\pi i\phi(t)}.
\end{equation}
The corresponding path $\alpha^*\Omega\rightsquigarrow \dot{\alpha}$ is given by
\begin{equation}
s\mapsto e^{2\pi i\phi(t)s},\qquad s\in[0,1]
\end{equation}
under the identification as above.
Note that if $\alpha$ is geodesic if and only if $\phi$ is locally constant.
Suppose we have curves $\alpha_1,\alpha_2$ with grading given by real-valued functions $\phi_1,\phi_2$ intersecting transversely in $p\in S$.
Then the intersection index is
\begin{equation}
i_p(\alpha_1,\alpha_2)=\left\lceil\phi_1(p)-\phi_2(p)\right\rceil.
\end{equation}

\subsection{Saddle connections}
\label{subsec_saddle}

Geodesics, in particular those of finite length, are fundamental in the study of flat surfaces.
The smooth part, $X_{\mathrm{sm}}$, of a flat surface $X$ is a Riemannian surface, and by our completeness requirement any geodesics on $X_{\mathrm{sm}}$ can either be extended indefinitely or converges towards a point in the complement $X_{\mathrm{sg}}$ of $X_{\mathrm{sm}}$.
A \textbf{saddle connection} is a geodesics converging towards points in $X_{\mathrm{sg}}$ (``saddles'') in both directions. 
The endpoints are not required to be distinct.

A \textbf{geodesic arc system} on a flat surface is a collection of saddle connections which intersect at most in the endpoints.

\begin{prop}
Let $X$ be a flat surface of finite type, then the size of any geodesic arc system is bounded above by a constant depending only on the topology of the real blow-up $Y=\widehat{X}$.
\end{prop}

\begin{proof}
Let $\mathcal A$ be a finite geodesic arc system on $X$ and $E=|\mathcal A|$ its cardinality.
Each saddle connections in $\mathcal A$ together with an arbitrary choice of orientation gives a non-zero class in $H_1(X,X_{\mathrm{sg}},\mathbb Z)$.
A complete set of relations between these classes is given by the polygons cut out by $\mathcal A$.
Let $F$ be the number of these polygons.
Any bi-gon cut out by $\mathcal A$ must contain a conical point with cone angle $\pi$ in its interior, so their number, $D$, is bounded by the number of compact boundary components of $Y$.
Assigning boundary edges to polygons we find that
\begin{equation}
F-D \leq \frac{2}{3} E.
\end{equation}
By assumption, $B=\mathrm{rk}H_1(X,X_{\mathrm{sg}},\mathbb Z)$ is finite, and since $E-F\leq B$ we get
\begin{equation}
E \leq 3(D+B)
\end{equation}
which completes the proof.
\end{proof}

Let $\mathcal A$ be a maximal geodesic arc system on $X$.
By the previous proposition it is finite.
Let us study its local geometry.
If $p\in X_{\mathrm{sg}}$ and $D$ is the corresponding boundary component of $\widehat X$, then the set of maximal geodesics starting at $p$ is identified with $D$.
A finite subset $E\subset D$ corresponds to (ends of) saddle connections in $\mathcal A$ which converge to $p$.
For components of $D\setminus E$ there are two possibilities: Either the maximal geodesics in these directions intersect a saddle connection in $\mathcal A$, or not. 
They cannot converge to a singularity by maximality of $\mathcal A$.
In the first case there is a triangle cut out by $\mathcal A$ with $p$ as one of its vertices.
In the second case there is an conical sector with cone angle $\ge\pi$, possibly infinite, again by maximality. 

\begin{prop}
Suppose $\mathcal A$ is a maximal geodesic arc system on a flat surfaces $X$.
Let $K$ be the union of $X_{\mathrm{sg}}$, the saddle connections in $\mathcal A$, and the triangles cut out by $\mathcal A$. 
Then $K$ is the convex hull of $X_{\mathrm{sg}}$, thus independent of $\mathcal A$.
\end{prop}

\begin{proof}
From the local geometry above we see that any geodesic which leaves $K$ stays in $X\setminus K$ for all subsequent times, so $K$ is convex. 
It contains $X_{\mathrm{sg}}$ by definition.
If $K'$ is convex and contains $X_{\mathrm{sg}}$, then it must also contain all saddle connections, in particular those in $\mathcal A$, and thus all triangles cut out by $\mathcal A$, so $K\subseteq K'$.
\end{proof}

We will refer to the convex hull of $X_{\mathrm{sg}}$ as the \textit{core} of the flat surface $X$, denoted $\mathrm{Core}(X)$.
The previous proposition shows that, for $X$ of finite type, it is covered by a finite number of singular points, saddle connections, and geodesic triangles, thus compact.

\begin{prop}
Let $X$ be a flat surface with the property that each component contains at least one singular point, $K$ the core of $X$.
Then $K$ is a deformation retract of $X$.
\end{prop}

\begin{proof}
The assumption ensures that each $p\in X$ has finite distance to $\mathrm{Core}(X)$.
A deformation retraction can be defined by moving $p\in X$ along a geodesic path to the unique point in $\mathrm{Core}(X)$ which is closest to $p$.
\end{proof}

Let $X$ be a connected flat surface with finite area and assume that $X$ is not smooth, which excludes only the case of the torus.
A maximal geodesic arc system on $X$ gives a triangulation with some triangles possibly degenerate.
By definition, a triangle is degenerate if two of its edges coincide, so that they meet in a conical singularity with cone angle $\pi$.
Using triangulations by saddle connections, we can now give a proof of Theorem~\ref{THM_period_map} in the case when the flat surface has finite area.
This result is classical, see e.g. Veech~\cite{veech93}, but we include a proof here for completeness.
The proof in the case of infinite area (but still finite type) will be given in the next section.

\begin{proof}[Proof of Theorem \ref{THM_period_map}, finite area case]
Fix a graded surface $X$ of finite type admitting flat structures with finite total area.
This means that $X$ is compact as a surface with boundary.
We also assume that there is at least one singularity, excluding the case of the torus which is easily handled directly.
Let $\Gamma=H_1(X,\partial X;\mathbb Z_{\tau})$ and $\Pi:\mathcal M(X)\to\mathrm{Hom}(\Gamma,\mathbb C)$ be the period map.
We want to show that $\Pi$ is a local homeomorphism.

Let $a_0\in\mathcal M(X)$ and equip $X$ with this flat structure.
Pick a maximal collection $\mathcal A$ of geodesic arcs on $X$.
These arcs give a triangulation of $X$, since the flat metric has finite area by assumption.
Fix a grading on each arc in $\mathcal A$, then each $\alpha\in\mathcal A$ gives a class $[\alpha]\in\Gamma$.
A complete set of relations between these classes comes from the list of triangles.
The number $\Pi(a_0)([\alpha])\in\mathbb C$ records the length and slope of an arc $\alpha\in\mathcal A$.

Since the number of triangles is finite, there is a contractible neighborhood $U$ of $\Pi(a_0)\in\mathrm{Hom}(\Gamma,\mathbb C)$ with the following property:
For any $b\in U$ and any three classes $\alpha_1,\alpha_2,\alpha_3$ in $\Gamma$ corresponding to the edges of a triangle of $\mathcal A$, the vectors $\Pi(b)([\alpha_i])$ form the edges of a non-degenerate triangle in $\mathbb C$.
This ensures that for each $b\in U$ we get an actual flat structure on $X$ essentially by direct construction.
Namely, we cut $X$ into its geodesic triangles, deform each of the triangles according to the values $\Pi(b)([\alpha])$, $\alpha\in\mathcal A$, then glue the triangles back together in the same combinatorial manner.
In this way, we have defined a continuous section $\Xi$ of $\Pi$ over $U$.

To complete the proof, it remains to be shown that $\Xi(U)$ is an open neighborhood of $a_0$.
Note that there is a neighborhood $V$ of $a_0\in\mathcal M(X)$ such that all saddle connections in $\mathcal A$ persist throughout $W$, i.e. do not break along a singularity.
Making $W$ sufficiently small, we can assume that $W\subset\Pi^{-1}(U)$.
Then the same collection of arcs $\mathcal A$, up to homotopy, again form a maximal collection of geodesic arcs for any flat surface in $W$.
This shows that $W\subset\Xi(U)$, so $\Xi(U)$ is a neighborhood of $a_0$, in fact open by the same arguments.
\end{proof}

\subsection{Horizontal strip decompositions}
\label{subsec_hsd}

Given a complex number $z\in\mathbb C$ with $\mathrm{Im}(z)>0$ the corresponding \textit{horizontal strip of finite height} is the flat surface with boundary
\begin{equation}
\left\{w\in\mathbb C\mid 0\leq\mathrm{Im}(w)\leq\mathrm{Im}(z) \right\}
\end{equation}
with marked points $\{0,z\}$ on the boundary.
The \textit{horizontal strip of infinite height} is just the closure of the upper half-plane with the origin as marked point.
It turns out that a generic surface of finite type with infinite area is obtained, in a canonical way, by gluing these types of pieces along their boundaries.

\begin{prop}
Let $X$ be flat surface of finite type with infinite area.
Assume that $X$ is connected and has at least one conical point.
Then, after possibly rotating the horizontal direction, the leaves converging towards a conical point cut the surface into horizontal strips as above.
\end{prop}

\begin{proof}
Assume that no horizontal leaf is a saddle connection.
This can always be achieved by rotating the horizontal direction, as there are at most countably many slopes of saddle connections.
In fact, the set of slopes of saddle connections is closed under the present assumptions, but this is not needed.
Note that no leaf can be a closed geodesic either, as a cylinder foliated by closed geodesics is bounded by one or more saddle connections.
Arguments of Strebel~\cite{strebel} show that there are then only two types of leaves. 
\textit{Generic leaves} which are closed and intersect the core of $X$ in a compact interval, and \textit{Critical leaves} which converge towards a conical point in one direction and eventually leave the core in the other.
Indeed, the boundary of the closure of a leaf remaining entirely in the core would be a union of leaves which are saddle connections, which is impossible.
In particular, the closure of a critical leaf adds only a single conical point.
The components of the union of all generic leaves are open parts of horizontal strips.
The exponential maps at the various conical points give identifications of the closures of these open components with the standard horizontal strips of finite or infinite height. 
\end{proof}

Let us say a bit more about the horizontal strip decompositions appearing in the previous proposition.
A conical point with finite cone angle lies on finitely many horizontal strips, while a conical point with infinite cone angle lies on infinitely many, all but finitely many of which will have infinite height.
Each horizontal strip of finite height contains a unique saddle connection.
These are exactly the saddle connections $\alpha_1,\ldots,\alpha_m$, which do not intersect any of the critical leaves.
We will refer to them as \textbf{simple saddle connections}.
Following the leaves and collapsing the horizontal strips of infinite height we construct a deformation retraction of the surface to the union of the $\alpha_i$ and the conical points.
This shows that $\alpha_i$ give a basis of $H_1(X,\partial X;\mathbb Z_\tau)$, after choosing gradings.

The graph formed by the $\alpha_i$ is identified with part of the Hausdorff version of the leaf space, $L$, of the horizontal foliation.
We may realize $L$ inside $X$ by adding the positive imaginary axis in each horizontal strip of infinite height.
Combinatorially, $L$ is a bit of a generalization of a ribbon graph with some vertices connected to infinitely many half edges (those corresponding to infinite-angle singularities).
More precisely, the set of half edges meeting a given vertex is a torsor over some cyclic group.
This will be discussed in more detail in Subsection~\ref{subsec_sgraphs}.

\begin{proof}[Proof of Theorem \ref{THM_period_map}, infinite area case]
Let $a_0\in\mathcal M(X)$ be a flat structure and assume first that there are no leaves which are saddle connections.
Be the considerations above, the critical leaves cut $X$ into horizontal strips.
Let $\alpha_1,\ldots,\alpha_m$ be the saddle connections corresponding to the horizontal strips of finite height.
Also choose the grading on each $\alpha_i$ so that $\mathrm{Im}\,\Pi(a_0)(\alpha_i)>0$.
This is possible since $\alpha_i$ is not horizontal and shifting the grading flips the sign of $\Pi(a_0)(\alpha_i)$.
We get a basis $[\alpha_1],\ldots,[\alpha_m]$ of $\Gamma=H_1(X,\partial X;\mathbb Z_\tau)$.
Consider the open subset $U\subset\mathrm{Hom}(\Gamma,\mathbb C)$ of maps sending each $[\alpha_i]$ to the upper half plane.
For each $b\in U$ we construct a flat surface $\Xi(b)$ with the same combinatorial type of horizontal strip decomposition as $a_0$, but different slopes and lengths of the $\alpha_i$ given by $b$.
The map $\Xi$ defines a section of $\mathcal M(X)\to\mathrm{Hom}(\Gamma,\mathbb C)$ over $U$.
To see that $\Xi(U)$ is open, note that any flat surface sufficiently close to one in $\Xi(U)$ has the same horizontal strip decomposition.

If $a_0$ has horizontal saddle connections, we argue as follows.
By assumption $\Gamma\neq 0$, so $\mathbb R\subset \widetilde{GL^+(2,\mathbb R)}$ acts on $\mathcal M(X)$ by rotation of the horizontal direction and each orbit contains a flat structure with only critical and generic leaves considered before.
As $\Pi$ is equivariant with respect to the action it must be a local homeomorphism near $a_0$ as well.
\end{proof}

\subsection{Complex-analytic point of view}
\label{subsec_complexanalytic}

\subsubsection*{Meromorphic quadratic differentials}

Let $C$ be a compact Riemann surface and $\varphi$ a non-zero meromorphic quadratic differential on $C$, i.e. a meromorphic section of the square of the canonical bundle.
If $D\subset C$ is the set of zeros and poles of $\varphi$, then $X_{\mathrm{sm}}=C\setminus D$ has the structure of a smooth flat surface, possibly incomplete.
The flat geometry near a zero or pole is analyzed by Strebel~\cite{strebel} by finding a local holomorphic coordinate in which the quadratic differential has a particularly simple form.
Let us recall the results.

\textit{Zeros and simple poles}. 
There is a local coordinate $z$ such that $\varphi=z^ndz^2$, $n=-1,1,2,3,\ldots$.
The metric has a conical singularity at $z=0$ with cone angle $(n+2)\pi$.
In particular, if $\varphi$ has no higher order poles, then the corresponding flat surface has finite total area.

\textit{Poles of order $n\ge 2$}.
If $n$ is odd then there is a local coordinate $z$ with $\varphi=z^{-n}dz^2$.
However, if $n$ is even then $\varphi$ admits a square root near $z=0$, and the residue $a\in\mathbb{C}$ of $\sqrt{\varphi}$ is an invariant up to sign.
It turns out to be the only local invariant, so that there is a coordinate $z$ with 
\begin{equation}
\sqrt{\varphi}=\left(z^{-n/2}+\frac{a}{z}\right)dz
\end{equation}
when $n\ge 4$ and $\sqrt{\varphi}=\frac{a}{z}dz$ when $n=2$.
In any case, we note that the flat metric near a higher order pole has infinite area and is complete, i.e. higher order poles do not lead to any additional singularities of the metric.

\begin{thm}
Let $C$ be a compact Riemann surface with non-zero meromorphic quadratic differential $\varphi$ with set of zeros and poles $D$. 
Then $|\varphi|$ gives $C\setminus D$ the structure of the smooth part of a flat surface of finite type with conical points which have cone angle an integer multiple of $\pi$.
Conversely, and flat surface of finite type without infinite-angle conical points is obtained in such a way.
\end{thm}

\begin{proof}
It remains to show the second part, so let $X$ be a flat surface of finite type without infinite-angle conical points.
First of all, we have a complex structure and non-vanishing holomorphic quadratic differential on the smooth part $X_{\mathrm{sm}}$.
We know that near a conical point the flat structure comes from a zero or first order pole of a quadratic differential, so the complex structure extends to all of $S$ and $\varphi$ extends meromorphically.
If $X$ has finite area, then it is compact and so $C=X$, $D=X_{\mathrm{sg}}$, and we are done.

For the case of infinite area we need to find a compact $K\subset X$ such that the components of the complement $X\setminus K$ are isometric to some punctured neighborhoods of higher order poles.
As a first step, take $K$ to be the core of $X$.
The boundary of $K$ in $X$ is a sequence of saddle connections meeting in exterior angles $\phi_p\ge\pi$, by convexity of $K$.
The condition on the global monodromy of the metric (existence of horizontal foliation) implies that
\begin{equation}
\sum_p(\phi_p-\pi)=n\pi
\end{equation}
where $p$ runs over the corners of a fixed component of the boundary of $K$, and $n\in\mathbb Z_{\ge 0}$ depends on that component.

Points in the complement of $K$ are parametrized by pairs $(l,d)$ of a line $l$ tangent to $K$ and a positive distance $d\in\mathbb R_{>0}$ from $K$.
In particular, each component of $X\setminus K$ is topologically a cylinder.

We enlarge $K$ so that its boundary is a sequence of closed geodesics and straight edges (finite length geodesics) meeting in exterior angles $3\pi/2$, not necessarily at conical points. 
Consider some component $B$ of the boundary of $K$.
If $B$ is a closed geodesic, then the corresponding component of $X\setminus K$ is identified with a flat cylinder which comes from a second order pole of a meromorphic quadratic differential.
Otherwise, $B$ is a union of an even number of edges with lengths $a_1,\ldots,a_{2n}$ in counterclockwise order.
If we parallel transport some edge outwards from $K$ by a distance $d\ge 0$ then the lengths of the adjacent edges increase by $d$. 
A simple proof by induction on $n$ shows that when $n$ is odd we can parallel transport the edges outwards so that they all have the same length $l$, which can be an arbitrary real number bigger than some constant depending on the original $K$.
Hence, in this case the local geometry is completely determined by $n$, and so must coincide with the one coming from a pole of order $n+2$.

When $n$ is even we can arrange the sequence of lengths to be of the form $l_1,l_2,l_3,l_3,\ldots,l_3$ after parallel transport of the edges, and add any positive real number to all elements at once.
Without loss of generality, the first edge (of length $a$) is horizontal, then the integral of the square root of the quadratic differential around the cylinder is just $a=\pm(l_1-l_2+i(l_1-l_3))$, so the geometry must be locally isometric to the one coming from the square of the abelian differential $(z^{-n-2}+a/z)dz$.
\end{proof}

\subsubsection*{Exponential-type singularities}

An \textbf{exponential-type singularity} is a transcendental singularity of a quadratic differential of the form
\begin{equation}
e^{f(z)}g(z)dz^2,\qquad f,g\text{ meromorphic}
\end{equation}
in some local coordinate $z$, where $f$ has a pole of order $n$ at $z=0$.
By change of coordinates we may assume that $f(z)=z^{-n}$.
The following proposition describes the local flat geometry near such a singularity.

\begin{prop} 
\label{prop_expflatgeom}
Let
\begin{equation}
\varphi=\exp(z^{-n})z^mg(z)dz^2
\end{equation}
with $n\in\mathbb{Z}_{>0}$, $m\in\mathbb{Z}$, $g$ holomorphic, non-vanishing on a neighborhood $U$ of $0\in\mathbb{C}$.
Set
\begin{equation}
D=\{0<|z|\leq r\}\subset U
\end{equation}
then the completion of $D$ with respect to $|\varphi|$ has $n$ additional points, all of which are $\infty$-angle singularities.
\end{prop}



\begin{lem}
Let $\rho\ge 1$, $\pi/2<\varepsilon<\pi$, $A=\{re^{i\phi}\mid r\ge\rho,|\phi|\leq\varepsilon\}\subset\mathbb{C}$, $h(z)$ a smooth function on $A$ such that for some $C_1,C_2>0$,$\lambda\in\mathbb{R}$
\begin{equation}
C_1|z|^{\lambda}\leq h(z)\leq C_2|z|^{\lambda}
\end{equation}
Then the completion of $A$ with respect to the metric
\begin{equation}
(e^{-\mathrm{Re}(z)}h(z))^2|dz|^2
\end{equation}
has a single additional point.
\end{lem}

\begin{proof}
1. Define $\eta(z)=(\max\{1,|\mathrm{Im}(z)|\})^{\lambda}e^{-\mathrm{Re}(z)}$,
then for sufficiently small $\varepsilon>0$, the boundary of $\{\eta\ge\varepsilon\}$ in $A$ is given by the curve
\begin{equation}
\mathrm{Re}(z)=\lambda\log\left(\max\{1,|\mathrm{Im}(z)|\}\right)-\log(\varepsilon).
\end{equation}
Moreover, for distinct values of $\varepsilon$, these curves have positive distance with respect to the Euclidean metric.

2. We claim that $e^{-\mathrm{Re}(z)}h(z)$ is bounded below on $\{\eta\ge\varepsilon\}$ for any $\varepsilon>0$.
By assumption,
\begin{equation}
e^{-\mathrm{Re}(z)}h(z)\ge C_1|z|^{\lambda}e^{-\mathrm{Re}(z)}.
\end{equation}
\underline{Case $\lambda\ge 0$}: For $|\mathrm{Im}(z)|\ge 1$ we have
\begin{align}
C_1|z|^{\lambda}e^{-\mathrm{Re}(z)} &\ge C_1|\mathrm{Im}(z)|^{\lambda}e^{-\mathrm{Re}(z)} \\
&= C_1\eta(z) \\
&\ge C_1\varepsilon
\end{align}
\underline{Case $\lambda<0$}:
\begin{equation}
C_1|z|^{\lambda}e^{-\mathrm{Re}(z)}\ge C_1\left(|\mathrm{Re}(z)|e^{-\mathrm{Re}(z)/\lambda}+|\mathrm{Im}(z)|e^{-\mathrm{Re}(z)/\lambda}\right)^{\lambda}
\end{equation}
The first term is bounded above, as $\mathrm{Re}(z)$ is bounded above by $-\log(\varepsilon)$, the second term is bounded above by $\varepsilon^{1/\lambda}$.

3. We claim that for every $\varepsilon>0$, the set $\{\eta\ge\varepsilon\}$ is complete with respect to the metric $g=\left(e^{-\mathrm{Re}(z)}h(z)\right)^2|dz|^2$, thus any Cauchy sequence $z_j\in A$ without limit must satisfy $\eta(z_j)\to 0$.
Namely, we have $g\ge Cg_{\mathrm{eucl}}$ on $\{\eta\ge\varepsilon/2\}$ be the previous step, and the boundary curves $\{\eta=\varepsilon\}$, $\{\eta=\varepsilon/2\}$ have some positive distance $\delta$, so
\begin{equation}
d(z_1,z_2)\ge C\min\left\lbrace|z_1-z_2|,2\delta\right\rbrace
\end{equation}
for $z_1,z_2\in\{\eta\ge\varepsilon\}$.
Hence, any sequence which is Cauchy with respect to $g$ is Cauchy for the standard metric.

4. We claim that $\mathrm{diam}\{\eta\leq\varepsilon\}\to 0$ as $\varepsilon\to 0$.
Assume $1>\varepsilon>0$ is sufficiently small so that $\{\eta\leq\varepsilon\}$ contains no $z$ with $\mathrm{Re}(z)\leq 0$, $|\mathrm{Im}(z)|\leq\rho$, then for $z\in\{\eta\leq\varepsilon\}$ the path $\alpha(t)=z+t$, $t\in[0,\infty)$ is contained in $\{\eta\leq\varepsilon\}$.
We compute the length of $\alpha$,
\begin{align}
l(\alpha) &= \int_0^{\infty}e^{-\mathrm{Re}(z)-t}h(z+t)dt \\
&\leq C_2\int_0^{\infty}e^{-\mathrm{Re}(z)-t}|z+t|^{\lambda}dt =:C_2I
\end{align}

\underline{Case $\lambda>0$}:
\begin{equation}
I\leq \int_0^{\infty}\left(|\mathrm{Re}(z)+t|e^{-(\mathrm{Re}(z)+t)/\lambda}+|\mathrm{Im}(z)|e^{-(\mathrm{Re}(z)+t)/\lambda}\right)^{\lambda}dt
\end{equation}
As $\eta(z)\leq\varepsilon$, $\mathrm{Re}(z)\ge-\log(\varepsilon)>0$, the first summand is bounded above by $\left(-\log(\varepsilon)+t\right)\varepsilon^{1/\lambda}e^{-t/\lambda}$, and the second by $\varepsilon^{1/\lambda}e^{-t/\lambda}$, so
\begin{align}
I &\leq\varepsilon\int_0^{\infty}(1-\log(\varepsilon)+t)^{\lambda}e^{-t}dt \\
  &\leq C\varepsilon e^{(1-\log(\varepsilon))/2} \\
  &=C\sqrt{e\varepsilon}
\end{align}
where $C=C(\lambda)$ is a constant.

\underline{Case $\lambda\leq 0$}:
\begin{align}
I &\leq\int_0^{\infty}e^{-\mathrm{Re}(z)-t}\left(\max\{1,|\mathrm{Im}(z)|\}\right)^{\lambda}dt \\
&=\eta(z)\leq\varepsilon
\end{align}
To prove the claim, it remains to show that for different values of $z_1,z_2$, the corresponding horizontal curves starting at $z_1,z_2$ have vanishing distance.
Let $\beta(t)=x+it$, $t\in[y_1,y_2]$, $x>\rho$, then
\begin{align}
l(\beta)&=\int_{y_1}^{y_2}e^{-x}h(x+it)dt \\
&\leq C_2e^{-x}\int_{y_1}^{y_2}|x+it|^{\lambda}dt\xrightarrow{x\to\infty}0
\end{align}
Which completes the proof of the lemma.
\end{proof}

\begin{proof}[Proof of Proposition~\ref{prop_expflatgeom}]
After a change of coordinates $w=e^{\pi i/n}/z$, $\varphi$ is of the form
\begin{equation}
\varphi=e^{-w^n}h(w)dw^2
\end{equation}
with $h$ meromorphic in a neighborhood of $\infty$.
Let $B=D^{-1}$ and fix $\varepsilon\in(\frac{1}{2},1)$. 
Define sectors
\begin{equation}
V_n=\left\{w\in B\mid \left|\arg(w)-\frac{2\pi k}{n}\right|<\frac{\pi\varepsilon}{n}\right\}
\end{equation}
and their complement
\begin{equation}
A=B\setminus\bigcup_{k=0}^{n-1}V_n.
\end{equation}
Then for $z=re^{i\phi}\in A$ we find $\cos(n\phi)\leq\cos(\pi\varepsilon)<0$, hence
\begin{equation}
|\varphi|=e^{-r^n\cos(n\phi)}|hdz^2|\ge C|dz|^2
\end{equation}
for some $C>0$.
Since the sets $V_k$ have positive distance with respect to the standard metric, this shows that every Cauchy sequence in $B$ without limit is eventually contained in one of the $V_k$.

On $V_k$ we perform a change of coordinates $u=w^n$, so
\begin{equation}
\varphi=e^{-2u}f(u)du^2
\end{equation}
and there are $C_1,C_2>0$ with
\begin{equation}
C_1|u|^{\lambda}\leq\sqrt{|f(u)|}\leq C_2|u|^{\lambda},\qquad \lambda=\frac{m-n+1}{2n}
\end{equation}
where $m=\mathrm{ord}_{\infty}h$. 
Applying the previous lemma, the first part of the theorem follows.

By the first part of the proof,
\begin{equation}
\bar{B}=:B\cup\{a_1,\ldots,a_n\}
\end{equation}
is a complete metric space.
Let $a=a_k$, the we must show that $a$ is an $\infty$-angle singularity of $\bar{B}$.

As $V_k$ is simply connected, we may choose a holomorphic $f:V_k\to\mathbb{C}$ with $(df)^2=\varphi$.
Choose $r>0$ such that $D_{2r}^*(0)\subset V_k$.
Then for any path $\alpha$ in $D_{2r}^*$ with endpoint $z_1,z_2\in D_r^*(0)$ we compute
\begin{equation}
\left|f(z_2)-f(z_1)\right|=\left|\int_{z_1}^{z_2}\sqrt{\varphi}\right|\leq\int_{\alpha}\sqrt{|\varphi|}
\end{equation}
hence
\begin{equation} \label{f_var_bound}
|f(z_2)-f(z_1)|\leq d(z_1,z_2)
\end{equation}
so that $f$ extends to a continuous map on $V_k\cup\{a\}$.

Without loss of generality, $f(a)=0$.
We claim that $f$ restricts to a covering $D_r^*(a)\to D_r^*(0)$.
By \eqref{f_var_bound}, the image of $D_r^*(a)$ under $f$ is contained in $D_r(0)$.
The radius of injectivity at $z\in D_r^*(a)$ is $r_1=d(z,a)$.
It cannot be larger, as the singular point $a$ would have to be contained in any disk of radius $>r_1$ centered at $z$, and it cannot be smaller by completeness of $\bar{B}$, and since any geodesic may be extended until it hits $a$ or leaves $D_{2r}(a)$.
We conclude that $D_{r_1}(z)$ is mapped isometrically to $D_{r_1}(f(z))$ by $f$, and thus $r_1=|f(z)|$.
This shows that $f$ restricts to a covering $D_r^*(a)\to D_r^*(0)$.

From the proof of the lemma, it follows that $a$ admits a fundamental system of neighborhoods $W$ with the property that $W\setminus\{a\}$ is simply-connected. 
Thus $D_r^*(a)$ must be simply connected, and $f:D_r^*(a)\to D_r^*(0)$ is a universal covering.
\end{proof}

Next we prove a kind of converse to the previous proposition, which is the main result of this subsection.

\begin{thm} \label{thm_flatcx}
Let $X$ be a flat surface of finite type.
Then there is a compact Riemann surface $C$ with quadratic differential $\varphi$ of exponential type giving $C\setminus D(\varphi)$ the structure of the smooth part of a flat surface isomorphic to $X_{\mathrm{sm}}$.
\end{thm}

\begin{proof}
Let $K:=\mathrm{Core}(X)$, which we may assume to contain more than one point.
Define a \textit{boundary walk} to be a piecewise geodesic path which follows the boundary of $K$ so that $S\setminus K$ lies to the right (and possibly also to the left), takes the rightmost possible direction at every singularity, and is maximal with these properties.
By rightmost direction we mean the following.
From each singularity there is a finite set $E$ of directions which lie in $\partial K$, and they are either cyclically or totally ordered.
So the direction from which we approach the singularity has some successor, by definition the rightmost direction, or is a maximal element, and the boundary walk ends.
Hence a boundary walk is either closed or starts and ends at infinite-angle singularities, and
there is a unique boundary walk starting/ending at each infinite-angle singularity.

Let $s_i\in X_{\mathrm{sg}}$, $i\in\mathbb Z/k$ be a cyclic sequence of infinite-angle singularities so that there is a boundary path from $s_i$ to $s_{i+1}$.
Choose a closed embedded curve $\alpha$ on $X_{\mathrm{sm}}$ which is a smoothing of the cyclic sequence of boundary paths and cuts $X$ into two pieces $X',X''$ so that $X'$ contains the $s_i$ and no other singularities.
The smooth part of $X'$ is topologically an annulus.
To complete the proof, it suffices to show that $X'$ corresponds to some punctured neighborhood of an exponential singularity.

Fix some direction $r_i$ from the singularity $s_i$.
For each integer $n>0$ consider the pair of geodesics $L_i,R_i$ which start at $s_i$ in the directions $r_i+n\pi/2$ and $r_i+n\pi/2$ respectively.
For $n$ sufficiently large, all geodesics $L_i,R_i$ can be extended indefinitely in $X'$.
We construct a flat surface $X_n$ as follows.
Cut $X'$ along $L_i,R_i$ and remove the (contractible) pieces left of each $L_i$ and right of each $R_i$.
Then glue each $L_i$ to $R_i$ to obtain $X_n$.
Note that $X_n$ has $k$ conical singularities, all with cone angle $n\pi$.
Let $U_n\subset X_n$ be the complement of the geodesic rays along which $X_n$ was glued, including the singularities.
If we also denote the corresponding subset of $X'$ by $U_n$, then 
\begin{equation}
U_n\subset U_{n+1}, \qquad \bigcup_n U_n=X'_\mathrm{sm}.
\end{equation}

By the previous results about flat surfaces without infinite-angle singularities,
there is a biholomorphic $f_n:X_n\to\mathbb D^*$ and meromorphic quadratic differential $\varphi_n$ on $\mathbb D$ with a pole at $0$ and $k$ zeros, such that $f_n$ becomes an isometry of flat surfaces.
Since $\mathbb D\subset \mathbb C$ is bounded, the $f_n$ form a normal family and, after passing to a subsequence, converge to a holomorphic $f:X'_\mathrm{sm}\to\mathbb C$.
For degree reasons, $f$ cannot be constant and consequently has image contained in $\mathbb D^*$.
By a standard application of Hurwitz's theorem, $f$ is injective as the limit of injective functions.
We wish to show that the image of $f$ is all of $\mathbb D^*$, and that the $\varphi_n$ converge to a quadratic differential with exponential singularity at the origin.

Let $g_n(z)dz^2=\varphi_n(z)$ and consider $h_n=(\partial g_n/\partial z)/g_n$, then $h$ is meromorphic on the unit disk with $k+1$ simple poles, one of which is at the origin.
Convergence of $f_n$ implies convergence of $h_n$, possibly after passing to a subsequence, to a meromorphic function $h$ with $k+1$ poles, counted with multiplicity.
As $S_1$ has $k$ infinite-angle singularities we must have a single pole of order $k+1$ at the origin, so that $\varphi_n$ converge to $\varphi$ of the form $e^{p(z)}q(z)dz$ with $p,q$ meromorphic and $p$ having a pole of order $k$ at the origin.
Completeness of $X$ requires the image of $f$ to be $\mathbb D^*$.
\end{proof}

\subsection{Voronoi and Delaunay partitions}
\label{subsec_voronoi}

In this subsection we review the Voronoi and Delaunay partitions of a flat surface, c.f. Masur--Smillie~\cite{ms}.
They will be used later to construct convergent sequences of flat surfaces.

Let $X$ be a flat surface, not necessarily of finite type, but so that every component contains a conical point.
For $p\in X$ let $\rho_p=d(p,X_{\mathrm{sg}})\ge 0$ be the distance to the closest conical point.
Further, define $\mu_p\ge 1$ to be the number of directions from $p$ in which the geodesic of length $\rho_p$ starting at $p$ ends at a conical point if $\rho_p>0$, and $\mu_p=1$ if $\rho_p=0$.
In other words, $\rho_p$ is the radius of the largest isometrically immersed disk centered at $p$, and $\mu_p$ is the number of points on its boundary mapped to $X_{\mathrm{sg}}$.
The Voronoi cells are the connected components of the level sets of $\mu_p$.
Voronoi $2$-cells, $1$-cells, and $0$-cells are components with $\mu_p=1$, $\mu_p=2$, and $\mu_p\ge 3$ respectively.
The $2$-cells are contractible and in bijection with $X_\mathrm{sg}$.
The $1$-cells are geodesically embedded open intervals, possibly of infinite length.
The union of the $0$-cells is a discrete subset of $X$.

The Delaunay polygonal subdivision is in a sense dual to the Voronoi partition.
Given $p\in X$ let $D$ the largest isometrically immersed disk centered at $p$, and let $C(p)\subset X$ be the image of the convex hull (in $D$) of the points on the boundary of $D$ which map to $X_\mathrm{sg}$.
Note that when $p$ is contained in a Voronoi 2-cell, 1-cell, or 0-cell, then $C(p)$ is a point, edge, or polygon respectively, and $C(p)$ only depends on the component of the Voronoi partition.
The sets $C(p)$ provide a polygonal subdivision of $\mathrm{Core}(X)$.
For a generic flat surface all polygons are triangles.
Although the Voronoi and Delaunay cell decompositions are combinatorially dual, a Delaunay edge need not always intersect its Voronoi dual.

For a flat surface $X$ of finite type consider the quantity
\begin{equation}
\rho(X):=\sup\{\rho(p)\mid p\in\mathrm{Core}(X)\}=d_H(X_\mathrm{sg},\mathrm{Core}(X))
\end{equation}
where $d_H$ denotes the Hausdorff distance of closed subsets.
We use $\rho(X)$ to get an upper bound on the lengths of Delaunay edges.

\begin{lem} \label{lem_DelaunayBound}
Let $X$ be a flat surface of finite type, and $C\ge 0$ the maximal length of any saddle connection on the boundary of $\mathrm{Core}(X)$.
Then the length of any Delaunay edge is bounded above by $\max(C,2\rho(X))$.
\end{lem}

\begin{proof}
Let $e$ be a Delaunay edge and $p\in X$ with $C(p)=e$. 
The geodesic $e$ is a chord of the largest immersed disk, $D$, centered at $p$.
If $p\in\mathrm{Core}(X)$, then $\mathrm{length}(e)\le 2\rho_p\le 2\rho(X)$.
On the other hand, if $p\notin\mathrm{Core}(X)$, then the preimage of the boundary of $\mathrm{Core}(X)$ is a either one or two chords in $D$, one of which has length greater than or equal to $\mathrm{length}(e)$, and less than or equal to $C$.
\end{proof}


\section{The topological Fukaya category of a surface}
\label{sec_category_def}

In this section we give an elementary definition of the (topological) Fukaya categories of surfaces, and prove several results about them.
First, in Subsection~\ref{subsec_ainfty}, we collect some definitions concerning $A_\infty$-categories.
In the following subsection we introduce a class of marked surfaces $S$ which is convenient for the definition of $\mathcal F(S)$, which is given in Subsection~\ref{subsec_arc_category} in terms of explicit $A_\infty$-structures.
In Subsection~\ref{subsec_quivers} we describe $\mathcal F(S)$ in terms of graded quivers with quadratic monomial relations, a central tool in what follows.
This is possible only for some $S$, but the remaining categories are obtained by localization, as is discussed in Subsection~\ref{subsec_localization}.
In the final subsection, we show that $\mathcal F(S)$ is the category of global sections of a certain cosheaf of categories on a ribbon graph in $S$.

\subsection{Preliminaries on $A_\infty$-categories}
\label{subsec_ainfty}

We will need some basic definitions of the $A_\infty$ language, for more details see e.g. \cite{seidel08,ks09}.

An $A_\infty$\textbf{-category}, $\mathcal A$, consists of a set $\mathrm{Ob}(\mathcal A)$ of objects, for each pair $X,Y\in\mathrm{Ob}(\mathcal A)$ a $\mathbb Z$-graded vector space $\mathrm{Hom}(X,Y)$, and structure maps
\[
\mu^n:\mathrm{Hom}(X_{n-1},X_n)\otimes\cdots\otimes\mathrm{Hom}(X_0,X_1)\to\mathrm{Hom}(X_0,X_n)[2-n]
\]
for each $n\ge 1$, satisfying the $A_\infty$-relations
\begin{equation} \label{A_infty_rels}
\sum_{i+j+k=n}(-1)^{\Vert a_k\Vert+\ldots+\Vert a_1\Vert}\mu^{i+1+k}(a_n,\ldots,a_{n-i+1},\mu^j(a_{n-i},\ldots,a_{k+1}),a_k,\ldots,a_1)=0
\end{equation}
where $\Vert a\Vert=|a|-1$ is the reduced degree.

An $A_\infty$-category $\mathcal A$ is \textbf{strictly unital} if for every object $X$ there is a $1_X\in\mathrm{Hom}^0(X,X)$ such that
\begin{align}
&\mu^1(1_X)=0 \\
&\mu^2(a,1_X)=(-1)^{|a|}\mu^2(1_Y,a)=a,\quad a\in\mathrm{Hom}(X,Y) \\
&\mu^k(\ldots,1_X,\ldots)=0\text{ for }k\ge 3
\end{align}
Strictly unital $A_\infty$-categories with $\mu^k=0$ for $k\ge 3$ correspond to small dg-categories with
\begin{equation}
da=(-1)^{|a|}\mu^1(a),\qquad ab=(-1)^{|b|}\mu^2(a,b).
\end{equation}
Indeed, in this case the first three $A_\infty$-relations correspond to $d^2=0$, the Leibniz rule, and associativity of the product.

An $A_\infty$-functor $F:\mathcal A\to\mathcal B$ is given by a map $\mathrm{Ob}\mathcal A\to\mathrm{Ob}\mathcal B$ and multilinear maps
\begin{equation}
F^d:\mathrm{Hom}_{\mathcal A}(X_{d-1},X_d)\otimes \cdots\otimes\mathrm{Hom}_{\mathcal A}(X_0,X_1)\to\mathrm{Hom}_{\mathcal B}(FX_0,FX_d)[1-d]
\end{equation}
for every $d\ge 1$, $X_0,\ldots,X_d\in\mathrm{Ob}\mathcal A$, satisfying the relations
\begin{gather}
\sum_{r\ge 1}\sum_{s_1+\ldots+s_r=d}\mu^r_\mathcal B(F^{s_r}(a_d,\ldots,a_{d-s_r+1}),\ldots,F^{s_1}(a_{s_1},\ldots,a_1)) \\
=\sum_{m,n}(-1)^{\Vert a_n\Vert+\ldots+\Vert a_1\Vert}F^{d-m+1}(a_d,\ldots,a_{n+m+1},\mu^m_\mathcal A(a_{n+m},\ldots,a_{n+1}),a_n,\ldots,a_1).
\end{gather}
We will mostly consider strict $A_\infty$-functors, i.e. those with $F^d=0$ for $d>1$.

\subsubsection*{Twisted complexes}
\label{subsec_twisted_complexes}

We briefly recall the construction of the category of twisted complexes, $\mathrm{Tw}\mathcal A$, over an $A_\infty$-category $\mathcal A$.
Twisted complexes can be thought of as formal vector bundles with flat connection.

The first step is to form $\mathrm{add}\mathbb{Z}\mathcal A$ whose objects are formal sums of tensor products of the form
\begin{equation}
V=\bigoplus_{X\in\mathrm{Ob}\mathcal A}V_X\otimes X
\end{equation}
with $V_X$ finite-dimensional graded vector spaces, zero for all but finitely many $X$.
Morphism spaces are defined by
\begin{equation}
\mathrm{Hom}(V\otimes X,W\otimes Y)=\mathrm{Hom}(V,W)\otimes\mathrm{Hom}(X,Y)
\end{equation}
and additivity.
When extending the $\mu^k$ a Koszul sign gets introduced:
\begin{equation}
\mu^k(\phi_k\otimes a_k,\ldots,\phi_1\otimes a_1)=(-1)^{\sum_{i<j}|\phi_i|\|a_j\|}\phi_k\cdots\phi_1\otimes \mu^k(a_k,\ldots,a_1) 
\end{equation}
Note that $\mathrm{add}\mathbb{Z}\mathcal A$ has a natural shift functor and formal finite direct sums.
Strictly speaking, $\mathrm{Ob}(\mathrm{add}\mathbb{Z}\mathcal A)$ is not a set, but since the category of finite-dimensional vector spaces is essentially small, this is a non-issue.

An object in $\mathrm{Tw}\mathcal A$ is given by a pair $(V,\delta)$ with $V\in\mathrm{add}\mathbb{Z}\mathcal A$ and $\delta\in\mathrm{Hom}^1(V,V)$. 
The first condition is that there is a direct sum decomposition of $V$ (i.e. of $\bigoplus V_X$ as an $\mathrm{Ob}\mathcal A\times\mathbb Z$-graded vector space) so that $\delta$ is strictly upper triangular.
This ensures that $\mu^k(\delta,\ldots,\delta)=0$ for $k$ big, so that the second condition, which is
\begin{equation}
\sum_{k\ge 1}\mu^k(\delta,\ldots,\delta)=0
\end{equation} 
makes sense.
Morphism spaces are just
\begin{equation}
\mathrm{Hom}((V,\delta),(W,\epsilon))=\mathrm{Hom}(V,W)
\end{equation}
and the structure maps are 
\begin{equation}
\mu^k(a_k,\ldots,a_1)=\sum_{n_0,\ldots,n_k\ge 0}\mu^{k+n_0+\ldots+n_k}(\underbrace{\delta_k,\ldots,\delta_k}_{n_k\text{ times}},a_k,\ldots,a_1,\underbrace{\delta_0,\ldots,\delta_0}_{n_0\text{ times}})
\end{equation}
where $a_i\in\mathrm{Hom}((V_{i-1},\delta_{i-1}),(V_i,\delta_i))$.
Again, this sum is actually finite by our requirement on the $\delta_i$.

The main fact we need about $\mathrm{Tw}\mathcal A$ is that its homotopy category, $H^0(\mathrm{Tw}\mathcal A)$, is triangulated. 
When we write $K_0(\mathrm{Tw}\mathcal A)$, we always mean the $K_0$-group of this triangulated category.

\subsection{Surfaces with marked boundary}
\label{subsec_marked}

In this subsection we define a class of surfaces which will be convenient for our definition of the Fukaya category, as well as cut and paste constructions.
Topologically, they are related to the real blow-ups of flat surfaces by gluing several infinite ends.

We will work with surfaces with corners, $S$, and denote their $0$- and $1$-dimensional strata by $\partial_0 S$ and $\partial_1 S$ respectively. 
The boundary is $\partial S=\partial_0 S\sqcup\partial_1 S$ set-theoretically.
A \textbf{marked surface} is a surface $S$ with corners together with a subset $M\subset\partial S$ which is the closure of a collection of components of $\partial_1 S$ such that if $B$ is a connected component of $\partial S$ then the following holds: If $B$ is smooth, then it belongs entirely to $M$, and if $B$ contains corners, then every other component of $\partial_1 S\cap B$ belongs to $M$, i.e. $\partial_0 S$ is exactly the boundary of $M$ in $\partial S$.
In particular, if $S$ is compact than each component of $\partial S$ is either a circle which belongs entirely to $M$ or a sequence of intervals which alternately belong to $M$ and its complement.

An \textbf{arc} in a marked surface $S$ is an embedded closed interval intersecting $M$ transversely in the endpoints, and not isotopic to an embedded interval in $M$ (the isotopy keeping endpoints in $M$).
We generally identify isotopic arcs.
A \textbf{boundary arc} is an arc which is isotopic to the closure of a component of $\partial S\setminus M$. 
Arcs which are not boundary arcs are referred to as \textbf{internal arcs}.
An \textbf{arc system} in $(S,M)$ is a collection of pairwise disjoint and non-isotopic arcs.
The arc system is \textbf{full} if it includes all the boundary arcs and cuts the surface into contractible, relatively compact components (polygons).

A map of marked surfaces $f:(S_1,M_1)\to (S_2,M_2)$ is an orientation preserving immersion with $f(M_1)\subset M_2$ mapping boundary arcs of $S_1$ to disjoint non-isotopic arcs in $S_2$.
The condition insures that a full arc system in $S_2$ which includes all the images of boundary arcs under $f$ can be lifted to a full arc system on $S_1$.
It is however not closed under composition in general.

\subsubsection*{Combinatorial description: graded ribbon graphs}

Let $(S,M,\eta)$ be a graded marked surface and $A$ a full system of graded arcs on it.
Such a collection of arcs is dual to a graph $\Gamma$ on $S$ having an edge for each element of $A$ and an $n$-valent vertex for each $2n$-gon cut out by $A$ (see Figure~\ref{fig_ribbon_graph2}).
\begin{figure}[h]
\centering
\includegraphics[scale=1]{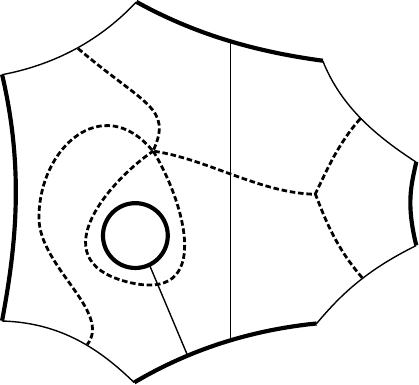}
\caption{Full system of arcs in a marked surface and dual ribbon graph (dashed).}
\label{fig_ribbon_graph2}
\end{figure}
By definition, a \textit{half-edge} is an edge of $\Gamma$ with orientation pointing towards a vertex of the graph (but not a boundary arc).
Since the graph is embedded in an oriented surface, there is a clockwise cyclic order on the set of half-edges pointing towards a given vertex (i.e. $\Gamma$ has the structure of a \textit{ribbon graph}).
Moreover, for any half-edge $h$ the grading determines an integer $d(h)$ as follows.
Let $\sigma(h)$ be the half edge following $h$ in the cyclic order.
Then $h,\sigma(h)$ correspond to a pair of arcs $\alpha,\beta$ in $A$.
Let $c:[0,1]\to M$ be the embedded curve which starts on $\alpha$, ends on $\beta$, follows the boundary with $S$ to the right, and bounds the polygon corresponding to the vertex that $h,\sigma(h)$ point to.
Also choose an arbitrary grading on $c$.
Then
\begin{equation}
d(h)=i_{c(0)}(\alpha,c)-i_{c(1)}(\beta,c)
\end{equation}
is independent of the choice of grading on $c$.
The numbers $d(h)$ have the property that if $v$ is a vertex of $\Gamma$, then
\begin{equation} \label{grg_grading_condition}
\sum_{h}d(h)=\mathrm{val}(v)-2
\end{equation}
where the sum is over all half-edges pointing to $v$ and $\mathrm{val}(v)$ is the valency of $v$.

Conversely, given a ribbon graph together with integers $d(h)$ for every half-edge $h$ satisfying \eqref{grg_grading_condition} one constructs a graded marked surface with arc system by gluing polygons with suitable grading foliation which is tangent to the arcs.
This is a convenient way to specify any compact graded marked surface by a finite amount of data.

\subsection{Minimal $A_\infty$-category of an arc system}
\label{subsec_arc_category}

Fix a field of scalars $\mathbb K$ throughout.
Let $S=(S,M,\eta)$ is a graded marked surface with system of graded arcs $A$.
We define a strictly unital $A_\infty$-category $\mathcal F_A(S)$ with $\mu^1=0$.
\begin{itemize}
\item \textbf{Objects}: The set of arcs in $A$.
\item \textbf{Morphisms}: A \textbf{boundary path} is a non-constant path in $M$ which follows the reverse orientation of the boundary (i.e. the surface lies to the right). 
Given arcs $X$ and $Y$, a basis of morphisms from $X$ to $Y$ is given by boundary paths, up to reparameterization, starting at an endpoint of $X$ and ending at an endpoint of $Y$, as well as the identity morphism if $X=Y$.
The degree of a boundary path $a$ from $p$ to $q$ joining arcs $X$ and $Y$ is by definition
\begin{equation}
|a|=i_p(X,a)-i_q(Y,a)
\end{equation}
for arbitrary grading of $a$.
\item \textbf{Composition}: Let $a,b$ be boundary paths defining morphisms from $X$ to $Y$ and $Y$ to $Z$ respectively. If $a$ and $b$ are composable then $(-1)^{|a|}\mu^2(b,a)=a\cdot b$, otherwise $\mu^2(b,a)=0$.
\item \textbf{Higher operations}: Consider a marked surface $(S',M')$ which is topologically a closed disk with $M'$ having $n\ge 3$ components. 
Let $a_1,\ldots,a_n$ be the distinct boundary paths (components of $M'$) ending in boundary arcs in clockwise order.
Given a map $f:(S',M')\to (S,M)$ sending all boundary arcs of $(S',M')$ to arcs in $A$ we get a sequence $f\circ a_1,\ldots,f\circ a_n$ of boundary paths in $(S,M)$, and call any such sequence a \textbf{disk sequence}.

We define higher $A_\infty$-operations so that if $a_1,\ldots,a_n$ is a disk sequence then
\begin{equation}
\mu^n(a_n,\ldots,a_1b)=(-1)^{|b|}b
\end{equation}
for basis morphisms $b$ with $a_1b\neq 0$, and
\begin{equation}
\mu^n(ba_n,\ldots,a_1)=b
\end{equation}
for paths $b$ with $ba_n\neq 0$, and $\mu^n$ vanishes on all sequences of paths not of the above forms. 
The lemma below ensures that this is really well-defined.
\end{itemize}

\begin{lem}
For a sequence of composable basis morphisms $a_n,\ldots,a_1$ there is at most one factorization $a_1=a_1'b$ with $a_1',a_2,\ldots,a_n$ a disk sequence.
If such a factorization exists with $b$ not an identity, then there is no factorization $a_n=ca_n'$ with $a_1,\ldots,a_{n-1},a_n'$ a disk sequence and $c$ not an identity.
The dual statement also holds. 
\end{lem}

\begin{proof}
To see the first statement note that if we have such a factorization and $\alpha_i$ is the arc on which $a_i$ ends, then the concatenation $a_1'\cdot\alpha_1\cdots a_n\cdot\alpha_n$ is a null-homotopic loop. 
Hence, $b$ must be homotopic (relative endpoints) to the concatenation $a_1\cdot\alpha_1\cdots a_n\cdot\alpha_n$ and is thus uniquely determined as a morphism.

For the second statement, assume such a factorization exists.
Then $c=b$ be the same argument as before, but $b$ and $a_n$ end at a different endpoints of $\alpha_n$, contradicting $a_n=ca_n'$.
\end{proof}

\begin{prop}
With the structure defined above, $\mathcal F_A(S)$ is a strictly unital $A_\infty$-category.
\end{prop}

\begin{proof}
Let $a_n,\ldots,a_1$ be a composable sequence of morphisms corresponding boundary paths.
The claim is that the $A_\infty$-equation \eqref{A_infty_rels} holds.
For $n=1,2,3$ there is nothing to prove, so we may assume $n\ge 4$.
The following types of non-zero terms cancel pairwise, ignoring signs for the moment:
\begin{gather*}
\mu^2(\mu^i(a_n,\ldots,a_1b),c) \\
\mu^i(a_n,\ldots,a_2,\mu^2(a_1b,c))
\end{gather*}
\begin{gather*}
\mu^2(a,\mu^i(bc_n,\ldots,c_1)) \\
\mu^i(\mu^2(a,bc_n),c_{n-1},\ldots,c_1)
\end{gather*}
\begin{gather*}
\mu^i(a_i,\ldots,a_{j+1},\mu^k(a_jb_k,\ldots,b_1),a_{j-1},\ldots,a_1c) \\
\mu^{i+k-2}(a_i,\ldots,a_jb_k,\ldots,\mu^2(b_1,a_{j-1}),a_{j-2},\ldots,a_1c)
\end{gather*}
\begin{gather*}
\mu^i(a_i,\ldots,a_{j+1},\mu^k(b_k,\ldots,b_1a_j),a_{j-1},\ldots,a_1c) \\
\mu^{i+k-2}(a_i,\ldots,a_{j+2},\mu^2(a_{j+1},b_k),\ldots,b_1a_j,a_{j-1},\ldots,a_1c)
\end{gather*}
as well as variants of the previous two with $c$ on the other side, and
\begin{gather*}
\mu^i(a_i,\ldots,a_2,\mu^j(a_1bc_j,c_{j-1},\ldots,c_1)) \\
\mu^j(\mu^i(a_i,\ldots,a_2,a_1bc_j),c_{j-1},\ldots,c_1))
\end{gather*}
All non-zero terms which can appear must indeed belong to one of the above pairs.
To check that the signs are in fact opposite, one uses that 
\begin{equation}
\sum \Vert a_i \Vert=-2
\end{equation}
for any disk sequence $a_1,\ldots,a_n$.
\end{proof}

There are two special cases for which the above description needs to be somewhat amended, as disks also contribute to $\mu^2$. 
The first is that of a square, in which the two boundary arcs are isotopic. 
They represent the same object $X$, which has $\mathrm{End}(X)=\mathbb K$.
The second case is when $S$ is a cylinder and $M=\partial S$.
Up to isotopy there is again only one arc, and the corresponding object $X$ has $\mathrm{End}(X)=\mathbb{K}[z^{\pm 1}]$, with the degree of $z$ depending on $\Omega$.

\begin{rem}
The $A_\infty$-category $\mathcal F_A(S)$ was considered by Bocklandt~\cite{bocklandt11}, at least in the case when $S$ has only marked boundary circles.
He shows that it computes the partially wrapped Fukaya category of $S$ as defined by Abouzaid--Seidel~\cite{abouzaid_seidel}.
We expect the proof to extend to the case of partially wrapped Fukaya categories of Auroux~\cite{auroux10}.
\end{rem}

Next, consider a map of graded marked surfaces $f:S_1\to S_2$ and arc systems $A_i$ in $S_i$ such that $f$ maps graded arcs in $A_1$ to graded arcs in $A_2$.
Then $f$ induces a strict $A_\infty$-functor $f_*$ from $\mathcal F_{A_1}(S_1)$ to $\mathcal F_{A_2}(S_2)$.
This just follows from the following fact: Suppose we have an immersed disk $D$ in $S_2$ whose boundary consists of arcs and boundary curves, then if $\partial D$ lifts to $S_1$, so does $D$.

\subsubsection*{The simplest example}

Consider the case when $S=(S,M,\eta)$ is a graded marked surface which is topologically a disk and $M$ has $n\ge 3$ components, and $A$ is the arc system containing exactly the boundary arcs, with some arbitrary grading.
The $A_\infty$-category $\mathcal F_A(S)$ has objects $E_k$, $k\in\mathbb Z/n$, and morphisms $a_k:E_k\to E_{k+1}$ of some degrees $|a_k|\in\mathbb Z$ with $\sum |a_k|=n-2$ which, together with the identity morphisms, form a basis of all morphisms.
The only non-zero $A_\infty$-terms come from strict unitality and
\begin{equation}
\mu^n(a_{k+n-1},\ldots,a_k)=1_{E_k}.
\end{equation}
See also \cite{nadler} for a discussion of these categories.

We make a simple observation which will be essential in what follows.
Namely, that the twisted complex
\begin{equation*}
E_1\xlongrightarrow{a_1}E_{2}[\| a_1\|]\longrightarrow\ldots\xlongrightarrow{a_{n-2}}E_{n-1}[\Vert a_1\Vert+\ldots+\Vert a_{n-2}\Vert]
\end{equation*}
is isomorphic to $E_n[-|a_n|]$, the inverse isomorphisms being given by $a_{n-1}$ and $a_n$.
In other words, $n-1$ of the boundary arcs already generate all of $\mathrm{Tw}(\mathcal F_A(S))$.

\subsubsection*{Morita invariance}
\label{subsec_morita_invariance}

Fix a graded marked surface $S$.
If $A\subset B$ are full arc systems, i.e. all arcs in $A$ also belong to $B$, then it is evident from the definition that $\mathcal F_A(S)$ is a full subcategory of $\mathcal F_B(S)$.

\begin{lem}
The inclusion functor $\mathcal F_A(S)\to \mathcal F_B(S)$ for full arc systems $A\subset B$ is a Morita equivalence.
\end{lem}

\begin{proof}
We can assume that $B$ is obtained from $A$ by adding a single graded arc $Y$.
Let $D$ be one of the disks cut out by $B$ and with $Y$ on its boundary.
Considering $D$ as a graded marked surface with minimal full arc system, we have a morphism $f:D\to S$ inducing a functor of $A_\infty$-categories.
From the discussion above, we know that in $\mathcal F(D)$ the object corresponding $X$ is isomorphic to a twisted complex over the other boundary arcs, $X_1,\ldots X_n$, of $D$.
As $A$ is a full arc system, there is exactly one boundary arc $X$ of $D$ which gets mapped to $X$, hence $X_1,\ldots,X_n$ get mapped to arcs in $A$.
Applying the functor $f_*$ we thus find that $Y$ is isomorphic to a twisted complex over objects in $\mathcal F_A(S)$.
In particular, the inclusion functor $\mathcal F_A(S)\to \mathcal F_B(S)$ is a Morita equivalence.
\end{proof}

\begin{prop}
The Morita equivalence class of $\mathcal F_A(S)$ is independent of $A$.
The equivalences are canonical in the higher categorical sense (explained in the proof).
\end{prop}

\begin{proof}
Suppose $A$, $B$ are full systems of graded arc.
As established in the previous lemma we get a canonical Morita equivalence when $A\subset B$.
The same is also clearly true when $A$ and $B$ just differ by grading shift.

Consider the set $\mathcal A$ of full arc systems up to isotopy on $(S,M)$, partially ordered by inclusion.
If we view $\mathcal A$ as a category in the usual way, then the arguments above show that we get a functor $A\mapsto\mathcal F_A(S)$ from $\mathcal A$ to the category of strictly unital $A_\infty$ categories and Morita equivalences.
From contractability of the classifying space $|\mathcal A|$ of full arc systems or ribbon graphs (see \cite{harer85,harer86}), it now follows that the various categories $\mathcal F_A(S)$ are canonically Morita equivalent.
\end{proof}

Thus, $\mathrm{Tw}\mathcal F_A(S,M)$ is essentially independent of $A$ and we may drop it from the notation and just write $\mathcal F(S)$, the \textbf{(topological) Fukaya category of $S$}.

\subsection{Formal generators}
\label{subsec_quivers}

Recall that an $A_\infty$-algebra $A$ is \textit{formal} if it is quasi-isomorphic to its cohomology $H^*A$. 
A generator $G$ of an $A_\infty$-category is \textit{formal} if its endomorphism algebra $\mathrm{End}(G)$ is.
We will see in this subsection that $\mathcal F(S)$ has a formal generator whenever $S$ has at least one boundary arc (unmarked part of $\partial S$) in each connected component.
The generator is just a direct sum of arcs, and its endomorphism algebra is the path algebra of a quiver with quadratic monomial relations.

To set up some notation, a \textbf{graded quiver $Q$ with quadratic monomial relations} is given by 
\begin{itemize}
\item $Q_0$ -- set of vertices
\item $Q_1$ -- set of arrows
\item $\partial_0,\partial_1:Q_1\to Q_0$ -- source and target maps
\item $|.|:Q_1\to\mathbb{Z}$ -- grading of arrows
\item $R\subset Q_1\times_{Q_0}Q_1$ -- quadratic monomial relations, i.e. a collection of pairs of composable arrows
\end{itemize}
Given a ground field $\mathbb K$ we may form the path category $\mathbb KQ=\mathbb KQ/R$, which is a graded linear category with set of objects $Q_0$ and basis of morphisms given by paths of arrows not containing any of the relations.

Let $S$ be a graded marked surface, compact for simplicity.
Note first that if $A$ is any arc system on $S$, then the graded category obtained from $\mathcal F_A(S)$ by forgetting the higher $\mu^k$ is of the form $\mathbb KQ$ for some graded quiver with quadratic monomial relations.
Namely, take as arrows the boundary paths which start and end at arcs of $A$ but do not cross any other endpoints, and quadratic relations coming from composable arrows which do not correspond to composable paths.
Furthermore, the higher $A_\infty$ operations of $\mathcal F_A(S)$ vanish under the following condition on $A$: Any disk cut out by $A$ is bounded by a boundary arc of $S$ which does not belong to $A$.
We call such systems of arcs \textbf{formal}.
Thus, under this condition, the category $\mathcal F_A(S)$ is of the form $\mathbb KQ$ for some graded quiver with quadratic monomial relations $Q$.
Call a formal system of arcs \textbf{full} if it cuts $S$ into disks each of which has exactly a single boundary arc (of itself) not belonging to $A$.
For a full formal system of arcs $A$ the category $\mathrm{Tw}(\mathcal F_A(S))$ is quasi-equivalent to $\mathcal F(S)$.

\begin{lem}
Let $S$ be a compact and connected graded marked surface. 
If $S$ has at least one boundary arc, then it has a full formal system of arcs.
\end{lem}

\begin{proof}
This follows from the fact that we can find a system of arcs on $S$ which cuts it into a single disk, the boundary of which must include all boundary arcs of $S$.
This in turn is seen from the dual ribbon graph. 
If it has more than one vertex we can contract some edges to reduce the number of vertices to one eventually.
\end{proof}

We say a graded quiver with relations is of \textbf{type F1} if it arises as above from a full formal system of arcs.
It is not difficult to see that a graded quiver with quadratic monomial relations is of type F1 if and only if
\begin{enumerate}
\item There are no cycles $a_1,\ldots,a_n$, $n\ge 1$, with $a_ia_{i+1}=0$ for all $i\in\mathbb{Z}/n$.
\item Each vertex has at most two incoming and outgoing arrows.
\item Let $a,b\neq c$ be arrows. If $ab,ac$ are defined, then $ab=0$ or $ac=0$, but not both. If $ba,ca$ are defined, then $ba=0$ or $ca=0$, but not both.
\end{enumerate}

\subsubsection*{Resolution of the diagonal}

Let $Q$ be a graded quiver with quadratic monomial relations and $A=\mathbb KQ/R$ be its path algebra.
We want to investigate regularity properties of $A$.
Denote the constant path at the vertex $i$ by $e_i$.
There is a splitting of $A$
\begin{equation}
A=\bigoplus_{i\in Q_0}Ae_i,\qquad A=\bigoplus_{i\in Q_0}e_iA
\end{equation}
as a left (resp. right) module over $A$.

Define $A^{\mathrm{op}}\otimes A$-modules
\begin{equation}
M_n = \bigoplus_{\substack{\alpha_1,\ldots,\alpha_n\in Q_1 \\ (\alpha_i,\alpha_{i+1})\in R}}Ae_{\partial_0\alpha_1}\otimes e_{\partial_1\alpha_n}A,\qquad n\ge 1
\end{equation}
\begin{equation}
M_0 = \bigoplus_{i\in Q_0}Ae_i\otimes e_iA
\end{equation}
connected by maps $f_n:M_n\to M_{n-1}$ such that for $a\otimes b$ in the $(\alpha_1,\ldots,\alpha_n)$-component of $M_n$
\begin{equation}
f_n(a\otimes b)=a\alpha_1\otimes b+(-1)^na\otimes\alpha_n b
\end{equation}
and $f_0:M_0\to A$, $f_0(a\otimes b)=ab$.

\begin{prop}
The sequence of bimodules
\begin{equation}
\begin{tikzcd}
\ldots \arrow{r} & M_2 \arrow{r}{f_2} & M_1 \arrow{r}{f_1} & M_0 \arrow{r}{f_0} & A \arrow{r} & 0
\end{tikzcd}
\end{equation}
is exact.
\end{prop}

\begin{proof}
For $a\otimes b\in M_n$ in the $(\alpha_1,\ldots,\alpha_n)$-component we have
\begin{align*}
f_{n-1}&(f_n(a\otimes b))= \\
&a\alpha_1\alpha_2\otimes b+(-1)^{n-1}a\alpha_1\otimes\alpha_n b+(-1)^na\alpha_1\otimes\alpha_n b-a\otimes\alpha_{n-1}\alpha_n b \\
&=0
\end{align*}
hence $f_{n-1}\circ f_n=0$.

It remains to show that $\mathrm{Ker}(f_n)\subset\mathrm{Im}(f_{n+1})$.
An element $x\in M_n$ is uniquely written as
\begin{equation}
x=\sum_{\substack{\alpha_1,\ldots,\alpha_n\in Q_1 \\ (\alpha_i,\alpha_{i+1})\in R}} \sum_{\substack{\quad\alpha\text{ path},\neq 0\text{ in }A \\ \partial_1\alpha=\partial_0\alpha_1}} \alpha\otimes b_{\alpha_1,\ldots,\alpha_n,\alpha}
\end{equation} 
with $b_{\alpha_1,\ldots,\alpha_n,\alpha}$ almost all zero.
For $x\neq 0$ define
\begin{equation}
l(x)=\max\left\{l(\alpha)\mid b_{\alpha_1,\ldots,\alpha_n,\alpha}\neq 0\right\}
\end{equation}
where $l(\alpha)$ is the length (i.e. number of arrows) of the path $\alpha$.
Observe that by definition of $f_n$, for every $x\neq 0$ there is a $y\in M_n$ with
\begin{equation}
l(y)=0,\qquad x-y\in\mathrm{Im}(f_{n+1}).
\end{equation}
Suppose now $x\in M_n$ with $l(x)=0$, so we can write
\begin{equation}
x=\sum_{\substack{\alpha_1,\ldots,\alpha_n \\ (\alpha_i,\alpha_{i+1})\in R}}e_{\partial_0\alpha_1}\otimes b_{\alpha_1,\ldots,\alpha_n}.
\end{equation}
The projection of $f_n(x)$ to the $\alpha_1,\ldots,\alpha_{n-1}$-component is
\begin{equation}
\sum_{\substack{\alpha\in Q_1\\ (\alpha,\alpha_1)\in R}}\alpha\otimes b_{\alpha,\alpha_1,\ldots,\alpha_{n-1}}+(-1)^n\sum_{\substack{\alpha\in Q_1\\ (\alpha_{n-1},\alpha)\in R}}e_{\partial_0\alpha_1}\otimes\alpha b_{\alpha_1,\ldots,\alpha_{n-1},\alpha}.
\end{equation}
Therefore, $f_n(x)=0$ implies $b_{\alpha_1,\ldots,\alpha_n}=0$ for all $\alpha_1,\ldots,\alpha_n\in Q_1$, $(\alpha_i,\alpha_{i+1})\in R$, so $x=0$, which completes the proof.
\end{proof}

\begin{prop} \label{gqqr_reg}
Let $Q$ be a graded quiver with quadratic monomial relations, $A=\mathbb KQ/R$ its path algebra.
\begin{enumerate}
\item Suppose there are no cyclic paths $\alpha_1\cdots\alpha_n$ in $Q$ with $(\alpha_i,\alpha_{i+1})\in R$, $i\in\mathbb{Z}/n$, then $A$ is homologically smooth.
\item Suppose there are no cyclic paths $\alpha_1\cdots\alpha_n$ in $Q$ with $(\alpha_i,\alpha_{i+1})\notin R$, $i\in\mathbb{Z}/n$, then $A$ is proper.
\end{enumerate}
\end{prop}

\begin{proof}
1. Under the stated condition on cyclic paths, we see that $M_n=0$ for $n\gg 0$, hence $A$ is perfect as an $A^{\mathrm{op}}\otimes A$-module.

2. The condition implies that only finitely many paths are non-zero in $A$, so $A$ has finite rank over $\mathbb K$.
\end{proof}

We apply these results to graded linear models of $\mathcal{F}(S)$.
In this case, the first condition of Proposition~\ref{gqqr_reg} is always satisfied, while the second is satisfied if $S$ has no boundary components without corners.

\begin{cor} \label{FS_dualizable}
Let $S$ be a compact graded marked surface without boundary components diffeomorphic to $S^1$, then $\mathcal{F}(S)$ is homologically smooth and proper.
\end{cor}

\subsection{Localization}
\label{subsec_localization}

We begin by recalling a version Drinfeld's construction for strictly unital $A_\infty$-categories, studied in detail in \cite{lo06}.
Let $\mathcal A$ be a strictly unital $A_\infty$-category over a field $\mathbb K$ and $E\in\mathrm{Ob}(\mathcal A)$. 
It will be convenient to use the notation $\mathcal A(X,Y)$ for $\mathrm{Hom}(X,Y)$ in this subsection.
We will define the quotient category $\mathcal A/E=\mathcal B$ of $\mathcal A$ by $E$, which will again be a strictly unital $A_\infty$-category.
Informally, $\mathcal B$ is obtained by freely adjoining a morphism $\varepsilon\in\mathcal B^{-1}(E,E)$ with $\mu^1(\varepsilon)=1_E$.
Set
\begin{align}
\mathcal B(X,Y)=&\mathcal A(X,Y) \oplus\nonumber \\
&\oplus\left(\mathcal A(E,Y)\otimes\mathbb K[1]\otimes\left(\bigoplus_{n\ge 0}\left(\mathcal A(E,E)\otimes\mathbb K[1]\right)^{\otimes n}\right)\otimes\mathcal A(X,E)\right)
\end{align}
as a graded vector space.
Write generators of the summand 
\[
\mathcal A(E,Y)\otimes\mathbb K[1]\otimes\left(\mathcal A(E,E)\otimes\mathbb K[1]\right)^{\otimes n-2}\otimes\mathcal A(X,E)
\]
as 
\[
a_n\cdot\ldots\cdot a_1
\]
with $a_n\in\mathcal A(E,Y)$, $a_1\in\mathcal A(X,E)$, and $a_i\in\mathcal A(E,E)$ for $2\leq i\leq n-1$, then
\begin{align}
|a_n\cdots\ldots\cdot a_1|&=|a_n|+\ldots+|a_1|-n+1 \\
\Vert a_n\cdots\ldots\cdot a_1\Vert&=\Vert a_n\Vert+\ldots+\Vert a_1\Vert.
\end{align}
Structure maps are given by
\begin{align}
&\mu^r(a_{n_r}\cdot\ldots\cdot a_{n_{r-1}+1},\ldots,a_{n_1}\cdot\ldots\cdot a_1):= \nonumber \\
&\sum_{\substack{i+j+k=n_r \\ j+k>n_{r-1} \\ k<n_1}}(-1)^{\Vert a_k\Vert+\ldots+\Vert a_1\Vert}a_{n_r}\cdot\ldots\cdot\mu^j(a_{j+k},\ldots,a_{k+1})\cdot a_k\cdot\ldots\cdot a_1
\end{align}
where $0=n_0<n_1<\ldots<n_r$, $r\ge 1$.
All this generalizes in a straightforward manner to quotients by full subcategories $\mathcal E\subset\mathcal A$.

We return to the setting of surfaces with marked boundary where we consider the following modification.
If $S$ is a marked surface and $E$ a boundary arc, we get a new marked surface $S'$ by adding $E$ to the marked points and smoothing the two corners on which $E$ ends.
We do not quite get a morphism $S\to S'$, since $E$ would need to map to an arc which isotopic to a path in $M'$.
For this reason we allow such arcs, which we call \textbf{null}, in this subsection.
The definition of the $A_\infty$-category of a system of arcs works as before with the following small change: The category is no longer minimal, and the definition of $\mu^1$ is modeled on the one for the higher $\mu^i$. 
This means that the boundary path which is isotopic to a null arc $E$ is a disk-sequence of length 1 and has differential the identity morphisms of that arc. 

\begin{prop} \label{prop_loc}
Let $A$ be an arc system on a graded marked surface $S$ and $E\in A$ a boundary arc.
By the modification as above we get $S'$ with arc system $A'$ containing a null arc.
Then there is a natural equivalence of $A_\infty$-categories
\begin{equation}
\mathcal F_A(S)/E\cong\mathcal F_{A'}(S')
\end{equation}
under $\mathcal F_A(S)$.
\end{prop}

\begin{proof}
Comparing definitions, one notices an evident strict $A_\infty$-functor
\begin{equation}
G:\mathcal F_{A'}(S')\to \mathcal F_A(S)/E
\end{equation}
compatible with the functors from $\mathcal F_A(S)$.
It is obtained by factoring a boundary path in $M'$ into its pieces which are alternately contained in $M$ and $E$.
The functor $G$ is not an isomorphism of $A_\infty$-categories, but becomes one once we drop $E$ from both the target and the source category.
By general properties, $E\in H^*(\mathcal F_A(S)/E)$ is a zero-object, and it is easily seen that the same is true in $H^*(\mathcal F_{A'}(S'))$, implying that $G$ is a quasi-equivalence.
\end{proof}

\subsection{Cosheaf of categories}
\label{subsec_cosheaf}

Let $S$ be a graded marked surface with full system of graded arcs $A$.
Dual to $A$ is a ribbon graph $G$ on $S$ which is a deformation retract of the pair $(S,\partial S\setminus M)$.
A constructible cosheaf $\mathcal E$ of $A_\infty$-categories on $G$ is given, concretely, by assigning an $A_\infty$-category $\mathcal C_v$ (resp. $\mathcal C_e$) to each internal vertex $v$ (resp. edge $e$) of $G$, together with functors $\mathcal C_e\to\mathcal C_v$ for every half-edge.
We understand this as a sheaf with values in the higher category of $A_\infty$-categories with isomorphisms the Morita-equivalences, so when taking global sections one needs to compute a homotopy colimit.

There is a canonical constructible cosheaf on $G$, defined as follows.
For an edge $e$, $\mathcal C_e$ is the category with a single object corresponding to the arc in $A$ dual to $e$ and endomorphism algebra $\mathbb K$.
For a vertex $v$, let $D$ be the dual polygon cut out by $A$, then $\mathcal C_v$ is the category $\mathcal F(D)$ where we take the arc system consisting just of boundary arcs.
Thus, $\mathcal C_v$ is Morita-equivalent to the path-algebra of an $A_n$ quiver, where $n+1$ is the valency of $v$.
The functors $\mathcal C_e\to \mathcal C_v$ are fully-faithful inclusions of boundary arcs of $D$.
By construction, we also have functors $\mathcal C_e\to\mathcal F_A(S)$ and $\mathcal C_v\to\mathcal F_A(S)$, and all these are compatible.
Hence, by universality there is a functor from global sections of $\mathcal E$, $\Gamma(G,\mathcal E)\to\mathcal F_A(S)$.
The main result of this subsections is that this is a Morita-equivalence.

\begin{thm} \label{CosheafGlobalSections}
Let $S$ be a graded marked surface with ribbon graph $G$ dual to a full graded system of arcs.
Then $\mathcal F(S)$ represents global sections of the constructible cosheaf $\mathcal E$ on $G$ defined above.
\end{thm}

The category of global sections $\Gamma(G,\mathcal E)$ is by definition the homotopy colimit of the diagram formed by the $\mathcal C_e$, $\mathcal C_v$, and the functors between them.
Thus we need to show that $\mathcal F(S)$ represents this homotopy colimit.
The strategy of the proof will be to first treat the case with boundary arcs where we have a formal generator, then deduce the general case by localization.

Let $Q_n$ be the quiver with $n\ge 1$ vertices, arrows $\alpha_{i,i+1}$ from the $i$-th to the $(i+1)$-st vertex of degrees $|\alpha_{i,i+1}|\in\mathbb Z$, and quadratic relations $\alpha_{i,i+1}\alpha_{i-1,i}=0$ for $1<i<n$.
Suppose that $Q=Q_{n_1}\sqcup\ldots\sqcup Q_{n_k}$ is a finite disjoint union quivers of this type, and that $P$ is a finite set with maps $f_1,f_2:P\to Q_0$ such that $f_1\sqcup f_2:P\sqcup P\to Q$ is injective.
Construct a new quiver with relations $R$ from $Q$ by identifying the vertices $f_1(p)$ and $f_2(p)$ for each $p\in P$, and keeping the same arrows and relations.
Passing to the associated path-categories $\mathcal P,\mathcal Q,\mathcal R$, where $P$ is thought of as a quiver without arrows, we get a coequalizer diagram
\begin{equation} \label{GradedQuiverCoequalizer}
\begin{tikzcd}
\mathcal P \arrow[bend left=15]{r}{f_1} \arrow[bend right=15]{r}[swap]{f_2} & \mathcal Q \arrow{r} & \mathcal R
\end{tikzcd}
\end{equation}
in the naive (1-categorical) sense.
\begin{prop}
The diagram \eqref{GradedQuiverCoequalizer} is a homotopy coequalizer in the category of dg-categories.
\end{prop}

\begin{proof}
To show this, we will make a cofibrant replacement $\widetilde{\mathcal Q}$ of $\mathcal Q$ making the diagram 
\begin{equation} \label{CofibrantCoqualizer}
\begin{tikzcd}
\mathcal P \arrow[bend left=15]{r}{f_1} \arrow[bend right=15]{r}[swap]{f_2} & \widetilde{\mathcal Q}
\end{tikzcd}
\end{equation}
cofibrant by the condition on $f_1\sqcup f_2$, and verify and the coequalizer of \eqref{CofibrantCoqualizer} is quasi-isomorphic to $\mathcal R$.

The graded linear path-category $\mathcal A$ of $Q_n$ has a cofibrant replacement given by the bar-cobar resolution $\mathcal B=\Omega B\mathcal A$.
The dg-category $\mathcal B$ is described explicitly by a dg-quiver with $n$ vertices, arrows $\beta_{ij}$ from the $i$-th to the $j$-th vertex, $1\le i<j\le n$, with degrees
\begin{equation}
|\beta_{ij}|=1+\sum_{i\le k<j}\left(|\alpha_{k,k+1}|-1\right)
\end{equation}
and differential
\begin{equation}
d\alpha_{ij}=\sum_{i<k<j}(-1)^{|\alpha_{kj}|}\alpha_{kj}\alpha_{ik}
\end{equation}
extended to paths by the graded Leibniz rule.
The functor $P:\mathcal B\to\mathcal A$ sends $\beta_{i,i+1}$ to $\alpha_{i,i+1}$ and all other arrows to zero.
The right inverse $I$ sending $\alpha_{i,i+1}$ to $\beta_{i,i+1}$ is not multiplicative, but an inverse of $P$ on the level of cohomology, as the homotopy $H$ defined by
\begin{equation} \label{BarCobarH}
H(\beta_{i_{k-1},i_k}\cdots\beta_{i_1,i_2})=\begin{cases} 
(-1)^{|\beta_{i_{k-1},i_k}|}\beta_{i_{k-2},i_k}\cdots\beta_{i_1,i_2} & \text{if }i_k=i_{k-1}+1,k\ge 3 \\
0 & \text{else}
\end{cases}
\end{equation}
satisfying $dH+Hd=1-IP$, shows.

We have a cofibrant replacement $\widetilde{\mathcal Q}=\Omega B\mathcal Q$ which is described by a disjoint union of the dg-quivers above.
The coequalizer, $\widetilde{\mathcal R}$ of the diagram \eqref{CofibrantCoqualizer} is associated with the dg-quiver we get by identifying pairs of vertices.
There is a dg-functor $\widetilde{\mathcal R}\to\mathcal R$, which we claim is a quasi-isomorphism.
To define the homotopy, let $\beta_1,\ldots,\beta_n$ be non-identity paths in the quiver of $\widetilde{\mathcal Q}$ which are not composable, but become so in $\widetilde{\mathcal R}$, so that $b:=\beta_n\cdots\beta_1$ is a morphism in $\widetilde{\mathcal R}$.
Note that any non-identity path in the quiver of $\widetilde{\mathcal R}$ canonically breaks up in such a way.
Also define the length of a path, $l(\beta)$, so that $\l(\beta_{ij})=j-i$ and $l(\beta_2\beta_1)=l(\beta_2)+l(\beta_1)$.
We define $H$ extending $\eqref{BarCobarH}$ in the following way.
If $l(\beta_k)=1$ for $1\le k \le n$, then $H(\beta_n\cdots\beta_1)=0$.
Otherwise, there exists a largest $k$ with $l(\beta_k)>1$ and 
\begin{equation}
H(\beta_n\cdots\beta_1):=(-1)^{|\beta_n|+\ldots+|\beta_{k+1}|}\beta_n\cdots H(\beta_k)\cdots\beta_1.
\end{equation}
Let us check the identity $dH+Hd=1-IP$ in $\widetilde{\mathcal R}$.
If $l(\beta_k)=1$ for all $k$, then evidently both sides vanish.
Otherwise, we have a largest $k$ with $l(\beta_k)>1$ and the right-hand side vanishes.
On the other hand,
\begin{equation}
(dH)(b)=(dH)(\beta_k)+\sum_{j<k}(-1)^{|H(\beta_k)|+\ldots+|\beta_{j-1}|}\beta_n\cdots H(\beta_k)\cdots d\beta_j\cdots\beta_1
\end{equation}
and
\begin{equation}
(Hd)(b)=(Hd)(\beta_k)+\sum_{j<k}(-1)^{|\beta_k|+\ldots+|\beta_{j-1}|}\beta_n\cdots H(\beta_k)\cdots d\beta_j\cdots\beta_1
\end{equation}
so that the left-hand side also vanishes.

We conclude that the original coequalizer diagram \eqref{GradedQuiverCoequalizer} is quasi-isomorphic to a cofibrant one, and hence also a homotopy coequalizer.
\end{proof}

\begin{proof}[Proof of Theorem \ref{CosheafGlobalSections}]
Let $S$ be a graded marked surface, which we may as well assume to be connected.
Suppose first that $S$ has at least one boundary arc.
Then $S$ has a full formal system of arcs $A$ with dual ribbon graph $G$.
Recall that $\Gamma(G,\mathcal E)$ is the homotopy colimit of a diagram with the various categories $\mathcal C_e,\mathcal C_v$ and functors between them.
Removing the categories $C_e$ for $e$ an edge ending in a boundary arc from the diagram does not change the colimit, since such $C_e$ are being included as subcategories into some $\mathcal C_v$ only once.
In the reduced diagram there is now, by formality of the arc system, at least one object in each $\mathcal C_v$ which is not in the image of any functor $\mathcal C_e\to\mathcal C_v$.
We remove such an object from each $\mathcal C_v$, producing a Morita-equivalent diagram.
Note that the resulting diagram is now a diagram of graded linear categories.
It is essentially a diagram of the form \eqref{GradedQuiverCoequalizer} where $\mathcal P$ is the coproduct of the $\mathcal C_e$ (with $P$ identified with the set of arcs in $A$) and $\mathcal Q$ is the coproduct of the $\mathcal C_v$.
The coequalizer $\mathcal R$ is just the graded linear model for $\mathcal F(S)$, and we know it is a homotopy coequalizer by the previous proposition.

The previous arguments also show that $\Gamma(G,\mathcal E)$ is independent of $G$, and hence that the theorem holds in fact for \textit{any} full system of arcs on a surface with boundary arc.
Namely, it suffices to check that global sections do not change when an edge of $G$ is contracted.
This can be checked in a neighborhood of that edge, so we can assume that $G$ has only one internal edge and is dual to a full arcs system $A$ on the disk $D$ with single internal arc.
Let $G'$ be the graph obtained by contracting the internal edge in $G$, and $\mathcal E'$ its cosheaf of categories.
As it is already established that both $\Gamma(G,\mathcal E)$ and $\Gamma(G',\mathcal E)$ are represented by $\mathcal F(D)$, they must be equivalent.

What remains is the case when $S$ has no boundary arcs.
As discussed in Subsection~\ref{subsec_localization}, the category $\mathcal F(S)$ is a localization of a category $\mathcal F(S')$ where $S'$ is obtained from $S$ by inserting a boundary arc on some boundary component.
So to complete the proof of the theorem it suffices to check that the corresponding categories of global sections have the same relation.

Choose some ribbon graph $G'$ on $S'$ dual to a full system of graded arcs.
To get a ribbon graph $G$ for $S$ we just need to remove the edge ending on the unique boundary arc.
This means that in the diagram computing global sections we change some $\mathcal C_v$ from type $A_n$ to type $A_{n-1}$, which is indeed just localization by the boundary arc.
Finally, taking the quotient by some object is a special case of a homotopy push out, thus commutes with colimits, so that $\Gamma(G,\mathcal E)$ is the quotient of $\Gamma(G',\mathcal E')$ by the boundary arc.
\end{proof}


\section{Tameness and geometricity}
\label{sec_tameness}

In this section we deal with the problem of classifying objects in $\mathcal F(S)$.
First, in Subsection~\ref{subsec_complexes_from_curves}, we assign objects in $\mathcal F(S)$ to certain immersed curves in $S$ with local system of vector spaces.
The purpose of the rest of the section is to prove that we get all objects in this way.
In Subsection~\ref{subsec_nets} we introduce \textit{nets} and study their representations.
Subsection~\ref{subsec_minimal_complexes} discusses a minimality condition on twisted complexes which ensures uniqueness up to isomorphism.
The proof of the classification is completed in Subsection~\ref{subsec_geometricity}.

\subsection{Twisted complexes from curves}
\label{subsec_complexes_from_curves}

Fix a graded marked surface $S$ and a ground field $\mathbb K$.
An immersed curve $c$ in $S$ is \textbf{unobstructed} if it does not bound an immersed teardrop, which is a map from the closed disk $D$ to the surface which takes $\partial D$ to $c$ and which is a smooth immersion at every point except one point of $\partial D$ (see Figure \ref{fig_teardrop}).
\begin{figure}[h]
\centering
\includegraphics[scale=1]{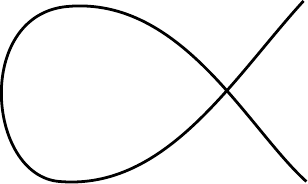}
\caption{Curve bounding a teardrop.}
\label{fig_teardrop}
\end{figure}
An \textbf{admissible curve} is an unobstructed graded curve $c$ such that one of the following holds:
\begin{enumerate}
\item The domain of $c$ is $S^1$, the image of $c$ is disjoint from $\partial S$, and $c$ represents a primitive class in $\pi_1(S)$.
\item The domain of $c$ is $[0,1]$, $c$ intersects $\partial S$ transversely in $M$ and only in the endpoints, and $c$ is not homotopic relative endpoints to a path in $M$.
\end{enumerate}
Suppose that $S$ is compact so that we have a category $\mathcal F(S)$, well defined up to canonical equivalence.
The purpose of this subsection is to show that an admissible curve together with a local system of finite dimensional $\mathbb K$-vector spaces (on its domain) gives an equivalence class of objects in $\mathcal F(S)$.

First, in the case when $S$ is topologically a disk any admissible curve $c$ is a graded arc. 
Thus we can find a full arc system $A$ which includes $c$, so that $c$ is an object of $\mathcal F_A(S)$ by definition.
The isomorphism class of that object is clearly well-defined in $\mathcal F(S)$ independently of the arc system.

Returning to the case of general $S$, we will first deal with admissible curves $c$ which have domain $[0,1]$.
For any such $c$ we can find a graded marked surface $S'$ which is of disk-type and with a map $f$ to $S$ so that $c$ is the image of an admissible curve $\tilde{c}$ under $f$.
To see this, consider the universal cover $\widetilde{S}$ of $S$ and lift a full arc system $A$ on $S$ to $\widetilde{A}$ on $\widetilde{S}$.
Lift $c$ to $\tilde{c}$ on $\widetilde{S}$ and take as $S'$ a closed disk which is cut out by arcs in $\widetilde{A}$ and which contains $\tilde{c}$.
Now, $\tilde{c}$ gives an equivalence class of objects in $\mathcal F(S')$, and the image under the functor $\mathcal F(S')\to\mathcal F(S)$ is independent of the choice of $(S')$.
This follows from the fact that if we have $S'$, $S''$ as above, then the maps $S'\to S$, $S''\to S$ both factor through a third $S'''\to S$ as can be seen by looking at the universal cover again.

Suppose now instead that $c$ is an admissible curve with domain $S^1$ and local system $V$ of finite-dimensional vector spaces on it.
We will follow the same strategy as before and assume first that $(S,M)$ is of annular type, i.e. topologically a compact annulus with corners on each boundary component.
Choose a cyclic sequence of disjoint non-isotopic arcs $X_i$, $i\in\mathbb Z/n$ so that at least one connects the two components of $\partial S$ and such that every component of $M$ contains either exactly two endpoints of the arcs, belonging $X_i,X_{i+1}$ for some $i$, or none of the endpoints (see Figure \ref{fig_annulus_arcs}).
\begin{figure}[h]
\centering
\includegraphics[scale=1]{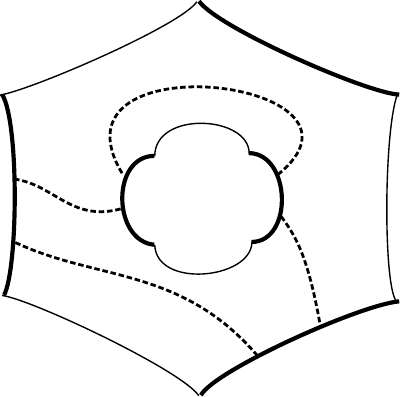}
\caption{Example of allowed sequence of arcs (dashed).}
\label{fig_annulus_arcs}
\end{figure}
Thus we get a sequence $a_i$, $i\in\mathbb Z/n$ of distinct boundary paths so that $a_i$ connects endpoints of $X_i,X_{i+1}$.
If we follow $X_0,a_0,X_1,a_1,\ldots$ we get a path which after suitable smoothing near the intersection points becomes a simple closed loop isotopic to $c$.
It is possible to choose grading and a local systems on $X_i$, $a_i$, so that the smoothed path is isotopic to $c$ as a graded curve with local system.
As a result, each $a_i$ will be morphism of degree $1$ either from $X_i$ to $X_{i+1}$ or in the other direction.
Further we get a vector space $V_i$ of sections of over $X_i$ and parallel transport $T_i:V_i\to V_{i+1}$ along $a_i$.
Consider the twisted complex
\begin{equation}
\left(\bigoplus V_i\otimes X_i,\sum T_i^{\pm 1}\otimes a_i\right)
\end{equation}
where the signs are determined by the direction of the morphisms $a_i$.
This is the object of $\mathcal F(S,M,\Omega)$ we assign to $c$.

\begin{lem}
The equivalence class of the twisted complex constructed depends only on the isotopy class of the graded curve $c$ with local system.
\end{lem} 

\begin{proof}
We claim first that we can replace $X_1,\ldots,X_{n-1}$ be a single arc $Y$ with the pair $X_0,Y$ giving an isomorphic twisted complex.
Indeed, cutting $S$ along $X_0$ we get a surface of disk type $S'$ in which the sequence of arcs $X_1,\ldots,X_{n-1}$ concatenates to a single graded arc $Y$.
The corresponding isomorphism of twisted complexes formed from $X_1,\ldots,X_{n-1}$ and $Y$ respectively was established in Subsection~\ref{subsec_morita_invariance}, and the claimed isomorphism of twisted complexes formed from $X_0,\ldots,X_{n-1}$ and $X_0,Y$ follows.

We have reduced the problem to the case of two arcs.
Any two pairs of arcs satisfying our requirements are related by Dehn twists (automorphisms) of $S$.
It is clear that the isotopy class of $c$ is invariant under Dehn twists.
To finish the proof we need to show that the twisted complex associated with $c$ is invariant under the induced autoequivalence, up to isomorphism.
This can be checked by direct computation, or by using the equivalence
\begin{equation}
H^0(\mathrm{Tw}(\mathcal F_{X_0,Y}(S)))=D^b(\mathbb P^1)
\end{equation}
under which $X_0,Y$ correspond to $O,O(1)[-1]$, the Dehn-twist to $\otimes O(1)$, and the twisted complex assigned to $c$ to a torsion sheaf.
\end{proof}

We have excluded above the case when one or both boundary components are entirely contained in $M$.
These cases can be handled fairly easily directly, or alternatively by the localization construction of the previous section.
The case of general $S$ is handled by finding maps from surfaces of annular type.
Here the argument uses the annular covering associated with $c$ instead of the universal covering.

\subsection{Representations of nets}
\label{subsec_nets}

In this subsection we classify representations of certain combinatorial structures, called \textit{nets}, based on ideas in the work of Nazarova--Roiter~\cite{nazarova_roiter} on a class of tame linear algebra problems.
We will later use this result to derive the classification of objects in Fukaya categories of surfaces with boundary.

A \textbf{net} is a quadruple $\mathcal X=(A,\alpha,B,\beta)$ where
\begin{itemize}
\item $A$ is a finite set,
\item $\alpha$ is a fixed-point-free involution on $A$,
\item $B$ is a finite set with partition $B=\bigsqcup_{i\in A}B_i$ into totally ordered sets,
\item $\beta$ is a fixed-point-free involution on $\mathrm{Dom}(\beta)\subset B$.
\end{itemize}
A morphism of nets $f:(A,\alpha,B,\beta)\to (A',\alpha',B',\beta')$ is given by maps $f_1:A\to A'$, $f_2:B\to B'$ such that 
\begin{align}
f_1\circ\alpha&=\alpha'\circ f_1 \\
f_2(\mathrm{Dom}\beta)&\subset\mathrm{Dom}\beta' \\
f_2(B\setminus\mathrm{Dom}\beta)&\subset B'\setminus\mathrm{Dom}\beta' \\
f_2|_{\mathrm{Dom}\beta}\circ\beta&=\beta'\circ f_2|_{\mathrm{Dom}\beta} \\
f_2(B_i)\subset {B'}_{f_1(i)}\text{ and } &f_2|_{B_i} \text{ is increasing }
\end{align}
The \textbf{height} of a net $\mathcal X=(A,\alpha,B,\beta)$ is defined as
\begin{equation}
h(\mathcal X)=\max_{i\in A}|B_i|.
\end{equation}
Nets of height 1 are disjoint unions of the following types of nets:
\begin{equation}
\label{string_type_net}
\begin{tikzcd}
\bullet\arrow[leftrightarrow]{r}{\alpha} & \bullet\arrow[leftrightarrow]{r}{\beta} & \bullet\arrow[leftrightarrow]{r}{\alpha} & \bullet\arrow[leftrightarrow]{r}{\beta} & \cdots\arrow[leftrightarrow]{r}{\alpha} & \bullet
\end{tikzcd}
\end{equation}
and
\begin{equation}
\label{band_type_net}
\begin{tikzcd}
\bullet\arrow[leftrightarrow]{r}{\alpha}\arrow[leftrightarrow,bend right=20]{rrrrr}{\beta} & \bullet\arrow[leftrightarrow]{r}{\beta} & \bullet\arrow[leftrightarrow]{r}{\alpha} & \bullet\arrow[leftrightarrow]{r}{\beta} & \cdots\arrow[leftrightarrow]{r}{\alpha} & \bullet
\end{tikzcd}
\end{equation}
where we have indicated elements of $A=B$ by dots and the action of $\alpha,\beta$ by arrows.
In both cases $|A|\ge 2$.

A \textbf{representation of a net} $(A,\alpha,B,\beta)$ is given by 
\begin{itemize}
\item a finite-dimensional vector space $V$ with a direct sum decomposition 
\[
V=\bigoplus_{i\in A/\alpha}V_i,
\]
\item for each $i\in A$ an increasing exhaustive filtration $\{F_j\}_{j\in B_i}$ on $V_i$ (i.e. $j\leq k\implies F_j\subset F_k$, $F_{\max B_i}=V_i$),
\item isomorphisms 
\[
\phi_i:\mathrm{gr}_iV\to\mathrm{gr}_{\beta(i)}V,\qquad i\in\mathrm{Dom}(\beta)
\]
with $\phi_{\beta(i)}=\phi_i^{-1}$.
Here $\mathrm{gr}_iV=F_i/F_{<i}$ is the associated graded.
\end{itemize} 
Representations of a net form a linear category with morphisms $(V,F_i,\phi_i)\to(W,G_i,\psi_i)$ being linear maps $f:V\to W$ such that $f(F_i)\subset G_i$ and 
\[
\begin{tikzcd}
\mathrm{gr}_iV \arrow{r}{\mathrm{gr}_if} \arrow{d}{\phi_i} & \mathrm{gr}_iW \arrow{d}{\psi_i} \\
\mathrm{gr}_{\beta(i)} \arrow{r}{\mathrm{gr}_{\beta(i)}f} V & \mathrm{gr}_{\beta(i)}W 
\end{tikzcd}
\]
commutes for all $i\in\mathrm{Dom}(\beta)$.
We note that the category of representations of a net as in \eqref{string_type_net} is the category of finite dimensional vector spaces, while for \eqref{band_type_net} we get the category of finite dimensional representations of $\mathbb Z$ (up to equivalence).

Given a morphism of nets $f=(f_1,f_2):\mathcal X=(A,\alpha,B,\beta)\to \mathcal X'=(A',\alpha',B',\beta')$ with $f_2$ injective on each $B_i$, there is an induced \textbf{pushforward} functor $f_*$ on the corresponding categories of representations.
If $V=(V,F_i,\phi_i)$ is a representation of $\mathcal X$, then we produce a representation $f_*V=W=(W,G_i,\psi_i)$ of $\mathcal X'$ with $W=V$ as vector spaces and
\begin{equation}
W_j=\bigoplus_{i\in A/\alpha,f_1(i)=j}V_i,\qquad j\in A'/\alpha'
\end{equation}
\begin{equation}
G_l=\bigoplus_{i\in A,l\in B'_{f_1(i)}}F_k,\qquad \text{where }k=\max\{r\in B_i\mid f_2(r)\leq l\}
\end{equation}
for which we have
\begin{equation}
\mathrm{gr}_jW=\bigoplus_{i\in B,f_2(i)=j}\mathrm{gr}_iV
\end{equation}
so that we can define
\begin{equation}
\psi_j=\bigoplus_{i\in B,f_2(i)=j}\phi_i.
\end{equation}
If $f_2$ fails to be injective on some $B_i$'s, then one can still define $f_*V$ given additional choices on $V$ (i.e. in a non-functorial way).
Namely, for each $i\in A$ and $j\in B'_{f_1(i)}\cap\mathrm{Dom}(\beta')$ with 
\begin{equation}
\left(f_2|_{B_i}\right)^{-1}(j)\cong \{1,\ldots,n\},\qquad n\ge 2
\end{equation}
choose a splitting of the inclusions $F_1V\subset\ldots\subset F_nV$.
This gives us an isomorphism 
\begin{equation}
\mathrm{gr}_1V\oplus\ldots\oplus\mathrm{gr}_nV\cong F_nV/F_{<1}V=\mathrm{gr}_jW
\end{equation}
allowing us to define $\psi_j$ as direct sums of $\psi_1\oplus\ldots\oplus\psi_n$ for various $i$.
Different choices of splittings give isomorphic $f_*V$.

\begin{thm} \label{nets_thm_exist}
Let $\mathcal X$ be a net, $V$ a representation of $\mathcal X$.
Then there is a net $\mathcal X'$ of height 1, a morphism $f:\mathcal X'\to\mathcal X$, and representation $V'$ of $\mathcal X'$ such that $f_*V'=V$.
\end{thm}

As a preliminary to the proof of the theorem, we discuss diagrams
\[
\begin{tikzcd}
X \arrow[bend left=15, two heads]{r}{A} \arrow[bend right=15, two heads]{r}[swap]{B} & Y
\end{tikzcd}
\]
of parallel surjective linear maps between finite dimensional vector spaces.
First, we have two increasing filtrations on $X$:
\begin{align}
A^{-1}0&\subset A^{-1}BA^{-1}0\subset A^{-1}BA^{-1}BA^{-1}0\subset\ldots \subset\underbrace{A^{-1}(BA^{-1})^{k-1}0}_{=:F_k} \subset \ldots \\
B^{-1}0&\subset B^{-1}AB^{-1}0\subset B^{-1}AB^{-1}AB^{-1}0\subset\ldots \subset\underbrace{B^{-1}(AB^{-1})^{k-1}0}_{=:G_k} \subset \ldots 
\end{align}
They necessarily stabilize with $F_\infty=F_N$, $G_\infty=G_N$ for $N$ big.
In general, these filtrations are not exhaustive, but
\begin{equation}
F_\infty=G_\infty.
\end{equation}
Furthermore, there are inverse isomorphisms
\[
\begin{tikzcd}
M_{i,j} \arrow[bend left=5]{r}{B^{-1}A} & M_{i-1,j+1} \arrow[bend left=5]{l}{A^{-1}B}
\end{tikzcd}
\]
for $i\ge 0$, $j\ge 1$, where
\begin{equation}
M_{i,j}=\frac{F_i\cap G_j}{(F_{i-1}\cap G_j)+(F_i\cap G_{j-1})}
\end{equation}
for $i,j\ge 1$.
In particular,
\begin{itemize}
\item $\mathrm{Ker}(B)$ has an exhaustive filtration $F\cap\mathrm{Ker}(B)$ with associated graded pieces $M_{i,1}$,
\item $\mathrm{Ker}(A)$ has an exhaustive filtration $G\cap\mathrm{Ker}(A)$ with associated graded pieces $M_{1,j}$,
\item there are isomorphisms $M_{i,1}\longrightarrow M_{1,i}$.
\end{itemize}

We proceed with the proof of the theorem.

\begin{proof}
Let $(V,\{F_i\},\phi_i)$ be a representation of a net $\mathcal X=(A,\alpha,B,\beta)$.
The proof is by induction over
\begin{equation}
\sum_{|B_i|\ge 2}\dim V_i
\end{equation}
for all nets simultaneously.
If $h(\mathcal X)=1$ we are done, so let us assume that $h(\mathcal X)\ge 2$.
We can also assume, by passing to subsets of $A$ and the $B_i$, that all associated graded $\mathrm{gr}_iV$ are non-zero and all $B_i$ are non-empty.

Let $r\in A$ with $|B_r|=h(\mathcal X)$.
Let $n=\max B_r$ and find the unique $k\in B_{\alpha(r)}$ such that
\begin{equation}
F_{<k}+F_{<n}\subsetneqq V_r,\qquad F_k+F_{<n}=V_r.
\end{equation}
Set $X_1=F_{<k}+F_{<n}$, $X_2=F_k\cap (F_{<k}+F_{<n})$ and consider the refinements of the two filtrations of $V_r$:
\begin{equation}
\ldots\subset F_{<n}\subset X_1\subset F_n=V_r,\qquad\ldots\subset F_{<k}\subset X_2\subset F_k\subset\ldots
\end{equation}
Note that the identity map induces an isomorphism of associated graded
\begin{equation}
F_n/X_1\longrightarrow F_k/X_2
\end{equation}
which follows from the general isomorphism theorem $S/(S\cap T)\cong (S+T)/T$.

\underline{Case $\beta(k)\neq n$.}
\begin{itemize}
\item If $n\in\mathrm{Dom}(\beta)$, $\beta(n)\in B_p$, can refine the filtration on $V_p$ to
\[
\ldots\subset F_{<\beta(n)}\subset X_3 \subset F_{\beta(n)} \subset \ldots
\]
where $X_3$ is the preimage of $\phi_n(X_1/F_{<n})$ under the quotient map $F_{\beta(n)}\twoheadrightarrow\mathrm{gr}_{\beta(n)}X$.
\item Similarly, if $k\in\mathrm{Dom}(\beta)$, $\beta(k)\in B_p$, can refine the filtration on $V_p$ to
\[
\ldots\subset F_{<\beta(k)}\subset X_4 \subset F_{\beta(k)} \subset \ldots
\]
where $X_4$ is the preimage of $\phi_k(X_2/F_{<k})$ under the quotient map $F_{\beta(k)}\twoheadrightarrow\mathrm{gr}_{\beta(k)}X$.
\end{itemize}
We construct a modified net $\mathcal X'=(A',\alpha',B',\beta')$ with
\begin{equation} \label{mod_net_A}
A'=A\sqcup\{1,2\},\qquad \alpha'(1)=2,\qquad\alpha'|_A=\alpha
\end{equation}
and $B'_1=\{1\}$, $B'_2=\{2\}$.
Moreover, if $n\in\mathrm{Dom}\beta$ we insert an element ``3'' into $B$ just before $\beta(n)$ with $\beta'(1)=3$, and similarly if $k\in\mathrm{Dom}\beta$ we insert an element ``4'' before $\beta(k)$ with $\beta'(2)=4$.
Note that the new elements of $B$ correspond to the additional pieces of the filtrations described above.
There is a morphism of nets $f:\mathcal X'\to\mathcal X$ with $1\mapsto n$, $2\mapsto k$, $3\mapsto\beta(n)$, $4\mapsto\beta(k)$.

$V$ determines a representation $V'=(V',F_i,\psi_i)$ of $\mathcal X'$ with
\begin{equation} \label{lifted_rep}
V_1'=V_2'=F_n/X_1\cong F_k/X_2,\qquad V_r'=V_{\alpha(r)}'=X_1
\end{equation}
where we suppress the canonical isomorphism $F_n/X_1\cong F_k/X_2$ from now on.
The filtrations on $V'_r$ are obtained from those on $V_r$ by restriction, and the filtrations containing $F_3,F_4$ by refinement.
The isomorphisms $\psi_n,\psi_k,\psi_1,\psi_2$ and their inverses are induced by $\phi_n$, $\phi_k$ and their inverses respectively.
We claim that $f_*V'$ is isomorphic to $V$.
Choose complements
\begin{align}
X_1/F_{<n}\oplus Y_1&=\mathrm{gr}_nV \\
X_2/F_{<k}\oplus Y_2&=\mathrm{gr}_kV.
\end{align}
If $n\in\mathrm{Dom}\beta$, then the choice of $Y_1$ is equivalent to a choice of complement of $X_3/F_{<\beta(n)}\subset \mathrm{gr}_{\beta(n)}V$, and similarly for $Y_2$. 
We use these complements in the definition of $f_*V'$.

We claim that we can find a $U\subset F_k$ such that
\begin{align}
F_n &\twoheadrightarrow \mathrm{gr}_n V\text{ induces an isomorphism } U\to Y_1 \label{splitLift2} \\
F_k &\twoheadrightarrow \mathrm{gr}_k V\text{ induces an isomorphism } U\to Y_2 \label{splitLift1} 
\end{align}
Namely, choose a $U\subset F_k$ with
\begin{equation}
X_2\oplus U=F_k
\end{equation}
such that \eqref{splitLift1} holds. 
We also have a direct sum $X_1\oplus U=V_r$, but \eqref{splitLift2} may fail.
However, we can ensure \eqref{splitLift2} after shearing $U$ in the directions of $F_{<k}$, which does not affect \eqref{splitLift1}.

The choice of $U$ gives an identification
\begin{equation}
V_r=X_1\oplus U=X_1\oplus (F_n/X_1)=V_r'\oplus V_1'
\end{equation}
which we use to define the isomorphism $V\to f_*V'$.
Compatibility with filtrations follows from $U\subset F_k$, and compatibility with $\phi_n$, $\phi_k$ since they are block-diagonal with respect to the chosen splittings of associated graded.

\underline{Case $\beta(k)=n$.}
We have a diagram of parallel surjections:
\[
\begin{tikzcd}
\mathrm{gr}_nV \arrow[two heads]{r}{p_1} \arrow{d}{\phi_n}[swap]{\cong} & F_n/X_1 \arrow{d}{\cong} \\
\mathrm{gr}_kV \arrow[two heads]{r}{p_2} & F_k/X_2
\end{tikzcd}
\]
By the discussion preceding the proof we get first of all a pair of non-exhaustive filtrations 
\begin{equation} 
0=G_0\subset G_1\subset \ldots \subset G_m, \qquad 0=H_0\subset H_1\subset \ldots \subset H_m=G_m 
\end{equation}
of $\mathrm{gr}_nV$, where $G_1=\mathrm{Ker}(p_1)$ and $H_1=\phi_n^{-1}(\mathrm{Ker}(p_2))$.
For the associated double graded $M_{ij}$, $1\leq i,j$, $i+j\leq m+1$, we have isomorphisms $M_{ij}\to M_{i+1,j-1}$ induced by $\phi_n$.
Further, we have restricted filtrations on $\mathrm{Ker}(p_1)$, $\mathrm{Ker}(p_2)$ lifting to refinements of the filtrations on $V_r$:
\begin{align} \label{refined_filt_2}
\ldots\subset F_{<n}&=X_{1,0}\subset X_{1,1}\subset\ldots\subset X_{1,m}=X_1\subset F_n  \\
\ldots\subset F_{<k}&=X_{2,0}\subset X_{2,1}\subset\ldots\subset X_{2,m}=X_2\subset F_k \subset\ldots 
\end{align}
with
\begin{equation}
X_{1,i}/X_{1,i-1}\cong M_{1,i},\qquad X_{2,i}/X_{2,i-1}\cong M_{i,1}
\end{equation}

We construct a modified net $\mathcal X'=(A',\alpha',B',\beta')$ with 
\begin{gather}
A'=A\sqcup\{1,2\}\sqcup\left\{(l,i,j)\mid l\in \{1,2\},i\ge 2,j\ge 2,i+j\leq m+1\right\} \\
\alpha'(1)=2,\qquad \alpha'(1,i,j)=(2,i,j),\qquad \alpha'|_A=\alpha \\
B'_1=\{1\},\qquad B'_2=\{2\},\qquad B'_{l,i,j}=\{l,i,j\} \\
B'_r=(B_r\setminus\{n\})\sqcup\{(1,1,1)<(1,1,2)<\ldots<(1,1,m)\} \\
B'_{\alpha(r)}=(B_{\alpha(r)}\setminus\{k\})\sqcup\{(2,1,1)<(2,2,1)<\ldots<(2,m,1)\} \\
\beta'(1)=2,\qquad\beta'(1,i,j+1)=(2,i+1,j),\quad i\ge 1,j\ge 1,i+j\leq m
\end{gather}
where in $B'_r$ (resp. $B'_{\alpha(r)}$) the new elements replace $n$ (resp. $k$) in the total order.
There is a morphism of nets $f:\mathcal X'\to\mathcal X$ with $1,(1,i,j)\mapsto n$, and $2,(2,i,j)\mapsto k$.

Let $G$ be the preimage of $G_m=H_m$ under the projection $F_n\to\mathrm{gr}_nV$.
$V$ determines a representation $V'=(V',F_i,\psi_i)$ of $\mathcal X'$ with
\begin{equation}
V'_r=X_1,\qquad V'_1=F_n/G,\qquad V'_{1,i,j}=M_{i,j},
\end{equation}
the restrictions of the filtrations~\eqref{refined_filt_2} to $X_1$, and $\psi_1,\psi_{1,i,j}$ induced by $\phi_n$.
We claim that $f_*V'$ is isomorphic to $V$.
Choose a complement $Y$ to $G\subset F_n$ with $Y\subset F_k$.
Further, let $Y_{1,1},\ldots,Y_{1,m}\subset\mathrm{gr}_nV$ with
\begin{equation}
X_{1,j}/F_{<n}=Y_{1,1}\oplus\ldots\oplus Y_{1,j}
\end{equation}
and define $Y_{i+1,j}=\phi_n(Y_{i,j+1})$.
We get $Y_{ij}\cong M_{ij}$ and
\begin{equation}
G/X_1=\bigoplus_{\substack{i,j\ge 2 \\ i+j<m+1}}Y_{ij}.
\end{equation}
Also we can use the $X_{1,j}$ and $X_{i,1}$ in the definition of $f_*V'$.
Choose a complement $Z$ to $X_1\subset G$ with $Z\subset F_k$, allowing us to lift $Y_{i,j}$ to $Z_{i,j}\subset F_k$.
Combining the various splittings we obtain an isomorphism
\begin{equation}
X_1\oplus(F_n/G)\oplus\bigoplus_{\substack{i,j\ge 2 \\ i+j<m+1}}M_{ij}\to V_r
\end{equation}
which we use to identify $f_*V'$ with $V$.
\end{proof}

Call a net as in \eqref{band_type_net} a \textit{cycle}. 
Note that for any cycle and any $n\ge 2$ there is a morphism from another cycle which is $n:1$, i.e. an $n$-fold ``covering''. 
These, and isomorphisms, are the only morphisms between connected nets of height 1.
We use this to formulate a strengthening of the previous theorem.

\begin{thm} \label{nets_thm_unique}
Let $\mathcal X$ be a net, $V$ an indecomposable representation of $\mathcal X$.
Then there exists a connected net $\mathcal X'$ of height 1, a morphism $f:\mathcal X'\to\mathcal X$ which does not factor through a covering, and an indecomposable representation $V'$ of $\mathcal X'$ such that $f_*V'=V$.
Moreover, for any triple $\mathcal X''$, $f''$, $V''$ with these properties there is an isomorphism of nets $g:\mathcal X'\to\mathcal X''$ over $\mathcal X$ such that $g_*V'\cong V''$.
\end{thm}

\begin{proof}
Existence follows from Theorem~\ref{nets_thm_exist}, which gives us $\mathcal X'$, $f$, $V'$, with $f_*V'=V$.
Since $V$ is indecomposable, $V'$ must be as well, and can only be supported on a single component of $\mathcal X'$ so we can take $\mathcal X'$ to be connected.
If $f$ factors through some covering, we just push $V'$ forward along it.

For uniqueness, consider the inductive procedure in the proof of Theorem~\ref{nets_thm_exist}.
After choosing a total order on $A$, the output $\mathcal X',f,V'$ depends, up to isomorphism, only on the isomorphism class of $V$.
Moreover, if $V$ is of the form $f_*V'$ for some connected net of height 1, $f:\mathcal X'\to X$ not factoring through a covering, and $V'$ indecomposable, then the output of the procedure applied to $V$ is isomorphic to the same $(\mathcal X',f,V')$.
\end{proof}

In applications of this theorem below it will be natural to consider nets which are not finite (only individual $B_i$ are).
Still, all our results extend to this case, as finite dimensional representations are supported on a finite subnet.

\subsection{Minimal twisted complexes}
\label{subsec_minimal_complexes}

Let $\mathcal A$ be a graded linear category which is \textbf{augmented} in the sense that there are splittings
\begin{equation}
\mathcal A(X,Y)=\mathcal A_e(X,Y)\oplus\mathcal A_r(X,Y)
\end{equation}
such that $A_e(X,Y)$ is $\mathbb{K}1_X$ for $X=Y$ and zero otherwise, and $\mathcal A_r(X,Y)$ are closed under composition.
We view $\mathcal A_r$ as a non-unital category with the same objects as $\mathcal A$.
Assume for the rest of the subsection that $\mathcal A_r$ is nilpotent in the sense that there exists an integer $N>0$ such that any composition of $N$ morphisms vanishes.
Recall that a functor $F$ is said to \textit{reflect isomorphisms} if $f$ is invertible whenever $F(f)$ is.

\begin{lem}
The functor $T:\mathrm{add}\mathbb Z(\mathcal A)\to\mathrm{add}\mathbb Z(\mathcal A_e)$ induced by the augmentation reflects isomorphisms.
\end{lem}

\begin{proof}
Let $M,N\in\mathrm{add}\mathbb Z(\mathcal A)$ and $\phi:M\to N$ such that $T(\phi)=\overline{\phi}$ is an isomorphism.
Thus we have graded vector spaces $M(X),N(X)$ for each $X\in\mathrm{Ob}(\mathcal A)$, and components of $\phi$
\begin{align*}
\phi_{X}\in &\mathrm{Hom}(M(X),N(X))\oplus\left(\mathrm{Hom}(M(X),N(X))\otimes\mathcal A_r(X,X)\right) \\
&\phi_{X,Y}\in\mathrm{Hom}(M(X),N(Y))\otimes\mathcal A_r(X,Y),\qquad X\neq Y
\end{align*}
Let $\overline{\phi}_X$ be the component of $\phi_X$ in $\mathrm{Hom}(M(X),N(X))$.
By assumption, the $\overline{\phi}_X$ are isomorphisms.
Composing $\phi$ with the morphism $\overline{\phi}^{-1}$ with components $\overline{\phi}_X^{-1}$, we may assume that $M(X)=N(X)$ and all $\overline{\phi}_X$ are identity morphisms.
Thus,
\begin{equation}
\phi=1-\epsilon\in \mathrm{End}^0(M),\qquad \epsilon^N=0
\end{equation}
which clearly has an inverse.
\end{proof}

A two-sided twisted complex over $\mathcal A$ is given by a pair $(M,\delta)$ with $M\in\mathrm{Ob}(\mathrm{add}\mathbb{Z}\mathcal A)$ and $\delta\in\mathrm{End}^1(M)$ with $\delta^2=0$.
We say that $(M,\delta)$ is \textbf{minimal} if the image $\overline{\delta}$ of $\delta$ under the functor $\mathrm{add}\mathbb Z\mathcal A\to\mathrm{add}\mathbb Z\mathcal A_e$ vanishes, i.e. $\delta$ has components in (the degree $1$ part of)
\begin{equation}
\mathrm{Hom}(M(X),M(Y))\otimes\mathcal A_r(X,Y).
\end{equation} 

\begin{prop} \label{prop_min_complexes}
\begin{enumerate}
\item Every twisted complex $A\in\mathrm{Tw}\mathcal A$ is isomorphic to a direct sum $A=A_m\oplus A_c$ with $A_m$ minimal and $A_c$ contractible.
\item Any homotopy equivalence between minimal twisted complexes is an isomorphism.
\end{enumerate}
\end{prop}

\begin{proof}
1. Let $(M,\delta)$ be a twisted complex over $\mathcal A$, $(M,\overline{\delta})$ its image in $\mathrm{Tw}(\mathcal A_e)$.
By semisimplicity of $\mathcal A_e$, $(M,\overline{\delta})$ is isomorphic to a direct sum $B_c\oplus B_m$ where $B_c$ is contractible and $B_m$ has trivial differential.
Therefore, $(M,\delta)$ is isomorphic to a twisted complex of the form
\[
\left(K\oplus K[-1]\oplus L,\begin{pmatrix} \delta_{11} & \delta_{12} & \delta_{13} \\ 1_K+\delta_{21} & \delta_{22} & \delta_{23} \\ \delta_{31} & \delta_{32} & \delta_{33} \end{pmatrix} \right)
\]
with $\overline{\delta}_{ij}=0$.
Using elementary row and column operations (automorphisms of $K\oplus K[-1]\oplus L$) one reduces the matrix to the form
\[
\begin{pmatrix}
0 & 0 & 0 \\ 1_K & 0 & 0 \\ 0 & 0 & \delta_L
\end{pmatrix}
\]
with $\overline{\delta}_L=0$, thus providing the direct sum decomposition.

2. Suppose $(M,\delta^M)$ and $(N,\delta^N)$ are minimal twisted complexes, $f:M\to N$ and $g:N\to M$ closed morphisms of degree $0$, $G:M\to M$ and $H:N\to N$ morphisms of degree $-1$ such that
\begin{equation}
1_M-gf=\delta^MG+G\delta^M,\qquad 1_N-fg=\delta^NH+H\delta^N.
\end{equation}
Using minimality, $\overline{gf}=1_M$ and $\overline{fg}=1_N$, i.e. $\overline{f},\overline{g}$ are inverse isomorphisms.
Hence, the claim follows from the previous lemma.
\end{proof}

\subsection{Classification of objects}
\label{subsec_geometricity}

We are now ready to state and prove the main result of this section.

\begin{thm} \label{thm_classification}
Let $S$ be a compact graded marked surface, then the construction of Subsection~\ref{subsec_complexes_from_curves} sets up a bijection between isomorphism classes of indecomposable objects in $H^0(\mathcal F(S))$ and isotopy classes of admissible curves with indecomposable local system.
\end{thm}

\begin{proof}
Let $Q$ be a graded quiver with quadratic monomial relations of type F1, i.e. associated with some systems of arcs $A$ on a compact graded marked surface $(S,M,\Omega)$. 
The graded linear category $\mathbb{K}Q$ has a natural augmentation with $(\mathbb{K}Q)_r$ generated by non-identity paths.
Assume that sufficiently long paths are zero in $\mathbb{K}Q$, equivalently that no components of $M$ are diffeomorphic to $S^1$, so that $(\mathbb{K}Q)_r$ is nilpotent and we can apply the results of the previous subsection.

Consider first the case when $Q$ is of the form
\begin{equation}
\bullet\longleftarrow\bullet\longleftarrow\cdots\longleftarrow\bullet
\end{equation}
without relations and arbitrary grading.
Let us number the vertices from left to right by $\{1,\ldots,n\}$ and let $\alpha_i$ denote the arrow from $i+1$ to $i$.
A minimal twisted complex over $\mathbb{K}Q$ is then given by finite-dimensional graded vector spaces $V_1,\ldots,V_n$, and for $1\leq i<j\leq n$ a linear map
\begin{equation}
\delta_{ij}:V_j\to V_i,\qquad \deg(\delta_{ij})=\deg(\alpha_i)+\ldots+\deg(\alpha_{j-1})-1
\end{equation}
such that the matrix $\delta=(\delta_{ij})$ has square zero.

Consider the direct sum
\begin{equation}
V=\bigoplus_{i=1}^nV_i.
\end{equation}
which has the filtration
\begin{equation}\label{first_filt}
V_1\subset V_1\oplus V_2\subset\ldots\subset V_1\oplus\ldots\oplus V_n.
\end{equation}
The endomorphism $\delta$ gives a three-step filtration
\begin{equation}
\mathrm{Im}\delta\subset\mathrm{Ker}\delta\subset V
\end{equation}
and an isomorphism of associated graded $V/\mathrm{Ker}\delta\cong\mathrm{Im}\delta$.
We find that we have a representation of the net
\begin{equation}
\left(\{1,2\}\times\mathbb{Z},(12)\times\mathrm{id},(\{1,\ldots,n\}\sqcup\{1,2,3\})\times\mathbb{Z},(13)\times\mathrm{id}\right)
\end{equation}
with total order so that $(i,k+d(\alpha_i)-1)<(i+1,k)$ on $\{1,\ldots,n\}\times\mathbb Z$ and $(1,n)<(2,n)<(3,n)$ on $\{1,2,3\}\times\mathbb Z$.
In fact, we get an embedding of the category of minimal twisted complexes into the category of representations of the above net as a full subcategory.
We do not get an equivalence of categories, since we only consider $\delta$ which decrease the filtration \eqref{first_filt}. 

Turn now to the case of general $Q$.
Let $D$ be the set of maximal non-zero paths in $\mathbb{K}Q$. 
Then for each element of $D$ there is a sub-quiver of $Q$ which is of the simple form above, and $Q$ is obtained from their disjoint union by identifying some pairs of vertices.
The corresponding net is now given by
\begin{equation}
A=D\times\{1,2\}\times\mathbb Z \\
\end{equation}
\begin{equation}
\alpha=\mathrm{id}\times(12)\times\mathrm{id} \\
\end{equation}
\begin{equation}
B=\left(\left\{(v,d)\in Q_0\times D\mid v\text{ on }d\right\}\sqcup \left(D\times\{1,2,3\}\right)\right)\times\mathbb Z=:B_g\sqcup B_a \\
\end{equation}
with partial order defined so that for an arrow $\alpha$ in $d\in D$ we have $(\partial_1(\alpha),n+d(\alpha)-1)<(\partial_0(\alpha),n)$ and $(d,1,n)<(d,2,n)<(d,3,n)$.
Further,
\begin{equation}
\beta=(\tau\sqcup(\mathrm{id}\times(13))\times\mathrm{id})
\end{equation} 
where $\tau$ sends a pair $(v,d_1)\in Q_0\times D$ with $v$ on $d_1,d_2\in D$, $d_1\neq d_2$ to $(v,d_2)$.
There is a fully-faithful functor from the category of minimal twisted complexes over $Q$ to the category of representations of $(A,\alpha,B,\beta)$.
Theorem~\ref{nets_thm_unique} gives us a classification of indecomposable minimal twisted complexes over $\mathbb KQ$, equivalently isomorphism classes in $H^0(\mathcal F(S))$ by Proposition~\ref{prop_min_complexes}.
Following the definitions, one identifies isotopy classes of admissible curves with local system with morphisms from nets of height 1 with indecomposable representation.
Of course, nets of the type \eqref{string_type_net} (resp. \eqref{band_type_net}) correspond to admissible curves with domain $[0,1]$ (resp. $S^1$).

It remains to deal with the case when $M$ has components diffeomorphic to $S^1$.
Let $S'$ be the surface obtained by adding two corners  and a boundary arc on each such component of $M$.
Then $H^0(\mathcal F(S))$ is a localization of the triangulated category $H^0(\mathcal F(S'))$, and the localization functor is essentially surjective.
Moreover, if the images of $E,F\in H^0(\mathcal F(S))$ corresponding to admissible curves become isomorphic under localization, then $E$ and $F$ differ by extensions by boundary arcs, so map to isotopic curves in $S$.
\end{proof}


\section{Comparison of Moduli spaces}
\label{sec_stability}

This section contains our main result, identifying $\mathcal M(S)$ with an open and closed subset of $\mathrm{Stab}(\mathcal F(S))$.
We begin by reviewing Bridgeland's axioms in Subsection~\ref{subsec_bridgeland}.
In Subsection~\ref{subsec_k0} we compute $K_0(\mathcal F(S))$ in terms of singular homology.
The main theorems are stated in Subsection~\ref{subsec_main} and proven in the remaining subsections.
Subsection~\ref{subsec_support} and \ref{subsec_hn} deal with the support property and Harder--Narasimham filtrations, respectively.
In the final Subsection~\ref{subsec_modulimap} we prove that the map on moduli spaces is complex bianalytic onto an open and closed subset.

\subsection{Stability structures on categories}
\label{subsec_bridgeland}

Fix a triangulated category $\mathcal C$ and a homomorphism
$\mathrm{cl}:K_0(\mathcal C)\to\Gamma$ to a finitely generated
abelian group $\Gamma$.
A \textbf{stability structure} (c.f. \cite{bridgeland07}, \cite{ks}, also called a \textit{stability condition}) on $\mathcal C$ is given by
\begin{itemize}
  \item for each $\phi\in\mathbb R$ a full additive subcategory $\mathcal C^{\phi}\subset\mathcal C$ of \textbf{semistable objects of phase $\phi$}, and
  \item an additive map $Z:\Gamma\to\mathbb C$, the \textbf{central charge}.
\end{itemize}
This data has to satisfy the following axioms.
\begin{enumerate}
  \item $\mathcal C^{\phi}[1]=\mathcal C^{\phi+1}$
  \item If $E\in\mathcal C^{\phi_1}$, $F\in\mathcal C^{\phi_2}$, $\phi_1>\phi_2$, then $\mathrm{Hom}(E,F)=0$.
  \item Every $E\in\mathcal C$ has a \textbf{Harder--Narasimhan} filtration:
  A tower of triangles
  \begin{equation*} \begin{tikzcd}
  0=E_0 \arrow{r} & E_1 \arrow{d}          & \cdots & E_{n-1} \arrow{r}  & E_n=E
  \arrow{d} \\
                  & A_1 \arrow[dotted]{ul} &        &                    & A_n
  \arrow[dotted]{ul}
  \end{tikzcd} \end{equation*}
  with $0\neq A_i\in\mathcal C^{\phi_i}$ and
  $\phi_1>\phi_2>\ldots>\phi_n$.
  The $A_i$ are called the \textbf{semistable components} of $E$.
  \item If $0\neq E\in\mathcal C^{\phi}$ then $Z(E):=Z(\mathrm{cl}([E]))\in\mathbb{R}_{>0}e^{\pi i\phi}$.
  \item The \textbf{support property}: For some norm $\|.\|$ on
  $\Gamma\otimes\mathbb R$ and $C>0$ we have an estimate
  \begin{equation}
  \|\mathrm{cl}(E)\|\leq C|Z(E)|
  \end{equation}
  for $E\in\mathcal C^{\phi}$.
\end{enumerate}
The set $\mathrm{Stab}(\mathcal C,\Gamma)$ of all stability structures has a natural topology which is induced by the metric
\begin{equation}
d(\sigma_1,\sigma_2)=\sup_{0\neq E\in\mathcal C}\left\{|\phi^-_{\sigma_2}(E)-\phi^-_{\sigma_1}(E)|,|\phi^+_{\sigma_2}(E)-\phi^+_{\sigma_1}(E)|,\left|\log\frac{m_{\sigma_2}(E)}{m_{\sigma_1}(E)}\right|\right\}
\end{equation}
c.f. \cite{bridgeland07}.
Here, for an object $E\in\mathcal C$ with semistable components $A_1,\ldots,A_n$ with phases $\phi_1>\ldots>\phi_n$ one defines $m(E)=\sum|Z(A_i)|$, $\phi^-(E)=\phi_n$, and $\phi^+(E)=\phi_1$.
The main result of \cite{bridgeland07} (see also \cite{ks}) is that the projection 
\begin{equation*}
\mathrm{Stab}(\mathcal C,\Gamma)\to\mathrm{Hom}(\Gamma,\mathbb C)
\end{equation*}
is a local homeomorphism, hence $\mathrm{Stab}(\mathcal C,\Gamma)$ has the structure of a complex manifold of dimension $\mathrm{rk}(\Gamma)$, if it is non-empty.

\subsection{Charge lattice}
\label{subsec_k0}

Recall that a graded surface has a double cover $\tau$ given by the orientations of the foliation lines, and $\mathbb Z_\tau=\mathbb Z\otimes_{\mathbb Z/2}\tau$ is the associated local system of abelian groups.
The following theorem describes $K_0$ of the Fukaya category in terms of homology with coefficients in $\mathbb Z_\tau$.
\begin{thm} \label{thm_k0}
Let $S=(S,M,\eta)$ be a compact graded marked surface, then there is a natural isomorphism of abelian groups
\[
K_0(\mathcal{F}(S)) \cong H_1(S,M;\mathbb{Z}_\tau).
\]
\end{thm}

\begin{proof}
Choose a full system $A$ of graded arcs on $S$.
Each arc defines a class in $H=H_1(S,M;\mathbb{Z}\sqrt{\Omega})$.
From cellular homology it follows that $H$ is the group generated by these arcs and with relations of the form $\pm X_1\pm\ldots\pm X_n=0$ where $X_1,\ldots,X_n$ are the arcs bounding a disk cut out by $A$.
Correspondingly, it is also clear that the arcs in $A$ generate 
\begin{equation}
K:=K_0(\mathcal F(S))=K_0(\mathrm{Tw}(\mathcal F_A(S)))
\end{equation}
and satisfy the same relations, which follows from the observation in Subsection~\ref{subsec_morita_invariance}.
By construction we get a surjective homomorphism $H\to K$ which is independent of the choice of arcs and natural with respect to maps of graded marked surfaces.
It remains to show that no additional relations are needed to present $K$.

Suppose first that $M$ has no components diffeomorphic to $S^1$.
Choosing a formal collection $A$ of arcs, we can present $\mathcal F(S)$ as the category $\mathrm{Tw}(\mathbb KQ)$ of twisted complex over the path category of a graded quiver $Q$ with quadratic monomial relations.
Note that $H$ is freely generated by $A$, so we need to verify that $K$ is freely generated by the vertices $Q_0$ of $Q$.
One way to see this is by using the explicit resolution of the diagonal of $\mathbb KQ$, showing that, as a bimodule, $\mathbb KQ$ is a repeated cone of Yoneda bimodules.
As a consequence, $\mathrm{Tw}(\mathbb KQ)\cong\mathrm{Perf}(\mathbb KQ)$, and an inverse of the map 
\begin{equation}
\mathbb Z^{Q_0}\to K_0(\mathrm{Perf}(\mathbb KQ))
\end{equation}
sending a vertex to the corresponding simple module, is induced by the dimension vector of a module.

For a general graded marked surface $S$ there is a $S'$ obtained by adding two corners to each component of $\partial S$ diffeomorphic to $S^1$.
Let $N$ denote set of boundary arcs of $S'$ created in this process, and $\mathcal N$ the corresponding full triangulated subcategory of $H^0(\mathcal F(S'))$ generated by $N$.
Each arc in $N$ has endomorphism algebra of the form $\mathbb K[x]/x^2$ with $|x|\in\mathbb Z$, so $\mathcal N$ is a product of categories of the form $H^0(\mathrm{Tw}(\mathbb K[x]/x^2))$.
We get a commutative diagram
\begin{equation}
\begin{tikzcd}
\mathbb Z^N \arrow{r}\arrow{d} & H_1(S',M',\mathbb Z\sqrt{\Omega'})\arrow{r}\arrow{d} & H_1(S,M,\mathbb Z\sqrt{\Omega})\arrow{r}\arrow{d} & 0 \\
K_0(\mathcal N)\arrow{r} & K_0(\mathcal F(S'))\arrow{r} & K_0(\mathcal F(S)) \arrow{r} & 0 \\
\end{tikzcd}
\end{equation}
where we use Proposition~\ref{prop_loc} to get the bottom row.
We claim that the rows are exact.
For the top row this is just the exact sequence of a triple. 
For the bottom row this uses the fact that $\mathcal N$ is idempotent complete, thus a thick subcategory so that Proposition~3.1 in \cite{sga5_8} can be applied.
Since the left and middle horizontal maps are isomorphisms, so is the right one.
\end{proof}

When a flat surface $S$ comes from a quadratic differential with higher order poles, the category $\mathcal F(S)$ includes some objects with have infinite length. 
We need to pass to the subcategory of $\mathcal F(S)$ which does not include these objects.
Formally, we want to allow $S=(S,M)$ to have boundary circles which are unmarked, i.e. not belonging to $M$.
Thus, define $\mathcal F(S,M)$ as the full subcategory of $\mathcal F(S,M')$ of objects corresponding to immersed curves avoiding $M'\setminus M$, where $M'$ is the union of $M$ and all boundary circles.

\begin{lem} \label{lem_boundaryfunctors}
The full subcategory $\mathcal F(S,M)\subset\mathcal F(S,M')$ is triangulated, i.e. closed under shifts and cones.
\end{lem}

\begin{proof}
For each component $B$ of the boundary we will introduce a functor $I_B$ from $\mathcal F(S,M')$ to the category of complexes of vector spaces (possibly unbounded), such that $E\in\mathcal F(S,M)$ if and only if $I_B(E)=0$ for all components $B$ of $M'$ which are not in $M$, which implies the claim.

Fix a boundary component $B\subset\partial S$. 
The grading and vectors tangent to the boundary define sections $s_1,s_2$ of $\mathbb P(TS|_B)$. 
If $B\cong S^1$, then since $S$ is oriented, $B$ and $\mathbb P(TS|_B)$ carry natural orientations, and so the intersection number, $m$, between $s_1$ and $s_2$ is well defined.
If $B\cong \mathbb R$ we take $m=0$.
Let $N=|m|$. 
If $N=0$, then $B$ has a grading in the sense of Subsection~\ref{subsec_grading}, and if $N>0$ then $B$ still has a $\mathbb Z/N$-grading: A lift of $s_2$ to the fiberwise $N$-fold covering of $\mathbb P(TS|_B)$. 
In either case, choose an appropriate grading on $B$.
Then for every graded curve $A$ ending on $B$ transversely, we get an intersection number $i(B,A)\in\mathbb Z/N$.

The cohomology of $I_B(E)$ will be bounded if $N=0$ and $N$-periodic if $N>0$.
It suffices to define $I_B$ on generators of $\mathcal F(S,M')$.
Let $A$ be an arc in $S$.
If both endpoints of $A$ do not lie on $B$, then $I_B(A)=0$.
If one endpoint of $A$ lies on $B$, then 
\begin{equation}
I_B(A)=\bigoplus_{\substack{j\in\mathbb Z \\ j\equiv i(B,A)}}\mathbb K[-j].
\end{equation}
Finally, if both endpoints of $A$ lie on $B$, then $I_B(A)$ is a sum of two graded vector spaces as above.
The map on morphism is defined in the obvious way so that $I_B(a)=0$ if $a$ is a boundary path which is not in $B$.
One checks that this gives a well defined $A_\infty$-functor.
\end{proof}

\begin{prop} \label{prop_mapfromk0}
Let $S$ be a compact graded marked surface which is allowed to have unmarked boundary circles.
Then there is a natural map
\begin{equation}
K_0(\mathcal F(S))\longrightarrow H_1(S,M;\mathbb Z_\tau).
\end{equation}
\end{prop}

The group $K_0(\mathcal F(S))$ is often infinitely generated, and so the above map will fail to be injective.
It may also fail to be surjective, e.g. for the cylinder with non-standard grading, where $\mathcal F(S)$ is trivial if both boundary circles are unmarked.

\begin{proof}
Let $S'$ be the graded marked surface obtained from $S$ by replacing each unmarked boundary circle with one marked and one unmarked boundary arc.
Then $\mathcal F(S)$ is also a subcategory of $\mathcal F(S')$, as follows from localization.
Note that 
\begin{equation}
\iota:H_1(S,M;\mathbb Z_\tau)\longrightarrow H_1(S',M';\mathbb Z_\tau)\cong K_0(\mathcal F(S'))
\end{equation}
is an inclusion and the image of $\iota$ contains the image of $K_0(\mathcal F(S))\longrightarrow K_0(\mathcal F(S'))$.
Hence we have the desired natural map.
\end{proof}

\subsection{Statement of main theorem}
\label{subsec_main}

Let $\widehat{X}$ be the real blow-up of a flat surface $X$ of finite type.
It contains a compact surface with marked boundary $S\subset \widehat{X}$, canonical up to isotopy, as a deformation retract.
The complement $\widehat{X}\setminus S$ is a disjoint union of puncture open disks, one for each higher order pole of the quadratic differential, and rectangles $\mathbb R_{>0}\times [0,1]$, one for each infinite-angle singularity (more precisely, non-cyclic boundary walk).
The marked boundary of $S$ is the part of $\partial S$ in $\partial\widehat{X}$.
Each removed punctured open disk is bounded by an unmarked boundary circle, while each removed rectangle is bounded on one end by an unmarked part of $\partial S$.
Thus, the Fukaya category $\mathcal F(X)=\mathcal F(S)$ is defined.
We also note that the inclusion of pairs $(S,M)\subset (\widehat{X},\partial\widehat{X})$ is a homotopy equivalence, so both give isomorphic relative homologies.

\begin{thm} \label{thm_stability}
Let $X$ be a flat surface of finite type.
Define subcategories $\mathcal C^{\phi}\subset H^0(\mathcal F(X))$ so that isomorphism classes of indecomposable objects correspond to unbroken graded geodesics of phase $\phi$ with indecomposable local system, and let 
\begin{equation}
Z:K_0(\mathcal F(S))\longrightarrow H_1(X,\partial X;\mathbb Z_\tau)=:\Gamma\to\mathbb C
\end{equation}
be the period map.
Then this data satisfies the axioms of a stability structure.
In particular, for a graded surface of finite type $X$ there is a continuous map
\begin{equation}
\mathcal M(X)\longrightarrow\mathrm{Stab}(\mathcal F(X))
\end{equation}
over $\mathrm{Hom}(\Gamma,\mathbb C)$.
\end{thm}

\begin{rem}
Let us emphasize that we put no restrictions on the ground field $\mathbb K$ of $\mathcal F(S)$.
In fact this category is defined over any ring, but already $\mathrm{Perf}(\mathbb Z)$ does not admit stability structures.
\end{rem}

\begin{thm}[Main Theorem] \label{ModuliMapThm}
Let $X$ be a graded surface of finite type, then the map $\mathcal M(X)\to\mathrm{Stab}(\mathcal F(X))$ is injective and its image open and closed, thus a union of components of $\mathrm{Stab}(\mathcal F(X))$.
\end{thm}

\begin{rem}
When $C$ is compact, then necessarily $g(C)=1$.
In this case one can use homological mirror symmetry for elliptic curves \cite{polishchuk_zaslow} and Atiyah's classification \cite{atiyah57} to derive a version of Theorem~\ref{thm_stability} and Theorem~\ref{thm_stability}.
The computation of $\mathrm{Stab}(D^b(E))$ in \cite{bridgeland07} shows that $\mathcal M(X)=\mathrm{Stab}(X)$.
\end{rem}
The proofs of Theorem~\ref{thm_stability} and Theorem~\ref{ModuliMapThm} will occupy the rest of this section.

\subsection{Support property}
\label{subsec_support}

Let $X$ be a flat surface of finite type and $Y=\widehat{X}$ its real blow-up.
We have as before
\begin{equation}
\Gamma=H_1(Y,\partial Y;\mathbb Z_\tau),\qquad\Gamma_{\mathbb R}=\Gamma\otimes_\mathbb{Z}\mathbb R=H_1(Y,\partial Y;\mathbb R_\tau)
\end{equation}

\begin{prop}
There is a norm $\|.\|$ on $\Gamma_{\mathbb R}$, $C>0$ such that for $\gamma\in\Gamma$ the class of a finite geodesic on $X$ there is an estimate
\begin{equation}
\|\gamma\|\leq C|Z(\gamma)|
\end{equation}
where $Z:\Gamma\to\mathbb C$ is the period map.
\end{prop}

\begin{proof}
The locally constant sheaf $\mathbb R_\tau$ is given by flat sections of a flat line bundle with metric for which we use the same notation.
We have
\begin{align}
\Gamma_{\mathbb R}^*&= H_1(Y,\partial Y;\mathbb R_\tau)^* \\
&= H^1(Y,\partial Y;\mathbb R_\tau) \\
&= H^1_{dR}(Y,\partial Y;\mathbb R_\tau)
\end{align}
Choose forms $\omega_1,\ldots,\omega_n$ representing a basis of $H^1_{dR}(Y,\partial Y;\mathbb R_\tau)$.
Let $\widehat{K}$ be the closure of the preimage of $K=\mathrm{Core}(X)$ under the inclusion $Y\setminus \partial Y\to X$.
Define a norm on $\Gamma_{\mathbb R}$ by
\begin{equation}
\|\gamma\|=\sum_i|[\omega_i](\gamma)|
\end{equation}
and set
\begin{equation} \label{supp_prop_const}
C=\sum_i\|\omega_i{|}_{\widehat{K}}\|_\infty
\end{equation}
where the norm is taken with respect to the flat metric.
Let $\alpha$ be a finite geodesic on $X$ with class $\gamma\in\Gamma$, then
\begin{equation}
l(\alpha)=|Z(\gamma)|
\end{equation}
where $l(\alpha)$ is the length.
Thus
\begin{align}
\|\gamma\| &= \sum_i|[\omega_i](\gamma)| \\
           &= \sum_i\left|\int_\alpha\omega_i\right| \\
           &\leq \sum_i l(\alpha)\|\omega_i|_{\widehat{K}}\|_{\infty} \\
           &= C|Z(\gamma)|.
\end{align}
\end{proof}

\subsection{Harder--Narasimhan filtrations}
\label{subsec_hn}

We continue with the proof of Theorem~\ref{thm_stability}.
The strategy is to first deal with the special case of a disk and annulus, then deduce the general case by passing to the universal and annular coverings.

Let $X$ be a flat surface with $X_\mathrm{sm}$ diffeomorphic to a disk and all conical singularities with infinite angle.
The corresponding graded marked surface $S$ is homeomorphic to a closed disk with $n+1=|X_\mathrm{sg}|$ marked and unmarked boundary arcs.
Consider the data of a stability structure on $\mathcal F(X)$ as in the statement of theorem.
We first check that $\mathrm{Hom}(A,B)=0$ for semistable objects $A,B\neq 0$ with $\phi(A)>\phi(B)$. 
Note that, by the long exact sequence for $\mathrm{Hom}$, it suffices to do this for $A,B$ stable, so that they correspond to graded arcs.
We consider the various cases.
\begin{enumerate}
\item If $A$ and $B$ differ only by a grading shift, the claim follows from $\mathrm{Ext}^{<0}(A,A)=0$.
\item If $A$ and $B$ are disjoint, then $\mathrm{Ext}(A,B)=0$.
\item If $A$ and $B$ intersect in a smooth point $p$ of the surface, then $\mathrm{Ext}(A,B)$ is concentrated in degree 
\begin{equation}
i_p(A,B)=\lceil\phi(A)-\phi(B)\rceil
\end{equation}
so $\phi(A)>\phi(B)$ implies $\mathrm{Ext}^0(A,B)=0$.
\item If $A$ and $B$ meet in an $\infty$-angle singularity $p$ in an angle $\phi>0$, then
\begin{equation}
i_p(A,B)=\phi(A)-\phi(B)+\frac{\phi}{\pi}
\end{equation}
and $\mathrm{Ext}(A,B)$ is either concentrated in degree $i_p(A,B)$ or zero, depending on the order in which they meet $p$.
\end{enumerate}

Next we check the HN-property. 
It suffices to do this for indecomposable objects $E$.
By the classification, $E$ corresponds to a graded arc $\alpha$ connecting two $\infty$-angle singularities.
A geodesic representative of $\alpha$ is a concatenation of graded saddle connections $\alpha_1,\ldots,\alpha_k$ corresponding to objects $A_1,\ldots,A_k$ in $\mathcal F(X)$.
Let $\phi_1>\ldots>\phi_l$ be the distinct phases of the $\alpha_i$ (generically, $l=k$).
For $1\leq i\leq l$ there is a semistable $B_i$ given by a twisted complex
\begin{equation}
B_i=\left(\bigoplus_{\phi(A_j)=\phi_i} A_j,\delta_i\right)
\end{equation}
where $\delta_i$ has non-zero coefficients corresponding to the singular points where the geodesic representative of $\alpha$ passes through without changing phase, which is equal to $\phi_i$.
Also,
\begin{equation}
E=\left(\bigoplus_j A_j, \delta\right)
\end{equation}
where $\delta$ has non-zero coefficients corresponding to all the singular points the geodesic representative passes through (i.e. where $\alpha_i$ meets $\alpha_{i+1}$).
Note that this really gives morphisms of degree 1 since the concatenation of the saddle connections smooths to a graded curve.
To show that the twisted complex representation of $E$ in terms of the $B_i$ is a HN-tower, we need to check that components of $\delta$ increase phase.
For this we use that fact that the concatenation of the $\alpha_i$ is a geodesic, and so $\alpha_i,\alpha_{i+1}$ necessarily meet at a singular point in an angle $\phi\ge\pi$, thus $\phi>\pi$ if the phase jumps.
The corresponding morphism, without loss of generality from $A_i$ to $A_{i+1}$, has degree 1, and so
\begin{equation}
1=\phi(A_i)-\phi(A_{i+1})+\frac{\phi}{\pi}
\end{equation}
which implies $\phi(A_i)<\phi(A_{i+1})$.

Suppose now that $X$ is a general flat surface of finite type.
The first claim is that $\mathrm{Hom}(A,B)=0$ for a pair of graded saddle connections $A,B$ with $\phi(A)>\phi(B)$.
To see this, let $Y$ be the flat surface so that its real blow-up $\widehat{Y}$ is the universal cover of $\widehat{X}$.
Note that all singularities of $Y$ have infinite cone angle.
Lift $A,B$ to graded curves $\widetilde{A},\widetilde{B}$ on $Y$ with the same phases.
We have
\begin{equation}
\mathrm{Hom}(A,B)=\bigoplus_{g\in\pi_1(\widehat{X})}\mathrm{Hom}(\widetilde{A},g\widetilde{B})
\end{equation}
and hence it suffices to prove the claim for $Y$.
But any two saddle connections in $Y$ are objects in some subcategory of $\mathcal F(Y)$ which is the Fukaya category of some surface of disk-type as above, and so our previous arguments can be applied.

Second, we claim that the HN-property holds for admissible curves $c$ with domain $[0,1]$.
To see this, lift $c$ to a graded curve $\tilde{c}$ in the universal cover $Y$.
We work in the category generated by all the saddle connections in a geodesic representative of $\tilde{c}$.
It corresponds to some surface of disk-type in the universal cover.
We get a HN-tower for $\tilde{c}$ in this category and push it forward to $\mathcal F(X)$ to get a tower for $c$ in that category.

We turn to the case when $X$ is of annular type, i.e. $X_\mathrm{sm}$ is diffeomorphic to an annulus with grading such that a simple closed loop around it is gradable, and all singularities have infinite cone angle.
Let $S\subset\widehat{X}$ be the corresponding compact graded marked surface.
The category $\mathcal F(X)$ is the bounded derived category of finite-dimensional representations of a quiver which is a cyclic chain or arrows, all in degree zero.
The number of arrows oriented in one or the other way is equal to the number of marked boundary arcs on the two components of $\partial S$, respectively.
In arguments below we omit straightforward computations of $\mathrm{Ext}^*(A,B)$ in $\mathcal F(X)$.

First we check that $\mathrm{Hom}(A,B)=0$ for stable $A,B$ with $\phi(A)>\phi(B)$.
It remains to consider the cases when either $A$ or $B$, or both, are closed geodesics with simple local system.
\begin{enumerate}
\item If $A$ and $B$ both correspond to loops with simple local system, then they either differ by a shift or else are orthogonal, i.e. $\mathrm{Ext}^*(A,B)=\mathrm{Ext}^*(B,A)=0$. 
In the first case it suffices to note that $\mathrm{Ext}^{<0}(A,A)=0$ since $A$ lies in an abelian category. 
\item If one of $A$ or $B$ is a geodesic loop and the other a saddle connection tending in both directions to the same component of $\partial S$, then $A$ and $B$ are orthogonal.
\item If one of $A$ or $B$ is a geodesic loop and the other a saddle connection from one component of $\partial S$ to the other, then $A$ and $B$ intersect in a unique point $p$ and $\mathrm{Ext}^*(A,B)$ is concentrated in degree $i_p(A,B)=\lceil \phi(A)-\phi(B)\rceil>0$.
\end{enumerate}
By passing to the annular covering (of $\widehat X$) we see, as in the case of two saddle connections, that the statement is true for general $X$ when at least one of $A$ or $B$ is a saddle connection.

For annular $X$ there is, up to isotopy and grading, a unique admissible curve $c$ in $X$ with domain $S^1$.
We need to check that HN-filtrations exist for the corresponding indecomposable objects $X$.
If $X$ has a geodesic loop (necessarily isotopic to $c$), then $X$ is semistable, so we are done.
If not, then $c$ has a geodesic representative which is a concatenation of saddle connections.
The corresponding collection of arcs satisfies the properties we required when assigning a twisted complex to $c$ in Subsection~\ref{subsec_twisted_complexes}.
Note that in this case $X$ cannot contain a semi-infinite flat cylinder, so $S$ has no unmarked boundary circles.
The same argument as in the disk case shows that this twisted complex gives a HN-filtration for $X$.
Again, passing to the annular covering, we find that HN-filtration exist for objects corresponding to admissible curves with local system and domain $S^1$ on general $X$.
Since all isomorphism classes of indecomposable objects in $H^0(\mathcal F(X))$ come from admissible curves, this completes the proof of the HN-property.

Let $X$ now be a general flat surface of finite type. 
It remains to show that for two objects $A,B\in\mathcal F(X)$ corresponding to closed geodesics $\alpha,\beta$ with local systems and $\phi(A)>\phi(B)$ we have $\mathrm{Hom}(A,B)=0$.
We may also assume that $\alpha$ and $\beta$ are non-isotopic as simple closed curves.
Let $Z$ denote the flat cylinder foliated by closed geodesics isotopic to $\alpha$.
Since $\Omega$ is assumed to not have any second-order poles, $Z$ has finite height with singular points on each boundary component, see Figure~\ref{fig_closed_geodesics}.
\begin{figure}[h]
\centering
\includegraphics[scale=1]{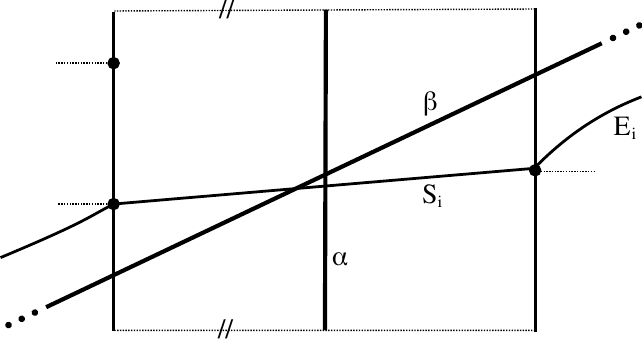}
\caption{The cylinder $Z$, cut along a dotted line, closed geodesics $\alpha,\beta$, saddle connections $S_i$, arcs $E_i$.}
\label{fig_closed_geodesics}
\end{figure}
The other loop $\beta$ can cross $Z$ any number of times, but does so always with the same slope.
The we can represent $B$ as an extension of graded saddle connections with local systems $S_i$, which are contained in $Z$ and connect the two components of its boundary, and arcs $E_i$ which are disjoint from the interior of $Z$.
Moreover, these saddle connections can be chosen so that $\phi(S_i)<\phi(B)$.
Namely, there is an isotopy moving $\beta$ to the concatenation of the $S_i$ and $E_i$, and it should move any point of intersection with $\partial Z$ so that $Z$ lies to the right.
From what we have shown in the annular case, it follows that $\mathrm{Hom}(A,S_i)=0$ for all $i$.
Further, $A$ is orthogonal to any $E_i$, so $\mathrm{Hom}(A,B)=0$ follows.
This completes the proof of Theorem~\ref{thm_stability}.

\subsection{The map $\mathcal M(X)\to \mathrm{Stab}(\mathcal F(X))$}
\label{subsec_modulimap}

We consider a construction for general categories with stability structure.
Let $\mathcal C$ be a triangulated category with stability structure $\sigma$, and $F:\mathcal C\to D(\mathrm{Mod}(\mathbb K))$ an exact functor to the derived category of chain complexes over a field $\mathbb K$.
We combine the two to define
\begin{equation}
C_{F,\sigma}=\bigcup_{\substack{E\in\mathcal C^{\mathrm{ss}},\,k\in\mathbb Z \\ H^k(F(E))\neq 0}}[\log|Z(E)|,+\infty)\times\{\pi(\phi(E)+k)\}\subset\mathbb{R}^2
\end{equation}
which is a union of horizontal rays starting at points of the form $\log(Z(E))$, $E$ semistable.
The support property ensures that the central charges $Z(E)\in\mathbb C$, $E$ semistable, form a discrete subset, hence $C_{F,\sigma}\subset\mathbb{R}^2$ is closed.

\begin{lem}
For $\sigma_1,\sigma_2\in\mathrm{Stab}(\mathcal C)$, $\mathcal C$ and $F$ as above, we get
\begin{equation}
d_H(C_{F,\sigma_1},C_{F,\sigma_2})\leq d(\sigma_1,\sigma_2)
\end{equation}
where $d_H$ is the Hausdorff distance on closed subsets induced by the max-metric on $\mathbb R^2$.
\end{lem}

\begin{proof}
Let $\epsilon=d(\sigma_1,\sigma_2)$, then by definition
\begin{gather}
\left|\phi_{\sigma_1}^-(E)-\phi_{\sigma_2}^-(E)\right|\leq\epsilon,\qquad
\left|\phi_{\sigma_1}^+(E)-\phi_{\sigma_2}^+(E)\right|\leq\epsilon \\
\left|\log(m_{\sigma_1}(E))-\log(m_{\sigma_2}(E))\right|\leq\epsilon
\end{gather}
for all $0\neq E\in\mathcal C$.
For a subset $A$ of a metric space let $B_{\epsilon}(A)$ be the set of points with distance $\leq\epsilon$ to $A$.
To show that $C_{F,\sigma_2}\subseteq B_{\epsilon}(C_{F,\sigma_1})$ it suffices to find, for every $\sigma_2$-semistable object $E$ with $H^k(F(E))\neq 0$, a $\sigma_1$-semistable object $E'$ with $H^k(F(E'))\neq 0$ and
\begin{equation}
|\phi_{\sigma_2}(E)-\phi_{\sigma_1}(E')|\leq\epsilon,\qquad
\log|Z_{\sigma_1}(E')|\leq\log|Z_{\sigma_2}(E)|+\epsilon.
\end{equation}
Let $A_1,\ldots,A_n$ be the $\sigma_1$-semistable components of $E$, $n\geq 1$, then by exactness of $F$ there is an $m$ with $H^k(F(A_m))\neq 0$, and we can take $E'=A_m$. 

By symmetry we also have $C_{F,\sigma_1}\subseteq B_{\epsilon}(C_{F,\sigma_2})$, thus $d_H(C_{F,\sigma_1},C_{F,\sigma_2})\leq\epsilon$.
\end{proof}

\begin{proof}[Proof of Theorem~\ref{ModuliMapThm}]
1. 
Let $B\subset\partial X$ be a connected component of the boundary.
As in the proof of Lemma~\ref{lem_boundaryfunctors} we choose a grading on $B$ to get a functor $I_B:\mathcal F(X)\to D(\mathrm{Mod}(\mathbb K))$.
Recall that if $B$ is gradable in the usual sense, then $I_B(E)$ is perfect for every $E\in\mathcal F(S)$.
On the other hand, if $B$ is only $\mathbb Z/n$-gradable, $n$ maximal, then $I_B(E)$ is $n$-periodic for every $E\in\mathcal F(S)$.

Suppose $X$ has a flat structure with corresponding stability structure $\sigma\in\mathrm{Stab}(\mathcal F(X))$.
If $B$ is a boundary component corresponding to a conical point $s$, then geometrically $C_{I_B,\sigma}$ is just the complement of the domain of the exponential map at $s$, $\mathrm{Dom}(\exp_s)$, in logarithmic coordinates. 
The Voronoi cell, $V_s$, around $s$ is contained in $\mathrm{Dom}(\exp_s)$, and in fact $V_s$ is determined by $\mathrm{Dom}(\exp_s)$ alone, since by definition $x\in V_s$ is and only if the disk centered at $x$ and with $s$ on the boundary is contained in the domain.
This shows that the flat structure is recovered from $\sigma$, proving injectivity.

2. Let $(F_i)$ be a sequence of flat structures on $X$ so that the corresponding sequence of stability structure $(\sigma_i)$ converges to $\sigma\in\mathrm{Stab}(\mathcal F(X))$.
We want to show that $F_i$ converge in $\mathcal M(X)$.

Let $\beta(F_i)$ be the maximal length of a boundary edge of $\mathrm{Core}(F_i)$.
The first step is to show that $\beta(F_i)\leq\beta$ for some $\beta>0$.
Let $E\in\mathcal F(X)$ be the direct sum of objects corresponding to all the boundary edges of $\mathrm{Core}(F_1)$ with arbitrary fixed grading.
An edge should appear twice in the sum if it does not bound the interior of $\mathrm{Core}(F_1)$.
The mass $m_{\sigma_1}(E)$ is then just the circumference of $\mathrm{Core}(F_1)$.
Since $\sigma_i\to\sigma$, we have an upper bound on $m_{\sigma_i}(E)\ge\beta(F_i)$, and the claim follows.

The second step is to show that $\rho(F_i)$ is bounded.
Fix a conical point $s$ of $S$. 
We will show that $d(s,x)$ is bounded for $x\in V_{s,i}\cap\mathrm{Core}(F_i)$, where $V_{s,i}$ is the Voronoi cell around $s$ in $F_i$, and since
\begin{equation}
\rho(F_i)=\max_{s\in (F_i)_{\mathrm{sg}}}d_H(\{s\}, V_{s,i}\cap\mathrm{Core}(F_i))
\end{equation}
this will imply the claim.
Consider the functor $I:\mathcal F(X)\to D(\mathrm{Mod}(\mathbb K))$ associated with $s$ as before.
We have closed subsets $C_i=C_{I,\sigma_i}\subset\mathbb R^2$ converging to $C=C_{I,\sigma}$ in the Hausdorff topology.
Let $\mathscr D$ denote the set of directions (of geodesics) from $s$.
The choice of grading of marked boundary as above gives an identification $\mathscr D\cong\mathbb R/n\pi$ when $s$ has cone angle $n\pi$ and $\mathscr D\cong\mathbb R$ when $s$ has infinite cone angle.
By the \textit{support} of $C_i$, $\mathrm{Supp}(C_i)$, we mean the closure of its projection to $\mathcal D$.
It follows from compactness of $\mathrm{Core}(F_i)$ that each $C_i$ is compactly supported.
By convergence, there is compact $J\subset\mathscr D$ containing all $\mathrm{Supp}(C_i)$ and $\mathrm{Supp}(C)$.
Let $r_i(\phi)\in\mathbb R_{\ge 0}\cup\{+\infty\}$ be the length of the maximal geodesic starting at $s$ in direction $\phi\in\mathscr D$ and which is entirely contained in $V_{s,i}$.
Note that the functions $r_i$ are determined by $C_i$, thus $\sigma_i$ alone, are continuous, and converge pointwise to $r:\mathscr D\to \mathbb R_{\ge 0}\cup\{+\infty\}$ determined by $\sigma$.
Define $q_i(\phi)\in\mathbb R_{\ge 0}$ similarly as the length of the maximal geodesic starting at $s$ in direction $\phi\in\mathscr D$ and which is entirely contained in $V_{s,i}\cap\mathrm{Core}(F_i)$.
We need to shows that $q_i$ are uniformly bounded. By definition, $q_i(\phi)\le r_i(\phi)$, and
\begin{equation} \label{qsupp}
(\phi-\pi,\phi+\pi)\cap J=\emptyset\quad\implies\quad q_i(\phi)=0.
\end{equation}
Suppose, for contradiction, that there is a sequence $\phi_i\in\mathscr D$ with $q_i(\phi_i)\to\infty$.
By \eqref{qsupp} the $\phi_i$ are contained in a compact subset of $\mathscr D$ and we may assume $\phi_i\to\phi$ after passing to a subsequence.
Now since $q_i(\phi_i)\to\infty$ we must have $r(\phi)=\infty$, which means that the interval $(\phi-\pi,\phi+\pi)$ is disjoint from $\mathrm{Supp}(C)$.
Thus also
\begin{equation}
\left[\phi-\frac{\pi}{2},\phi+\frac{\pi}{2}\right]\cap \mathrm{Supp}(C_i)=\emptyset
\end{equation}
for sufficiently large $i$.
The sector bounded by directions $\phi-\pi/2$ and $\phi+\pi/2$ in $\mathrm{Dom}(\exp_s)$ cannot map entirely to $\mathrm{Core}(F_i)$, as this contradicts $\rho(F_i)<\infty$.
Hence the part of the sector which maps to $\mathrm{Core}(F_i)$ is a triangle $\Delta$ or just the point $s$.
The length of the side of the triangle opposite the vertex $s$ is bounded above by $\beta$.
Thus also $q_i|(\phi-\pi/2,\phi+\pi/2)$ has an upper bound in terms of $\beta$, which is independent of $i$, a contradiction.

We have establishes that $\rho(F_i)$ and $\beta(F_i)$ have upper bounds independent of $i$.
By Lemma~\ref{lem_DelaunayBound} the length the Delaunay edges on all the $F_i$ is bounded above.
Thus, there are only a finite number of polygonal subdivisions up to isotopy arising as Delaunay partitions in the sequence $F_i$.
Passing to a subsequence, we may hence assume that the Delaunay subdivisions for the $F_i$ are all isotopic, i.e. combinatorially equivalent.
Since $\sigma_i\to\sigma$, lengths of edges converge to positive numbers as $i\to\infty$.
It may still happen that some polygons become degenerate in the limit, i.e. their area goes to zero.
If we glue back together the set of non-degenerate limiting polygons, we get a flat structure $F$ which is the limit of the $F_i$ and corresponds to $\sigma$.
\end{proof}


\section{Cluster-like structures}
\label{sec_cluster}

\subsection{S-graphs}
\label{subsec_sgraphs}

In this section we specialize to the case of flat surfaces with infinite area.
It turns out that for generic choice of horizontal direction the heart of the t-structure of the corresponding stability structure is Artinian with a finite number of simple objects and described by a certain type of finite graph.
Tilting of the t-structure leads to mutation of the graph.

Suppose $X$ is a flat surface of finite type with infinite area.
Recall from Subsection~\ref{subsec_hsd} that, for generic choice of horizontal direction, the leaves converging to the conical singularities cut the surface into horizontal strips of finite or infinite height.
The combinatorial structure of this horizontal strip decomposition can be encoded in terms of a finite graph, $G$, with set of vertices $X_{\mathrm{sg}}$ and set of edges the horizontal strips of finite height.
An edge is attached to the singularities which lie on the boundary of the corresponding horizontal strip.
We also need to record two additional pieces of data:
\begin{enumerate}
\item For each $v\in X_{\mathrm{sg}}$ with finite (resp. infinite) cone angle a cyclic (resp. total) order on the set $H_v$ of half-edges meeting that vertex given by the counter-clockwise order of horizontal strips around $s$.
\item For each pair of successive half-edges $a<b\in H_v$ a positive integer, $d(a,b)$, so that $d(a,b)-1$ is the number of horizontal strips of infinite height between $a$ and $b$ as we go around $v$.
\end{enumerate}
We call a graph as above, i.e. a finite graph with orders on $H_v$ and integers $d(a,b)$, an \textbf{S-graph}. 
In order to record the entire flat structure we attach a number $Z(e)$ in the upper half-plane, $\mathbb H$, to each edge $e$, given by the vector between the singularities on the boundary of the corresponding horizontal strip.
We summarize the construction in the following proposition.

\begin{prop} \label{prop_Sgraph}
There is a bijective correspondence between (isomorphism classes of)
\begin{enumerate}
\item Flat surfaces of finite type with infinite area with the property that there are no horizontal leaves of finite length, and
\item S-graphs with choice of number $Z(e)\in\mathbb H$ for each edge $e$.
\end{enumerate}
\end{prop}

The data of an S-graph $G$ may be encoded in terms of a graded quiver with relations $Q$.
Vertices of $Q$ are edges of $G$.
Arrows of $Q$ are given by pairs $(h_1,h_2)$ of successive half edges $h_1<h_2\in H_v$ in degree $d(h_1,h_2)$.
Relations are quadratic: $(h_1,h_2)\cdot(h_3,h_4)=0$ when $h_2\neq h_3$ belong to the same edge.
Note that in the graded linear path category $\mathcal A=\mathbb{K}Q$ we have 
\begin{equation}
\mathrm{Hom}^{<0}(A,B)=0,\qquad \mathrm{Hom}^0(A,B)=\begin{cases}\mathbb K & \text{ if }A=B \\ 0 & \text{ else }\end{cases}
\end{equation}
These properties ensure that the extension closure of the objects of $\mathcal A$ in $\mathrm{Tw}\mathcal A=\mathrm{Perf}\mathcal A$ is the heart of an Artinian bounded t-structure, and objects of $\mathcal A$ are precisely the simple ones. 

\begin{prop}
Let $X$ be a flat surface with infinite area such that no horizontal leaves have finite length, and $\mathcal A$ be the heart of the t-structure for the stability structure on $\mathcal F(X)$ determined by the flat metric, then $\mathcal A$ is Artinian with a finite number of simple objects.
In terms of the S-graph, $\mathcal A$ is described by the graded quiver with relations considered above.
\end{prop}

\begin{proof}
Choose a family $X_t$, $t\in [0,1]$ of flat structures so that $X_0=X$, none of the $X_t$ have finite horizontal leaves, and the simple saddle connections of $X_1$ are vertical.
All the $X_t$ have horizontal strip decompositions described by the same S-graph, only the parameters $Z(e)$ change, and are in $i\mathbb R_{>0}$ for $X_1$.
The corresponding stability structures $\sigma_t$ all have the same heart of the t-structure, as none of them have semistable objects with real central charge (by the condition on the horizontal leaves), i.e. there is no wall-crossing of the second kind.

The flat surface $X_1$ has a purely one-dimensional core, which is just the union $K$ of the simple saddle connections.
This is clear, since $K$ contains all singularities, and exterior angles are $\ge\pi$.
We conclude that $X_1$ has no closed geodesics and all saddle connections are simple.
Thus, stable objects of $\sigma_1$ in the heart of the t-structure are just the simple geodesics.
Also, the central charge of $\sigma_1$ takes values in $i\mathbb R$, so simple objects in the heart of the t-structure are the same as stable ones.

The quiver determined by the S-graph is a special case of the graded quiver with relations describing $\mathrm{End}(G)$ of a formal generator of $\mathcal F(X)$ (Subsection~\ref{subsec_quivers}).
\end{proof}

\subsubsection*{Mutation}

We have assumed above that there are no horizontal saddle connections.
Points in the moduli space where this does happen, i.e. some stable object has real central charge, are by definition on walls of the second kind.
Let us consider a generic path $\gamma:(-\epsilon,\epsilon)\to\mathcal M(X)$ which crosses one of these walls at time $t=0$.
There are two possibilities, either the foliation of $\gamma(0)$ has a single horizontal saddle connection, $e$, or some of the leaves are closed loops foliating a cylinder.
The latter is a more complicated kind of wall-crossing, and we will restrict attention to the former case.

Making $\epsilon$ smaller, we may assume that the horizontal foliations of $\gamma(t)$, $t\neq 0$, have no finite-length leaves.
Let $G_-,G_+$ be the S-graphs for $\gamma(t)$ for $t<0$, $t>0$ respectively.
The saddle connection $e$ becomes simple for small non-zero values of $t$, and so can be identified with an edge of both $G_-$ and $G_+$.
The graphs $G_-,G_+$ are related by mutation, which we want to describe explicitly.
There are two cases: \textit{Left mutation} if the phase of $e$ is moving clockwise, \textit{right mutation} if it is moving counterclockwise.
Let us assume the former.
Looking at the horizontal strip decompositions one sees that $S_+$ is obtained from $S_-$ as follows.
For a half-edge $a$ denote by $S(a)$ and $S^{-1}(a)$ its successor and predecessor respectively.
\begin{enumerate}
\item Let $a,b$ be the half-edges of $e$. 
Decrement $d(a,S(a))$ and $d(b,S(b))$ by one.
\item If $d(a,S(a))=0$, then slide the end of $S(a)$ along $e$ so that it becomes the new predecessor, $c$, of $b$ with $d(c,b)=1$. Otherwise, increment $d(S^{-1}(b),b)$ by one.
\item Do the previous step with $a$ and $b$ switched.
\end{enumerate}
The process is illustrated in Figure~\ref{fig_mutation}.
\begin{figure}[h]
\centering
\includegraphics[scale=1]{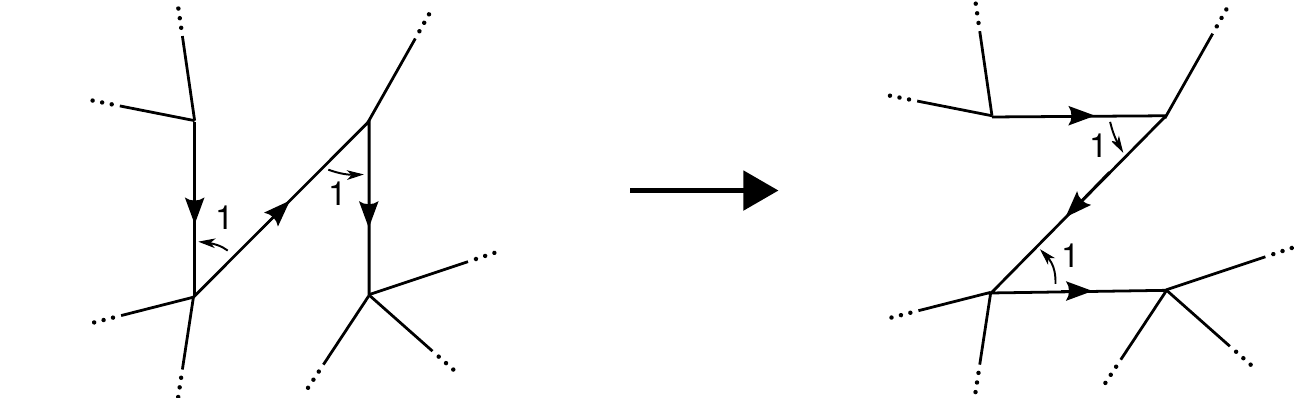}
\caption{Left mutation of the S-graph.}
\label{fig_mutation}
\end{figure}
Right mutation is then simply the inverse operation.

\subsection{Example: Dynkin type $A_n$}

Let $S$ be a graded surface with interior diffeomorphic to a disk and with $n+1\ge 2$ marked boundary components.
Such a graded surface is unique up to isomorphism.
The Fukaya category $\mathcal F(S)$ is equivalent to the bounded derived category of any $A_n$-type quiver.

\begin{thm} \label{thm_stabdynkin}
For $S$ as above of $A_n$-type,
\begin{equation}
\mathcal M(S)=\mathrm{Stab}(\mathcal F(S))\cong\mathbb C^n.
\end{equation}
\end{thm}

The case $n=2$ of this theorem is contained in \cite{bqs}.

\begin{proof}
The proof consists of two parts, first to show that $\mathcal M(S)\cong\mathbb C^n$ as complex manifolds, and then to show that the map $\mathcal M(S)\to\mathrm{Stab}(\mathcal F(S))$ is surjective.

1. We may assume that the interior of $S$ is identified with $\mathbb C$, and the grading given by the standard horizontal foliation on $\mathbb C$.
Define an analytic map $u:\mathbb C^n\to\mathcal M(S)$ by
\begin{equation}
(a_0,\ldots,a_{n-1})\mapsto (\exp\left(z^{n+1}+a_{n-1}z^{n-1}+\ldots+a_0\right)dz^2,a_0).
\end{equation}
This requires some explanation.
By the results of Subsection~\ref{subsec_complexanalytic}, the quadratic differentials above define a flat structure on $\mathbb C$, the completion of which is, as a topological space, independent of the  parameters $a_i$.
We assume that $S$ is the real blow-up of this completion, so that we get a flat structure on $S$.
To determine a point in $\mathcal M(S)$ we also need a homotopy class of paths from the horizontal foliation of the quadratic differential to the standard horizontal foliation on $z$, which is what $a_0$ is used for.
Since $S$ is contractible, such a homotopy class of paths is determined by its restriction to a single point, which we take to be the origin, where $a_0$ defines a path $\exp((1-t)a_0)dz^2$ in $(T^*_0\mathbb C)^{\otimes 2}$, $t\in [0,1]$.

An analytic automorphism of $\mathbb C$ acting on the set of polynomials of the form $z^{n+1}+a_{n-1}z^{n-1}+\ldots$ must be of the form $z\mapsto\zeta z$, $\zeta^{n+1}=1$.
All these extend to automorphisms of $S$, but only for $\zeta=1$ is it isotopic to the identity.
This implies injectivity of $u$.
For surjectivity, let $S'$ be a contractible flat surface with $n+1$ conical points, all with infinite cone angle.
By the results of Subsection~\ref{subsec_complexanalytic}, $S'$ comes from a quadratic differential $\varphi$ with exponential singularities on a compact Riemann surface $C$.
Since $S'$ is contractible, $C$ must be $\mathbb CP^1$ and $\varphi$ must have a single exponential singularity.
After applying a suitable M\"obius transformation, the exponential singularity is at $\infty$ and $\varphi$ is of the form $\exp\left(z^{n+1}+a_{n-1}z^{n-1}+\ldots+a_0\right)dz^2$.

2. Let $\sigma\in\mathrm{Stab}(\mathcal F(S))$ be a stability structure, $\mathcal A$ the heart of its t-structure.
As $\mathcal F(S)$ has only a finite number of indecomposable objects up to shift, $\mathcal A$ must be Artinian with $n=\mathrm{rk}K_0(\mathcal F(S))$ simple objects $E_1,\ldots,E_n$, up to isomorphism.
Each $E_i$ corresponds to an isotopy class of graded arcs in $S$, which we denote by the same symbol.
Since $\mathrm{Ext}^{\leq 1}(E_i,E_j)=0$ for $i\neq j$, we can choose representatives of the $E_i$ which intersect only at the endpoints. 
Because the $E_i$ give a basis of $H_1(S,\partial S;\mathbb Z_\tau)$, they necessarily form a full formal system of graded arcs.
Its structure is given by an S-graph which is a tree with $n$ edges.
Rotating $\sigma$ slightly by an angle $\epsilon$, we may assume that no central charge lies in $\mathbb R$.
Using Proposition~\ref{prop_Sgraph} we get a flat surface $X$ corresponding to $\sigma$, after rotating the horizontal direction back by $-\epsilon$.
We already showed in the first part that $X$ belongs to $\mathcal M(S)$.
\end{proof}

\begin{prop}
Let $S$ be the graded marked surface of $A_n$-type, $n>1$, then the canonical map $\mathrm{MCG}(S)\to\mathrm{Aut}(\mathcal F(S))$ from the mapping class group of graded automorphisms to the group of autoequivalences up to natural equivalence, is an isomorphism.
The group is abelian, isomorphic to $\mathbb Z^2/\mathbb Z(2,n+1)$.
\end{prop}

In the case $n=1$ this map is surjective with kernel $\mathbb Z/2$.

\begin{proof}
The assumption $n>1$ guarantees that distinct graded boundary arcs of $S$ correspond to distinct isomorphism classes of objects in $\mathcal F(S)$.
Since any element of $\mathrm{MCG}(S)$, which are all represented by rotations, is determined by its action on graded boundary arcs, this implies injectivity.

Note that (isomorphism classes of) objects $E\in\mathcal F(S)$ corresponding to graded boundary arcs are intrinsically characterized, e.g. by being indecomposables with $\mathcal F(S)/E$ equivalent to the bounded derived category of an $A_{n-1}$ quiver.
Thus an automorphism $\Phi$ of $\mathcal F(S)$ induces a map on the set $\mathcal B$ of graded boundary arcs.
$\mathcal B$ has a natural total order, isomorphic to the total order on $\mathbb Z$, so that $\dim\mathrm{Ext}(E_i,E_j)=1$ if $j=i+1$, and $0$ otherwise.
This implies that the induced map on $\mathcal B$ is a shift, hence comes from the action of an element of $\mathrm{MCG}(S)$.
Finally, any autoequivalence of $\mathcal F(S)$ which fixes all objects is seen to be naturally isomorphic to the identity functor.  
\end{proof}

\begin{rem}
Fix a grading on each boundary arc of $S$ and let $B_1,\ldots,B_{n+1}$ be the Yoneda-duals of the corresponding objects in $\mathcal F(S)$.
Their sum is a functor $B:\mathcal F(S)\to\mathcal C\cong\mathrm{Perf}(\mathbb K^{n+1})$.
In order to make the statement of the above proposition work for $n=1$ one needs to replace $\mathrm{Aut}(\mathcal F(S))$ by the group of pairs of autoequivalences, one of $\mathcal F(S)$, one of $\mathcal C$, which are compatible with $B$, up to natural equivalence.
Something like this also seems to be needed in order to generalize the proposition to other marked surfaces.
In the context of Fukaya--Seidel categories, $B$ is the restriction to the distinguished fiber.
\end{rem}

\begin{cor}
Let $S$ be the graded marked surface of $A_n$-type, $n>1$, then
\begin{align}
\mathrm{Stab}(\mathcal F(S))/\mathrm{Aut}(\mathcal F(S))&=\mathcal M(S)/\mathrm{MCG}(S) \\
&=\left\{e^{P(z)}dz^2\mid \deg P=n+1 \right\}/\mathrm{Aut}(\mathbb C) \\
 &=\left\{ e^{z^{n+1}+a_{n-1}z^{n-1}+\ldots+a_0}dz^2\right\}/\mathbb Z/(n+1)\mathbb Z
\end{align}
as complex orbifolds.
\end{cor}

In order to gain a better understanding of the wall and chamber structure on $\mathcal M(S)=\mathrm{Stab}(\mathcal F(S))$ one needs to study possible types of cores and S-graphs, giving chambers between walls of the first and second kind respectively.
We already know that S-graphs which occur for $S$ are those with underlying graph a tree with $n$ edges and total order on each set of half-edges.
Possible types of cores, even generic ones, are more difficult to enumerate except for small values of $n$.
Figure~\ref{fig_n3_convhull} gives some examples for $n=3$.
\begin{figure}[h]
\centering
\includegraphics[scale=1]{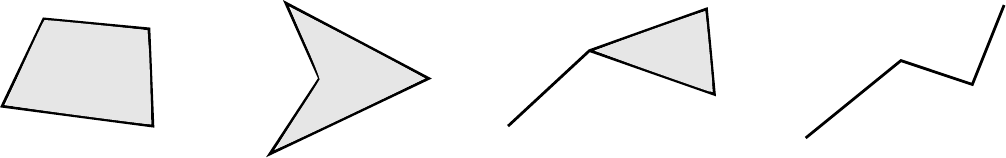}
\caption{Possible topologies of $K$ when $n=3$ for a generic stability structure. Note that the number of stable objects up to shift is 6, 5, 4, 3 respectively.}
\label{fig_n3_convhull}
\end{figure}

\subsection{Example: Affine type $\widetilde{A_n}$}

Fix positive integers $p,q$ and let $n+1=p+q$. 
Let $S$ be a graded surface with interior an annulus and $p$ (resp. $q$) marked boundary components, each diffeomorphic to an open interval, on the left (resp. right) end, and grading so that a simple closed loop around the annulus has vanishing Maslov index, i.e. is gradable.
Such $S$ is uniquely determined up to isomorphism.
For concreteness we take $S$ to be the real blow-up of the completion of $\mathbb C^*$ with respect to the metric $|\varphi|$, where $\varphi=\exp(z^p+z^{-q})dz^2$, and grading given by the horizontal foliation of $\varphi$.
The Fukaya category $\mathcal F(S)$ is equivalent to the bounded derived category of any quiver with underlying graph the extended Dynkin diagram $\widetilde{A_n}$ and $p$ (resp. $q$) arrows directed counterclockwise (resp. clockwise).

\begin{thm} \label{thm_stabaffine}
For $S$ as above of $\widetilde{A_n}$-type,
\begin{equation}
\mathcal M(S)=\mathrm{Stab}(\mathcal F(S))\cong\mathbb C^{n+1}.
\end{equation}
\end{thm}

When $n=1$ the category $\mathcal F(S)$ is equivalent to the bounded derived category of coherent sheaves on $\mathbb P^1$ and it was proven by Okada~\cite{okada} that $\mathrm{Stab}(D^b(\mathbb P^1))\cong\mathbb C^2$.
For $n=2$ it was shown by different methods in \cite{dk14} that $\mathrm{Stab}(\mathcal F(S))$ is contractible.

\begin{proof}
As in the Dynkin case the proof consists of two parts, first to show that $\mathcal M(S)=\mathbb C^{n+1}$ as complex manifolds, and then to show that the map $\mathcal M(S)\to\mathrm{Stab}(\mathcal F(S))$ is surjective.

1. To a point $(a_{p-1},\ldots,a_{-q})\in\mathbb C^{n+1}$ we assign the flat surface given by the quadratic differential
\begin{equation}
\exp(z^p+a_{p-1}z^{p-1}+\ldots+a_{-q+1}z^{-q+1}+\exp(a_{-q})z^{-q})dz^2
\end{equation}
on $\mathbb C^*$.
Since $\mathbb C^{n+1}$ is contractible and the dependence on the parameters continuous, we can resolve the $\mathrm{MCG}(S)$-ambiguity and get a map $u:\mathbb C^{n+1}\to\mathcal M(S)$.
Arguments as in the proof of Theorem~\ref{thm_stabdynkin} show that $u$ is an isomorphism.

2. Let $\sigma\in\mathrm{Stab}(\mathcal F(S))$.
Unlike the Dynkin case, the heart of the t-structure of $\sigma$ may not be Artinian.
We will show however that this happens non-generically and the t-structure becomes Artinian after possibly rotating $\sigma$ slightly.
Let $\Gamma=K_0(\mathcal F(S))\cong\mathbb Z^{n+1}$, and let $N\subset\Gamma$ be the subgroup generated by a simple curve around the annulus.
Equivalently, $N$ is the kernel of the Euler form on $\Gamma$.
Let $E\subset\Gamma$ be the set of classes of indecomposable objects.
It follows from the classification of indecomposable objects that $E/N\subset\Gamma/N$ is finite.
Hence, the image of $E$ under $Z:\Gamma\to\mathbb C$ has finitely many orbits under translation by the vectors $Z(N)$, and thus phases of stable objects can only accumulate towards $\mathrm{Arg}(Z(N))$, or not at all if $Z(N)=\{0\}$.
In particular, after possibly rotating $\sigma$, $0$ is not in the closure of the set of phases of stable objects, and thus the heart of the t-structure Artinian.
From here we can proceed as in the proof of Theorem~\ref{thm_stabdynkin}.
\end{proof}

\section{Open problems and further directions}

\subsection*{Moduli of objects}

For a partially wrapped Fukaya category $\mathcal F$ of a surface $S$ our classification describes $\mathrm{Ob}(\mathcal F)$ as a set in terms of curves on $S$.
However, $\mathrm{Ob}(\mathcal F)$ has a much richer structure of a derived stack.
\begin{problem}
Describe the geometry of $\mathrm{Ob}(\mathcal F)$.
\end{problem}
Also, the case of Fukaya categories of closed surfaces remains.
This would be an A-side generalization of Atiyah's classification for the elliptic curve.
However, for $g>1$ the category is only $\mathbb Z/2$-graded and there are no t-structures or stability structures.
\begin{problem}
Classify objects of $\mathcal F(S)$ when $S$ is a closed surface with $g(S)>1$, assuming this is a tame problem.
\end{problem}

\subsection*{Higher dimensions}

There is at present no complete conjectural picture extending the results of this paper to higher dimensional symplectic manifolds.
Investigations into this direction were started by Thomas~\cite{thomas01} and Thomas--Yau~\cite{thomas_yau02} with the aim of finding an algebraic notion of stability describing special Lagrangian submanifolds and their mean curvature flow.
More recently, their proposal has been updated by Joyce~\cite{joyce_conj}, taking into account recent developments on Fukaya categories, mean curvature flow, and Bridgeland's axiomatics.
The conjecture is that given a compact Calabi--Yau $M$ one has a stability structure on its Fukaya category, $\mathcal F(M)$, such that singular special Lagrangians in $M$, with additional data to make them objects in $\mathcal F(M)$, are (all?) stable objects.
The central charge is given by integration of the holomorphic volume form.
Mean curvature flow applied to an object in $\mathcal F(S)$, at least with sufficiently small phase variation, should converge to the components of the Harder--Narasimhan filtration.
It is however still an open problem to describe (a component of) $\mathrm{Stab}(\mathcal F(M))$ in general, even conjecturally.
\begin{problem}
Find a geometric description of $\mathrm{Stab}(\mathcal F(M))$ for higher dimensional $M$.
\end{problem}

\bibliographystyle{plain}
\bibliography{stab1}

\Addresses

\end{document}